\documentclass[a4paper,reqno]{amsart}

\usepackage{amsmath}
\usepackage{amssymb}
\usepackage{amscd}
\usepackage{amsthm}
\usepackage{type1cm}
\usepackage{tcolorbox}
\usepackage{mathrsfs}
\usepackage{adjustbox}

\usepackage{tikz}
\usepackage{tikz-3dplot} 
\usetikzlibrary{math} 
\usetikzlibrary{arrows.meta} 

\usepackage[colorlinks,citecolor=darkgreen,linkcolor=red]{hyperref}
\definecolor{darkgreen}{rgb}{0.0, 0.6, 0.0}
\definecolor{quiverblue}{RGB}{0,92,145}
\definecolor{quiverred}{RGB}{185,55,35}
\tikzset{
  dgq vertex/.style={circle,fill=white,text=black,inner sep=2pt,
    minimum size=4.5mm},
  dgq zero/.style={-{Stealth[length=4pt,width=3pt]},
    color=black,line width=.55pt},
  dgq minus one/.style={-{Stealth[length=4pt,width=3pt]},
    color=quiverred,line width=.75pt,densely dotted},
  dgq minus two/.style={-{Stealth[length=4pt,width=3pt]},
    color=quiverblue,line width=.65pt,densely dashed},
  dgq label/.style={fill=white,inner sep=1.3pt}
}
\newcommand{\quiverdegreelegend}{%
\begin{tikzpicture}[baseline=-.5ex,every node/.style={font=\footnotesize}]
\draw[dgq zero] (0,0)--(.65,0);
\node[anchor=west] at (.8,0) {degree $0$};
\draw[dgq minus one] (3.2,0)--(3.85,0);
\node[anchor=west,text=quiverred] at (4,0) {degree $-1$};
\draw[dgq minus two] (6.6,0)--(7.25,0);
\node[anchor=west,text=quiverblue] at (7.4,0) {degree $-2$};
\end{tikzpicture}%
}

\tcbuselibrary{breakable, skins, theorems}

\usepackage{tikz-cd}

\def\A{\mathcal{A}}
\def\B{\mathcal{B}}
\def\C{\mathcal{C}}
\def\D{\mathcal{D}}
\def\E{\mathcal{E}}
\def\F{\mathcal{F}}

\def\K{\mathcal{K}}

\def\M{\mathcal{M}}

\def\P{\mathcal{P}}
\def\S{\mathcal{S}}
\def\T{\mathcal{T}}
\def\U{\mathcal{U}}

\def\X{\mathcal{X}}
\def\Y{\mathcal{Y}}
\def\Z{\mathcal{Z}}

\DeclareMathOperator{\Mod}{\mathsf{Mod}}
\DeclareMathOperator{\md}{\mathsf{mod}}
\renewcommand{\mod}{\md}
\DeclareMathOperator{\proj}{\mathsf{proj}}

\DeclareMathOperator{\add}{\mathsf{add}}

\DeclareMathOperator{\per}{\mathsf{per}}
\DeclareMathOperator{\thick}{\mathsf{thick}}
\DeclareMathOperator{\pvd}{\mathsf{pvd}}

\DeclareMathOperator{\sg}{\mathsf{sg}}

\DeclareMathOperator{\Hom}{Hom}
\DeclareMathOperator{\End}{End}

\DeclareMathOperator{\Ext}{Ext}
\DeclareMathOperator{\op}{op}

\DeclareMathOperator{\Ker}{Ker}
\DeclareMathOperator{\Cok}{Cok}

\DeclareMathOperator{\rad}{rad}
\DeclareMathOperator{\HH}{HH}
\DeclareMathOperator{\HC}{HC}
\DeclareMathOperator{\RHom}{\mathbb{R}Hom}
\DeclareMathOperator{\REnd}{\mathbb{R}End}

\DeclareMathOperator{\dg}{dg}
\DeclareMathOperator{\fp}{fp}

\DeclareMathOperator{\ch}{char}

\def\gl{\mathop{\rm gl.dim}\nolimits}
\def\lgl{\mathop{\rm l.gl.dim}\nolimits}
\def\rgl{\mathop{\rm r.gl.dim}\nolimits}
\def\wgl{\mathop{\rm w.gl.dim}\nolimits}
\def\pd{\mathop{\rm proj.dim}\nolimits}

\def\fd{\mathop{\rm flat.dim}\nolimits}

\theoremstyle{plain}
\newtheorem{IntroThm}{Theorem}

\newtheorem{Thm}{Theorem}[section]
\newtheorem{Lem}[Thm]{Lemma}
\newtheorem{Prop}[Thm]{Proposition}
\newtheorem{Cor}[Thm]{Corollary}
\theoremstyle{definition}
\newtheorem{Def}[Thm]{Definition}
\newtheorem{Ex}[Thm]{Example}
\newtheorem{Rem}[Thm]{Remark}

\newcommand{\FRAC}[2]{\leavevmode\kern.1em\raise.5ex\hbox{\the\scriptfont0 #1}\kern-.1em/\kern-.15em\lower.25ex\hbox{\the\scriptfont0 #2}}

\title{Auslander correspondence for higher stable dg categories and cluster Morita theory}
\author{Ryu Tomonaga}
\address{Graduate School of Mathematical Sciences, The University of Tokyo, 3-8-1 Komaba, Meguro-ku, Tokyo, 153-8914, Japan}
\email{ryu-tomonaga@g.ecc.u-tokyo.ac.jp}
\makeatletter
\@ifundefined{subjclassname@2020}{%
  \@namedef{subjclassname@2020}{\textup{2020} Mathematics Subject Classification}%
}{}
\makeatother
\subjclass[2020]{Primary 18G80, Secondary 16E45, 16G70, 18G35}
\keywords{Auslander correspondence, differential graded categories,
  cluster tilting, Calabi--Yau structures, Koszul duality}

\begin{document}

\begin{abstract}
We prove a cluster Morita theorem: every pretriangulated dg category
with a $d$-cluster tilting subcategory is realized as the $d$-cluster
dg category of that subcategory regarded as a connective dg category.
Reversing the Amiot--Guo construction, this yields a bijection between
$d$-stable dg categories and pretriangulated dg categories with a
chosen $d$-cluster tilting subcategory. Our Auslander correspondence
characterizes $d$-stability by coherence, weak global dimension at most
$d+1$, and a $(d+1)$-shifted duality on finitely presented modules.

For a locally finite $d$-stable dg algebra over a field, right
$(d+1)$-Calabi--Yau structures on its perfectly valued derived dg
category correspond bijectively to right $d$-Calabi--Yau structures on
its $d$-cluster dg category. In characteristic zero, this removes the
pseudo-compact enhancement hypothesis from Keller--Liu's theorem on
Amiot's conjecture.

Via Koszul duality, we characterize $d$-stability over any field for
locally finite minimal completed dg path algebras (and their analogues
for graded species) by explicit conditions on the graded arrows and the
quadratic differential. For $d=1$, these are deformed derived
mesh algebras of Dynkin type; for $d=2$, they are complete twisted
Ginzburg dg algebras, yielding via cluster Morita theory a twisted analogue of
Amiot's conjecture in arbitrary characteristic without a Calabi--Yau
hypothesis.
\end{abstract}

\maketitle
\tableofcontents

\section*{Introduction}

Cluster categories were introduced to model cluster mutation in a
triangulated setting. In the classical examples of simply laced Dynkin
type, clusters in the corresponding cluster algebra correspond
bijectively to isomorphism classes of basic $2$-cluster tilting objects,
and the variables in a cluster correspond to the indecomposable direct
summands of the associated object \cite{FZ02,BMRRT,Kel05}.

For a quiver with potential $(Q,W)$ over a field, let
$\Gamma=\Gamma(Q,W)$ be its Ginzburg dg algebra \cite{Gin06}. When the
Jacobian algebra $H^0\Gamma$ is finite dimensional, Amiot defined the
cluster category of $(Q,W)$ as the Verdier quotient
\[
\per\Gamma/\pvd\Gamma,
\]
where $\per\Gamma$ is the perfect derived category and $\pvd\Gamma$
consists of dg modules with finite-dimensional total cohomology. This
quotient is $2$-Calabi--Yau, and the image of the free module $\Gamma$
is a $2$-cluster tilting object with endomorphism algebra $H^0\Gamma$
\cite{Ami09}. Guo extended this construction to arbitrary $d\geq2$: if
$A$ is a $(d+1)$-Calabi--Yau connective dg algebra with
finite-dimensional $H^0A$, then $\per A/\pvd A$ is $d$-Calabi--Yau and
the image of $A$ is a $d$-cluster tilting object \cite{Guo}.

Can the construction of Amiot and Guo be reversed? At the dg level, the
quotient functor $q_A$ and its action on the free right $A$-module $A_A$
are given by
\[
q_A\colon\per_{\dg}A\longrightarrow
\C_{d,\dg}(A):=\per_{\dg}A/^{\mathbb L}\pvd_{\dg}A,
\qquad A_A\longmapsto\overline A=q_A(A_A).
\]
Applying $q_A$ to endomorphisms of $A_A$ and taking connective
truncations yields a quasi-isomorphism
\[
\tau^{\leq0}A\xrightarrow{\;\sim\;}
\tau^{\leq0}\End_{\C_{d,\dg}(A)}(\overline A).
\]
This suggests starting from a pretriangulated dg category $\C$ with a
chosen $d$-cluster tilting object $M$, forming its connective
endomorphism dg algebra $A_M:=\tau^{\leq0}\End_\C(M)$, and asking
whether $\C_{d,\dg}(A_M)\simeq\C$ via a quasi-equivalence sending
$\overline{A_M}$ to $M$. More generally, given a $d$-cluster tilting
subcategory of $H^0\C$, the corresponding input is the full connective
dg subcategory $\M\subseteq\tau^{\leq0}\C$ on its objects.

For every commutative ring $k$ and every $d\geq1$, the cluster Morita
theorem (Theorem~\ref{thm:intro-cluster-morita}) gives this reconstruction
and a bijection between
$d$-stable dg categories and pairs $(\C,\M)$ as below, with $H^0\C$
idempotent complete:
\nopagebreak
\[
\vcenter{\hbox{%
\begin{tikzpicture}[
  panel/.style={draw=black!70,fill=black!3,rounded corners=2pt,
    align=center,line width=.5pt,inner sep=0pt,minimum height=32pt},
  entry/.style={inner sep=0pt},
  mapsto arrow/.style={|->,>=Stealth,line width=.5pt,
    shorten <=8pt,shorten >=8pt}
]
\node[panel,minimum width=120pt,font=\bfseries] (stable)
  {$d$-stable dg categories};
\node[panel,anchor=west,minimum width=156pt] (pairs)
  at ([xshift=164pt]stable.east)
  {\textbf{Pairs} $(\C,\M)$, $\C$ pretriangulated\\
   $\M\subseteq\tau^{\leq0}\C$ is $d$-cluster tilting};

\draw[<->,>=Stealth,line width=.5pt,shorten <=4pt,shorten >=4pt]
  (stable.east) -- node[above=2pt,inner sep=0pt] {$1{:}1$} (pairs.west);

\node[entry] (input) at ([yshift=-32pt]stable.center) {$\A$};
\node[entry] (realization) at ([yshift=-32pt]pairs.center)
  {$\bigl(\C_{d,\dg}(\A),\overline{\A}\bigr)$};
\node[entry] (subcategory) at ([yshift=-53pt]stable.center) {$\M$};
\node[entry] (pair) at ([yshift=-53pt]pairs.center) {$(\C,\M)$};

\draw[mapsto arrow] (input.east) --
  node[above=2pt,inner sep=0pt] {take the $d$-cluster dg category}
  (realization.west);
\draw[mapsto arrow] (pair.west) --
  node[above=2pt,inner sep=0pt] {retain the chosen subcategory}
  (subcategory.east);
\end{tikzpicture}%
}}
\]
Here $\overline{\A}$ is the canonical $d$-cluster tilting subcategory
of $\tau^{\leq0}\C_{d,\dg}(\A)$, and the two assignments are mutually
inverse on quasi-equivalence classes, where equivalences of pairs
preserve the distinguished subcategory.

This correspondence brings together several themes developed in the
paper. First, our Auslander correspondence
(Theorem~\ref{thm:intro-auslander-correspondence}) characterizes the
left-hand side intrinsically by coherence, a bound on weak global
dimension, and a shifted duality on finitely presented modules; this
duality also produces a fundamental domain for $\C_{d,\dg}(\M)$
(Theorem~\ref{thm:intro-cluster-morita}). Second, over a field, we prove
a bijection between right $(d+1)$-Calabi--Yau structures on the bounded
derived dg category of finitely presented modules and right
$d$-Calabi--Yau structures on the $d$-cluster dg category
(Theorem~\ref{thm:intro-cy}), which in characteristic zero removes the
pseudo-compact enhancement hypothesis from Keller and Liu's realization
theorem on Amiot's conjecture
(Theorem~\ref{thm:intro-cy-realization}). Third, for locally finite
connective dg algebras over any field, Koszul duality translates
$d$-stability into a graded Frobenius condition on the ordinary Yoneda
algebra (Theorem~\ref{thm:intro-yoneda-frobenius}) and into explicit
conditions on the generators and quadratic differential of a minimal
completed tensor model (Theorem~\ref{thm:intro-mesh-models}). Together
with cluster Morita theory, these models yield explicit realizations by
deformed derived mesh algebras for $d=1$
(Theorem~\ref{thm:intro-mesh-realization}) and by complete twisted
Ginzburg dg algebras for $d=2$ in arbitrary characteristic without a
Calabi--Yau hypothesis
(Theorem~\ref{thm:intro-twisted-ginzburg-normal-form}).

\subsection*{Auslander correspondence and cluster Morita theory}

The problem of recognizing triangulated categories with a cluster
tilting object by algebraic models has a rich history. Over an
algebraically closed field, Keller and Reiten showed that a Hom-finite
idempotent-complete algebraic $2$-Calabi--Yau triangulated category with
a cluster tilting object whose endomorphism quiver is acyclic is
triangle equivalent to the classical cluster category of that quiver in
Keller's triangulated orbit model \cite{Kel05,KR08}. Subsequent work
extended such recognition theorems to generalized cluster categories via
dg quotients \cite{ART11,Han22Morita}, while Keller's Calabi--Yau
completions provided a systematic supply of connective dg algebras for
these quotients \cite{Kel11}.

A parallel line of research stems from the classical Auslander
correspondence, which characterizes endomorphism algebras of additive
generators of representation-finite module categories by having global
dimension at most two and dominant dimension at least two
\cite{Aus66,Au71}. In higher Auslander--Reiten theory, Iyama replaced
the additive generator by a $d$-cluster tilting module and replaced both
homological bounds by $d+1$ \cite{Iya05,Iya07a,Iya07b,Iya11,IJ17}. In
both settings, homological conditions on the endomorphism algebra
characterize precisely those algebras that arise from a distinguished
subcategory of an ambient category. See also Mochizuki's Auslander
correspondence for proper connective dg algebras \cite{Moc26}.

In the triangulated setting, let $k$ be a commutative ring, let $\C$ be
a pretriangulated dg $k$-category with idempotent-complete homotopy
category $\T=H^0\C$, and let $\M\subseteq\tau^{\leq0}\C$ be a full
additive connective dg subcategory. Recall from \cite{IYos} that $\M$ is
\emph{$d$-cluster tilting} if $H^0\M$ is functorially finite in $\T$
(meaning that every object of $\T$ admits both a left and a right
$H^0\M$-approximation) and
\[
\begin{aligned}
H^0\M
&=\{X\in\T\mid \T(H^0\M,X[i])=0\text{ for }1\leq i<d\}\\
&=\{X\in\T\mid \T(X,H^0\M[i])=0\text{ for }1\leq i<d\}.
\end{aligned}
\]
In a $2$-Calabi--Yau category, Serre duality interchanges the two
orthogonality conditions for $d=2$ \cite{KR07}. Iterated right
$H^0\M$-approximations yield a tower of triangles whose remainder lies
in $H^0\M$ after $d-1$ steps by the orthogonality conditions; hence
every object of $\T$ is an iterated extension of objects in
$H^0\M,H^0\M[1],\ldots,H^0\M[d-1]$. In particular, $\M$ generates $\T$.
By derived Morita theory \cite{Kel94,Toen07}, $\C$ is therefore
recovered from the full dg subcategory of $\C$ on the objects of $\M$,
retaining morphism complexes in all cohomological degrees. A central
feature of our cluster Morita theorem is that, because $\M$ is
$d$-cluster tilting, all positive-degree morphism data are redundant:
the connective truncation $\M\subseteq\tau^{\leq0}\C$ alone determines
$\C$ uniquely up to quasi-equivalence.

To characterize the connective dg categories that arise in this way, one
must capture the higher exactness inherited from the ambient category
$\C$. When $H^0\M$ is closed under the $d$-fold suspension $[d]$, Geiss,
Keller and Oppermann showed that it carries a $(d+2)$-angulated
structure \cite[Theorem~1]{GKO} (see also \cite{Jasso,HLN21} for
$d$-exact and $d$-exangulated analogues). A general $d$-cluster tilting
subcategory need not be preserved by $[d]$, however, and passing to the
homotopy category $H^0\M$ forgets the higher homotopies that govern
iterated approximations. Retaining the connective dg subcategory
$\M\subseteq\tau^{\leq0}\C$ overcomes both obstructions: iterated
approximations produce homotopy $d$-kernels and $d$-cokernels inside
$\M$.

Let $\M$ be an additive connective dg category with $H^0\M$
idempotent complete. A one-sided twisted complex in degrees
$0,\ldots,d+1$ is \emph{left $d$-exact} if applying $\M(Z,-)$ and
totalizing gives vanishing cohomology through degree $d$ for every
$Z\in\M$. It then gives a \emph{homotopy $d$-kernel} of its last
morphism. Right $d$-exactness and homotopy $d$-cokernels are defined
dually using $\M(-,Z)$ (Definition~\ref{def:homotopy-d-kernel}).
Following Mochizuki and Nakaoka \cite[Definition~4.12]{MN26}, we call
$\M$ \emph{$d$-stable} if every closed degree-zero morphism has both a
homotopy $d$-kernel and a homotopy $d$-cokernel and the two exactness
conditions coincide (Definition~\ref{def:d-stable}). For $d=1$, this
recovers Chen's stable dg categories \cite[Definition~6.1]{Che26}.
Mochizuki and Nakaoka proved that every $d$-cluster tilting subcategory
$\M\subseteq\tau^{\leq0}\C$ is $d$-stable
\cite[Corollary~7.23]{MN26} (see also
Proposition~\ref{prop:cluster-tilting-d-stable}).

Our first main theorem characterizes $d$-stability purely in terms of
homological properties of $\M$-modules, providing a dg analogue of the
Auslander--Iyama correspondence (Theorem~\ref{thm:auslander-correspondence}).

\begin{IntroThm}[Auslander correspondence for $d$-stable dg categories]
\label{thm:intro-auslander-correspondence}
Let $d\geq1$ and let $\M$ be an additive connective dg category over a
commutative ring $k$ with $H^0\M$ idempotent complete. Then $\M$ is
$d$-stable if and only if
\begin{enumerate}
\item\label{it:intro-auslander-coherence} $\M$ is coherent,
\item\label{it:intro-auslander-dimension} $\wgl\M\leq d+1$, and
\item\label{it:intro-auslander-duality}
the standard perfect-module duality $\RHom_\M(-,\M)$ restricts to
a duality
\[
\mod H^0\M\xrightarrow{\ \sim\ }\mod (H^0\M)^{\op}[-d-1].
\]
\end{enumerate}
\end{IntroThm}

Here $\mod H^0\M$ denotes the category of finitely presented right
$H^0\M$-modules. Roughly speaking, coherence and the bound on weak
global dimension correspond to the existence of homotopy $d$-kernels and
$d$-cokernels (Proposition~\ref{prop:homotopy-kernels}), while the
shifted duality in (\ref{it:intro-auslander-duality}) encodes the
coincidence of left and right $d$-exactness. Moreover, conditions
(\ref{it:intro-auslander-coherence}) and
(\ref{it:intro-auslander-dimension}) imply that
every finitely presented $H^0\M$-module is perfect over $\M$, so the
full dg subcategory $\D^b_{\fp,\dg}(\M)$ of dg $\M$-modules with bounded
finitely presented cohomology is contained in $\per_{\dg}\M$. We can
therefore define the \emph{$d$-cluster dg category} of a $d$-stable dg
category $\M$ as the dg quotient
\[
\C_{d,\dg}(\M):=\per_{\dg}\M/^{\mathbb L}\D^b_{\fp,\dg}(\M).
\]
The following theorem establishes the correspondence in the diagram
above and describes a fundamental domain for this quotient.

\begin{IntroThm}[Fundamental domain and cluster Morita theorem]
\label{thm:intro-cluster-morita}
Let $d\geq1$.
\begin{enumerate}
\item\label{it:intro-fundamental-domain}
If $\M$ is a $d$-stable dg category, let $\F_\M$ be the full
dg subcategory of $\per_{\dg}\M$ on the extension closure
$\M*\M[1]*\cdots*\M[d-1]$. Then the quotient functor induces
a quasi-equivalence
\[
\tau^{\leq0}\F_\M\xrightarrow{\ \sim\ }\tau^{\leq0}\C_{d,\dg}(\M)
\]
identifying $\M$ with a $d$-cluster tilting subcategory of
$\tau^{\leq0}\C_{d,\dg}(\M)$. Moreover, $H^0\C_{d,\dg}(\M)$ is
idempotent complete.
\item\label{it:intro-cluster-reconstruction}
If $\C$ is a pretriangulated dg category with $H^0\C$
idempotent complete and $\M\subseteq\tau^{\leq0}\C$ is a $d$-cluster
tilting subcategory, then $\M$ is $d$-stable and there is a
quasi-equivalence $\C_{d,\dg}(\M)\simeq\C$ compatible with the
inclusions of $\M$.
\end{enumerate}
\end{IntroThm}

In particular, when $k$ is a field, Theorem~\ref{thm:intro-cluster-morita}
completely answers the opening question on reversing the Amiot--Guo
construction. Call a locally finite connective dg $k$-algebra $A$
\emph{$d$-stable} if $\proj_{\dg}A$ is $d$-stable. If $A$ is $d$-stable, then
$\D^b_{\fp,\dg}(\proj_{\dg}A)\simeq\pvd_{\dg}A\subseteq\per_{\dg}A$,
so the quotient $\C_{d,\dg}(A)=\per_{\dg}A/^{\mathbb L}\pvd_{\dg}A$ is
again locally finite, and the correspondence restricts to mutually
inverse bijections between quasi-isomorphism classes of locally finite
connective $d$-stable dg $k$-algebras and quasi-equivalence classes of
locally finite pretriangulated dg $k$-categories with idempotent-complete
homotopy category and a chosen $d$-cluster tilting object
(Corollary~\ref{thm:Amiot-conj}).

The fundamental domain in
Theorem~\ref{thm:intro-cluster-morita}(\ref{it:intro-fundamental-domain})
extends those of Amiot \cite[Theorem~2.9]{Ami09} and Guo
\cite[Theorem~2.15]{Guo}; for $d=1$, the realization is due to Chen
\cite[Proposition~3.26]{Che24}. In our proof
(Theorem~\ref{thm:fundamental-domain}), the standard $t$-structure on
$\D(\M)$ and the shifted duality from
Theorem~\ref{thm:intro-auslander-correspondence} cut out the shift
window $0,\ldots,d-1$. For reconstruction in
(\ref{it:intro-cluster-reconstruction})
(Theorem~\ref{thm:morita-cluster-dgcat}), the extension of
$\M\hookrightarrow\C$ to $\per_{\dg}\M\to\C$ identifies
$\tau^{\leq0}\F_\M$ with $\tau^{\leq0}\C$ using $d$-rigidity, and
this upgrades to $\C_{d,\dg}(\M)\simeq\C$ because a pretriangulated dg
category is determined by its connective truncation
(Lemma~\ref{lem:connective-recovery}).

Our results complement several recent reconstruction theorems in
related settings. Independently, Hanihara obtained a quotient reconstruction under a
uniform bound on the lengths of right $H^0\M$-resolutions
\cite[Theorems~5.16--5.17]{Han26Truncation}, which in particular yields
the reconstruction in
Theorem~\ref{thm:intro-cluster-morita}(\ref{it:intro-cluster-reconstruction})
for $d$-cluster tilting subcategories
\cite[Corollary~5.22]{Han26Truncation}. In our approach, the intrinsic
homological characterization of $d$-stability in
Theorem~\ref{thm:intro-auslander-correspondence} produces the explicit
fundamental domain in
Theorem~\ref{thm:intro-cluster-morita}(\ref{it:intro-fundamental-domain})
and establishes the two-sided correspondence between $d$-stable dg
categories and pairs $(\C,\M)$.
In the
periodic setting of the derived Auslander--Iyama correspondence, Jasso
and Muro recover algebraic triangulated categories over a perfect field
with a $d\mathbb Z$-cluster tilting object from periodic algebra data
and prove uniqueness of dg enhancements \cite{JM26}; by contrast, we
allow general $d$-cluster tilting subcategories without periodicity and
determine the ambient dg category relative to the given connective
enhancement. In the exact-category setting, Kvamme proved an analogous
realization for weakly idempotent-complete $d$-exact categories, where
the ambient exact category is unique up to exact equivalence
\cite{Kva26}. Finally, Fushimi recently established a cluster Morita
theorem for negative cluster categories using simple-minded systems and
cosingular dg quotients \cite{Fus26Morita}.

\subsection*{Calabi--Yau structures and Amiot's conjecture}

We now work over a field $k$. In the one-object case,
Corollary~\ref{thm:Amiot-conj} realizes a pair $(\C,M)$ as
$\C_{d,\dg}(A_M)$. When $\C$ carries a Calabi--Yau structure, it is
natural to ask what additional structure is induced on $A_M$ and whether
$A_M$ admits a concrete presentation. For $d=2$, presenting $A_M$ as a
Ginzburg dg algebra realizes $(\C,M)$ by the cluster category of a
quiver with potential and identifies $\End_{H^0\C}(M)=H^0A_M$ with its
Jacobian algebra. This is directly related to Amiot's conjecture asking
whether every $2$-Calabi--Yau tilted algebra is Jacobian
\cite[Question~2.20(1)]{Ami11}.

Recall that a \emph{right $n$-Calabi--Yau structure} on a dg
$k$-category is a class in dual cyclic homology $D\!\HC_{-n}$ whose
image in dual Hochschild homology $D\!\HH_{-n}$ induces a nondegenerate
bimodule pairing, where $D=\Hom_k(-,k)$. On a locally finite
pretriangulated dg category, such a structure induces Serre duality with
shift $[n]$ on the homotopy category. Amiot used relative duality to
construct a Serre functor on the Verdier quotient
\cite[Section~1]{Ami09}. At the dg level, Keller's localization sequence
\cite{Kel99} yields connecting maps in Hochschild and cyclic homology,
and Hanihara and Liu established a dg descent formalism for Calabi--Yau
structures \cite[Section~4]{HL25}. Combining their framework with the
cohomological gap coming from $d$-cluster tilting, we extend Keller and
Liu's nondegeneracy detection \cite[Proposition~6.4.1]{KL23} to ordinary
dg categories.

\begin{IntroThm}[Calabi--Yau correspondence]
\label{thm:intro-cy}
Let $A$ be a locally finite connective $d$-stable dg algebra over a
field $k$. The connecting isomorphism in dual cyclic homology induces a
bijection
\[
\begin{Bmatrix}
\text{right $(d+1)$-Calabi--Yau structures}\\
\text{on $\pvd_{\dg}A$}
\end{Bmatrix}
\xrightarrow{\ \sim\ }
\begin{Bmatrix}
\text{right $d$-Calabi--Yau structures}\\
\text{on $\C_{d,\dg}(A)$}
\end{Bmatrix}.
\]
Moreover, the connecting isomorphism on dual Hochschild homology
preserves and detects nondegeneracy.
\end{IntroThm}

More generally, Theorem~\ref{thm:intro-cy} holds for any locally finite
$d$-stable dg $k$-category $\M$, with $\D^b_{\fp,\dg}(\M)$ in place of
$\pvd_{\dg}A$ (Theorem~\ref{thm:delta-detects-nondegeneracy} and
Corollary~\ref{cor:bij-CYstr}).

Applied to $A=A_M$, the inverse bijection in Theorem~\ref{thm:intro-cy}
lifts a right $d$-Calabi--Yau structure on $\C$ to a right
$(d+1)$-Calabi--Yau structure on $\pvd_{\dg}A_M$ over a field of
arbitrary characteristic. In characteristic zero, Keller and Liu proved
a realization theorem for Amiot's conjecture when the ambient
triangulated category is the stable category of a Frobenius exact
pseudo-compact category and the associated dg enhancement carries a
right $d$-Calabi--Yau structure \cite[Theorem~6.2.1]{KL23}.
Their pseudo-compact model provides a dg algebra to which Van den
Bergh's structure theorem \cite[Theorem~9.4]{VdB15} applies.
By constructing a minimal completed tensor model directly from the
ordinary dg algebra $A_M$ (Proposition~\ref{prop:direct-minimal-tensor-model})
and lifting the Calabi--Yau structure via Theorem~\ref{thm:intro-cy}, we
obtain a dg realization from an ordinary dg enhancement.

\begin{IntroThm}[Calabi--Yau realization in characteristic zero]
\label{thm:intro-cy-realization}
Let $k$ be a field of characteristic zero and let $d\geq1$. Let $\C$ be
a locally finite pretriangulated dg $k$-category with $H^0\C$ idempotent
complete, endowed with a right $d$-Calabi--Yau structure, and let $M$ be
a basic $d$-cluster tilting object.
\begin{enumerate}
\item\label{it:intro-cy-preprojective}
There is a $(d+1)$-dimensional deformed dg preprojective algebra
$\Pi$ with finite-dimensional $H^0\Pi$ and a quasi-equivalence of dg
categories
\[
\C_{d,\dg}(\Pi)\simeq\C,
\qquad
\overline\Pi\longmapsto M.
\]
\item If $d=2$ and $k$ is algebraically closed, there is a Jacobi-finite
quiver with formal potential $(Q,W)$ such that the algebra in
(\ref{it:intro-cy-preprojective}) can be chosen as $\Pi=\Gamma(Q,W)$.
\end{enumerate}
\end{IntroThm}

More precisely, in characteristic zero the semisimple quotient of
$H^0A_M$ is separable over $k$, so $A_M$ is quasi-isomorphic to a
minimal completed tensor dg algebra $\Gamma$
(Proposition~\ref{prop:direct-minimal-tensor-model}). Transporting the
lifted right $(d+1)$-Calabi--Yau structure to $\pvd^{\mathrm{pc}}_{\dg}\Gamma$
and applying Keller and Liu's left--right correspondence
\cite[Proposition~4.4.1]{KL23} yields a left $(d+1)$-Calabi--Yau
structure on $\Gamma$, to which Van den Bergh's structure theorem
\cite[Theorem~9.4]{VdB15} (with the case $d=1$ treated as in
\cite[Section~6.4]{KL23}) applies directly; see
Theorem~\ref{thm:Amiot-without-pseudocompact} and
Corollary~\ref{cor:Amiot-quiver-potential}.

As we show in
Theorem~\ref{thm:intro-twisted-ginzburg-normal-form}(\ref{it:intro-twisted-ginzburg-realization})
below, when $d=2$ the algebra $A_M$ of a basic $2$-cluster tilting
object with separable semisimple endomorphism quotient also admits a
complete twisted Ginzburg presentation in arbitrary characteristic,
without any Calabi--Yau hypothesis.

\subsection*{Koszul duality and the Yoneda algebra}

For a locally finite connective dg algebra $A$ over a field $k$, the
Auslander criterion in Theorem~\ref{thm:intro-auslander-correspondence}
can be tested on the simple $H^0A$-modules. Set
$S=H^0A/\rad H^0A$. Its Koszul dual $A^!=\REnd_A(S)$ is a locally finite
positive dg algebra whose cohomology is the Yoneda algebra
$E_A=\Ext_A^*(S,S)$. Using Fushimi's contravariant Koszul duality
between locally finite connective and positive dg algebras
\cite{BGS96,Fus}, we translate the shifted duality on $\mod H^0A$ into
shifted self-injectivity of $A^!$, and show that this dg property is
already detected by the ordinary graded algebra $E_A$.

\begin{IntroThm}[Yoneda--Frobenius criterion]
\label{thm:intro-yoneda-frobenius}
Let $d\geq1$ and let $A$ be a locally finite connective dg algebra over
a field $k$. Set $S=H^0A/\rad H^0A$ and $E_A=\Ext_A^*(S,S)$. Then $A$ is
$d$-stable if and only if $E_A$ is finite dimensional and
\[
\add DE_A=\add E_A[d+1]
\qquad\text{as graded right $E_A$-modules}.
\]
If $S$ is basic, this additive-closure condition is equivalent to the
graded Frobenius condition $DE_A\cong E_A[d+1]$.
\end{IntroThm}

Every morphism of graded $E_A$-modules out of a shift of $E_A$ lifts to
a morphism out of the corresponding shift of $A^!$ in $\D(A^!)$, and
this lifting property passes to finite direct sums and direct summands.
This explains why the ordinary Yoneda multiplication, with no higher
$A_\infty$-operations, suffices to detect shifted self-injectivity
(Proposition~\ref{prop:cohomological-frobenius} and
Theorem~\ref{thm:yoneda-frobenius}). Note that this Frobenius
isomorphism is only one-sided, whereas a Calabi--Yau structure requires
a bimodule pairing together with a cyclic lift.

Koszul duality also provides a second description of the $d$-cluster dg
category. Under the equivalent conditions of
Theorem~\ref{thm:intro-yoneda-frobenius}, the Koszul dual $A^!$ is
proper, and composing Koszul duality with $k$-duality interchanges the
perfect and perfectly valued subcategories. Passing to dg quotients
yields a quasi-equivalence
\[
\C_{d,\dg}(A)\simeq\sg_{\dg}(A^!)
\qquad
\bigl(\sg_{\dg}(A^!):=\pvd_{\dg}A^!/^{\mathbb L}\per_{\dg}A^!\bigr),
\]
connecting our construction with the singularity categories of Buchweitz
and Orlov and with Hanihara's comparison for formal dg algebras
(Corollary~\ref{cor:koszul-singularity}) \cite{Buch21,Orl04,Han20}.

\subsection*{Completed tensor models}

When a connective dg algebra is presented by a minimal completed tensor
model, its Yoneda multiplication can be read directly from the quadratic
part of the differential. Over a field $k$, let $R$ be a finite-dimensional
basic semisimple $k$-algebra and let $V$ be a graded $R$-bimodule
concentrated in nonpositive degrees and finite dimensional in each
degree. Write $A=(\widehat T_R(V),\partial)$ for the degreewise
tensor-length completion equipped with a continuous degree-one
derivation $\partial$ satisfying $\partial^2=0$, $\partial|_R=0$, and
$\partial(V)\subseteq\widehat T_R(V)_{\geq2}$. Let $\partial_2$ be the
quadratic part of $\partial$ and denote the right
$R$-dual by $(-)^\vee=\Hom_R(-,R)$.

\begin{IntroThm}[Quadratic recognition and completed tensor models]
\label{thm:intro-mesh-models}
Let $d\geq1$. For a minimal completed tensor dg algebra $A$ as above,
$\Ext_A^*(R,R)$ is graded Frobenius of parameter $d+1$ if and only if
\begin{enumerate}
\item[(M1)] $V^{<-d}=0$,
\item[(M2)] $V^{-d}$ is an invertible $R$-bimodule, and
\item[(M3)] for every $0\leq i<d$, contraction with the
$(-i,-d+i+1)$-component of $\partial_2|_{V^{-d}}$ induces an isomorphism of
$R$-bimodules
\[
(V^{-i})^\vee\otimes_R V^{-d} \xrightarrow{\ \sim\ } V^{-d+i+1}.
\]
\end{enumerate}
If $H^0A$ is finite dimensional, these conditions are equivalent to
$d$-stability of $A$. Moreover, a locally finite connective dg
$k$-algebra $B$ whose semisimple quotient is basic and separable over
$k$ is $d$-stable if and only if it is quasi-isomorphic to a minimal
completed tensor dg algebra over that quotient satisfying
{\rm (M1)--(M3)}.
\end{IntroThm}

Explicit minimal projective resolutions of the semisimple $A$-modules
identify the Yoneda product with the dual of $\partial_2$, turning the
Frobenius pairings between complementary Yoneda degrees into the
contractions in {\rm (M3)}. Separability is needed only to construct a
minimal model of a given dg algebra, whereas quadratic recognition
applies to any given model over an arbitrary field
(Theorem~\ref{thm:quadratic-recognition},
Corollary~\ref{cor:d-stable-species-model}, and
Proposition~\ref{prop:direct-minimal-tensor-model}).

\paragraph{Mesh models for $d=1$.}
When $d=1$, the orthogonality conditions in the definition of
$d$-cluster tilting are vacuous, so a $1$-cluster tilting object is
simply an additive generator of the entire homotopy category.
Riedtmann's theory of stable translation quivers \cite{Rie80} underlies
the Dynkin tree classes of finite triangulated categories identified by
Xiao--Zhu \cite{XZ05} and Amiot \cite{Ami07}, and Hanihara characterized
the endomorphism algebras of their basic additive generators over a
perfect field by twisted $3$-periodicity \cite{Han20Aus}.

Our framework yields a complete dg presentation directly from the
Auslander--Reiten species, which encodes the vertex division algebras,
the bimodules of irreducible morphisms, and the translation and pairings
defining the mesh relations. For a finite stable translation species
$\mathfrak S$, let $A(\mathfrak S,\delta)$ be the minimal completed
tensor dg algebra whose differential on the degree $-1$ generators is
$\rho_\kappa+\delta$, where $\rho_\kappa$ is the mesh tensor and
$\delta$ takes values in tensor length at least three. When every
component of $\mathfrak S$ has Dynkin tree class, we call
$A(\mathfrak S,\delta)$ a \emph{deformed derived mesh algebra}
(Definition~\ref{def:derived-mesh-algebra}).

\begin{IntroThm}[Dynkin criterion and realization by mesh dg algebras]
\label{thm:intro-mesh-realization}
Let $k$ be a field.
\begin{enumerate}
\item\label{it:intro-dynkin-criterion}
Let $\mathfrak S$ be a finite stable translation species and $\delta$
a higher-length perturbation. Then $H^0A(\mathfrak S,\delta)$ is finite
dimensional if and only if every valued tree class of $\mathfrak S$ is
Dynkin, in which case $A(\mathfrak S,\delta)$ is $1$-stable.
\item\label{it:intro-mesh-realization}
Let $\C$ be a locally finite pretriangulated dg $k$-category with
$H^0\C$ idempotent complete and of finite type. Let $M$ be its basic
additive generator and set $\Lambda=\End_{H^0\C}(M)$. If
$\Lambda/\rad\Lambda$ is separable over $k$, then the
Auslander--Reiten species $\mathfrak S_{\mathrm{AR}}$ of $H^0\C$ has
valued Dynkin tree classes, and there is a deformed derived mesh
algebra $A(\mathfrak S_{\mathrm{AR}},\delta)$ quasi-isomorphic to $A_M$.
The model induces a quasi-equivalence
\[
\C_{1,\dg}\bigl(A(\mathfrak S_{\mathrm{AR}},\delta)\bigr)\simeq\C
\]
carrying the image of the free module to $M$.
\end{enumerate}
\end{IntroThm}

These results are proved in Theorems~\ref{thm:derived-mesh-stable}
and~\ref{thm:ar-mesh-presentation} and
Corollary~\ref{cor:translation-species-realization}, using a positive
weight criterion for the Cartan matrix in
(\ref{it:intro-dynkin-criterion}) and Auslander--Reiten triangles
arising from minimal resolutions of the vertex simples in
(\ref{it:intro-mesh-realization}). At the level of ordinary algebras,
modifying mesh relations by higher-length terms also appears in the
deformed mesh algebras of Bia{\l}kowski, Erdmann, and Skowro\'nski
\cite{BES12Mesh}.

\paragraph{Twisted Ginzburg models for $d=2$.}
When $d=2$, the contractions in {\rm (M3)} pair the degree-zero
generators with the degree $-1$ generators. The invertible bimodule in
degree $-2$ induces an automorphism twist of the coefficient algebra
$R$, and the differential in degree $-2$ extends this twist to a
continuous automorphism $\theta$ of the degree-zero completed tensor
algebra. Together with a \emph{reduced formal $\theta$-twisted
hyperpotential} $\rho$---which specifies the degree-zero relations and
satisfies the compatibility condition forced by $\partial^2=0$---this
defines a \emph{complete twisted Ginzburg dg algebra}, generalizing the
Ginzburg construction \cite{Gin06} in Ladkani's hyperpotential framework
\cite{Lad14}.

\begin{IntroThm}[Twisted Ginzburg normal form and realization]
\label{thm:intro-twisted-ginzburg-normal-form}
Let $k$ be a field of arbitrary characteristic.
\begin{enumerate}
\item A minimal completed tensor dg algebra over a finite-dimensional
basic semisimple $k$-algebra $R$ satisfies {\rm (M1)--(M3)} for $d=2$ if
and only if it admits a complete twisted Ginzburg presentation over $R$
arising from a reduced formal twisted hyperpotential. When $H^0$ is
finite dimensional, it is $2$-stable.
\item\label{it:intro-twisted-ginzburg-realization}
Let $\C$ be a locally finite pretriangulated dg $k$-category with
$H^0\C$ idempotent complete, and let $M$ be a basic $2$-cluster tilting
object whose endomorphism algebra has a separable semisimple quotient.
Then $A_M$ is quasi-isomorphic as an ordinary dg algebra to a complete
twisted Ginzburg dg algebra $\Gamma$. In particular, there is a
quasi-equivalence $\C_{2,\dg}(\Gamma)\simeq\C$ sending the image of the
free module to $M$.
\end{enumerate}
\end{IntroThm}

For fixed generators and duality data, the normal form uniquely
determines both the twist $\theta$ and the required change of degree
$-1$ generators (Theorem~\ref{thm:twisted-mesh-basis-characterization}).
In the split case $R=k^{Q_0}$, the vertex permutation of $\theta$ shifts
the endpoints of the reverse arrows, while its full action on arrows
enters the degree $-2$ differential. This twist is visible even when all
vertices are fixed: for the Kronecker quiver with zero hyperpotential
($\rho=0$), the automorphism exchanging the two parallel arrows
preserves $H^0=kQ$ but alters the Yoneda multiplication, so the two
resulting dg algebras---and hence the corresponding pairs $(\C,M)$---are not
quasi-equivalent (Example~\ref{ex:kronecker-arrow-twist}).

Part~(\ref{it:intro-twisted-ginzburg-realization}) gives a twisted
analogue of the categorical realization in Amiot's conjecture.
Over an algebraically closed field, the separability hypothesis is
automatic, and the model takes the explicit quiver form
$\Gamma(Q,\theta,\rho)$. By choosing a dg enhancement \cite{Kel06}, this
also realizes any Hom-finite idempotent-complete algebraic triangulated
category with a basic $2$-cluster tilting object, whose ordinary
endomorphism algebra is $H^0\Gamma$. Notably, this realization relies
only on $2$-stability and requires no Calabi--Yau hypothesis
(Corollary~\ref{cor:twisted-ginzburg-realization}).

\subsection*{Organization of the paper}
Section~\ref{sec:preliminaries} collects homological preliminaries on
connective dg categories and characterizes the existence of homotopy
$d$-kernels.
Sections~\ref{sec:Aus-corr} and~\ref{sec:cluster-cat} prove the
Auslander correspondence, the fundamental domain theorem, and the
cluster Morita theorem. Section~\ref{sec:Amiot-conj} establishes the
Calabi--Yau correspondence and the characteristic-zero realization.
Section~\ref{sec:loc-fin} develops the Koszul-dual and Yoneda--Frobenius
criteria together with the singularity category description.
Section~\ref{sec:quadratic-recognition} proves quadratic recognition and
constructs the mesh and twisted Ginzburg models. Finally,
Appendix~\ref{app:dg-species} constructs minimal completed tensor models
and computes minimal resolutions of semisimple modules and their Yoneda
products.

\subsection*{Conventions}
Throughout the paper, $k$ denotes a commutative ring.  All dg categories
are small, $k$-linear, and cohomologically graded.  Unless otherwise specified, modules are
right modules, and subcategories are full.  We call $\A$
\emph{connective} if $H^i\A(X,Y)=0$ for all $i>0$ and all $X,Y\in\A$.
For every connective additive dg category $\A$,  we
assume that its homotopy category $H^0\A$ is idempotent complete.
For a connective dg category $\A$, we denote by $\D^-_{\fp}(\A)$, $\D^{\le 0}_{\fp}(\A)$, and $\D^b_{\fp}(\A)$ the full subcategories of $\D(\A)$ consisting of dg modules whose cohomology is bounded above, concentrated in nonpositive degrees, and bounded, respectively, with $H^iM\in\mod H^0\A$ for all $i\in\mathbb Z$; we write $\D^b_{\fp,\dg}(\A)$ for the corresponding full dg subcategory.

When $k$ is a field, a dg category $\A$ is called {\it locally finite} if $H^i\A(X,Y)$ is finite dimensional for all $X,Y\in\A$ and all
$i\in\mathbb Z$, and it is \emph{proper} if
$\bigoplus_{i\in\mathbb Z}H^i\A(X,Y)$ is finite dimensional for every pair
$X,Y\in\A$. A dg algebra $A$ is called \emph{positive} if $H^iA=0$ for all
$i<0$ and $H^0A$ is semisimple.

\subsection*{AI Declaration}
The author used ChatGPT and Codex (5.6 Sol Extra high) to explore
simplifications of part of the proof of
Theorem~\ref{thm:delta-detects-nondegeneracy} and to revise the exposition
of earlier drafts.
In Section~\ref{sec:quadratic-recognition} and
Appendix~\ref{app:dg-species}, the author used Codex (6 Sol) to help
develop rigorous proofs for general species from the author's original
ideas in the quiver case.
The author also used Codex (6.2 Sol) and Gemini Neon to restructure
and rewrite the introduction and to refine the proofs.
The author takes full responsibility for the manuscript.

\section*{Acknowledgements}
The author expresses his gratitude to Xiaofa Chen, Riku Fushimi, Norihiro Hanihara, Osamu Iyama, Gustavo Jasso, Bernhard Keller and Nao Mochizuki for fruitful discussions. This work was supported by the WINGS-FMSP program at the Graduate School of Mathematical Sciences, the University of Tokyo, and JSPS KAKENHI Grant Number JP25KJ0818.

\section{Preliminaries}\label{sec:preliminaries}

This section collects the homological tools used in the proof of the
Auslander correspondence.  We first study homological dimensions and
finite projective resolutions over a connective dg category.  Modules
over $H^0\A$ will be identified with the heart of the standard
$t$-structure on $\D(\A)$.

\begin{Def}
Let $\A$ be a connective dg category.
\begin{enumerate}
\item For $M\in\D(\A)$ and $n\in\mathbb{Z}$, we write $\fd_{\A}M\le n$ if $M\otimes_{\A}^{\mathbb{L}}N\in\D^{\ge-n}(k)$ holds for every $N\in\Mod H^0\A^{\op}\subseteq\D(\A^{\op})$.
\item For $M\in\D(\A)$ and $n\in\mathbb{Z}$, we write $\pd_{\A}M\le n$ if $\RHom_\A(M,N)\in\D^{\le n}(k)$ holds for every $N\in\Mod H^0\A\subseteq\D(\A)$.
\item For $n\ge0$, we write $\wgl\A\le n$ if $\fd_{\A}M\le n$ holds for every $M\in\Mod H^0\A\subseteq\D(\A)$.
\item For $n\ge0$, we write $\rgl\A\le n$ if $\pd_{\A}M\le n$ holds for every $M\in\Mod H^0\A\subseteq\D(\A)$.
\end{enumerate}
We define the left global dimension by
\[
\lgl\A:=\rgl(\A^{\op}).
\]
\end{Def}

\begin{Rem}
When $\A=A$ is a connective dg ring, the right global dimension in (4)
agrees with the global dimension introduced by Minamoto
\cite[Theorem~5.1]{Min21}, and with the global dimension of $\D(A)$
relative to its standard heart $\Mod H^0A$ in
\cite[Definitions~2.2 and 2.4]{Kos26}: indeed,
$\pd_A M\le n$ is equivalent to $\D(A)(M,N[i])=0$ for every
$N\in\Mod H^0A$ and $i>n$.
For arbitrary dg modules, our projective dimension uses a different
normalization from \cite{Min21}: if $H^0A\ne0$, then
$\pd_A(A[n])=n$, whereas all shifts of $A$ have projective dimension
zero in Minamoto's convention; see also \cite[Remark~2.5]{Kos26}.
\end{Rem}

\begin{Rem}
Note that the weak global dimension is left-right symmetric by the symmetry of the derived tensor product:
\[\wgl\A=\wgl\A^{\op}.\]
\end{Rem}

To control finitely presented cohomology throughout such resolutions,
we use the following dg version of coherence.

\begin{Def}
A connective dg category $\A$ is called
\begin{enumerate}
\item {\it right coherent} if $H^0\A$ is right coherent, and $H^{-n}\A(-,A)$ is finitely presented as a right $H^0\A$-module for every $A\in\A$ and $n\ge0$.
\item {\it left coherent} if $\A^{\op}$ is right coherent.
\item {\it coherent} if $\A$ is both right coherent and left coherent.
\end{enumerate}
\end{Def}

Note that $\A$ is right coherent if and only if $H^0\A$ is right coherent and $\per\A\subseteq\D^-_{\fp}(\A)$. In this form, right
coherence leads naturally to the following projective resolution
criterion.

\begin{Lem}\label{projrescoh}
Let $\A$ be a right coherent connective dg category. Then for $M\in\D^{\le0}(\A)$, the following conditions are equivalent.
\begin{enumerate}
\item $M\in\D_{\fp}^{\le0}(\A)$.
\item There exist exact triangles
\[M_{i+1}\to P_i\to M_i\dashrightarrow\]
in $\D(\A)$ for $i\ge0$ such that $M_0=M$, $M_i\in\D^{\le0}(\A)$, and
$P_i\in\add(\A)$.
\end{enumerate}
\end{Lem}

\begin{proof}
Since $H^0\A$ is right coherent, $\mod H^0\A\subseteq\Mod H^0\A$ is a wide
subcategory (i.e., an exact abelian subcategory closed under extensions). Thus $\D^-_{\fp}(\A)\subseteq\D(\A)$ is a triangulated subcategory. Moreover, the right coherence of $\A$ implies $\add \A\subseteq\D^{\le0}_{\fp}(\A)$.

Suppose first that $M\in\D_{\fp}^{\le0}(\A)$.  We construct the triangles
inductively.  Set $M_0=M$.  If $M_i\in\D_{\fp}^{\le0}(\A)$, take a right $\add\A$-approximation
\[
P_i\longrightarrow M_i,
\]
which exists since $H^0M_i$ is finitely generated as an $H^0\A$-module. Complete it to an exact triangle
\[
M_{i+1}\longrightarrow P_i\longrightarrow M_i\dashrightarrow.
\]
Since $M_i,P_i\in\D^-_{\fp}(\A)$, we have $M_{i+1}\in\D^-_{\fp}(\A)$. The surjectivity of $H^0(P_i)\to H^0(M_i)$ and the long exact cohomology
sequence show that $M_{i+1}\in\D^{\le0}(\A)$.

Conversely, suppose that the triangles in (2) are given.  Repeated use
of the octahedral axiom yields, for every $r\ge0$, an object
\[
Q_r\in\add(\A)*\add(\A)[1]*\cdots*\add(\A)[r]
\]
and an exact triangle
\[
M_{r+1}[r]\longrightarrow Q_r\longrightarrow M\dashrightarrow.
\]
Indeed, one may take $Q_0=P_0$; if $Q_r$ has been constructed, apply the
octahedral axiom to the morphisms
\[
P_{r+1}[r]\longrightarrow M_{r+1}[r]\longrightarrow Q_r
\]
to obtain $Q_{r+1}\in Q_r*\add(\A)[r+1]$ and the corresponding triangle.
Since $M_{r+1}\in\D^{\le0}(\A)$, the morphism $Q_r\to M$ induces an
isomorphism
\[
H^{-n}(Q_r)\xrightarrow{\sim}H^{-n}(M)
\qquad\text{for all }0\le n<r.
\]
Since $Q_r\in\D^-_{\fp}(\A)$ holds, this implies $H^{-n}(M)\in\mod H^0\A$ for $0\le n<r$. Thus we obtain $M\in\D^-_{\fp}(\A)$.
\end{proof}

The preceding lemma constructs possibly infinite resolutions. The
following proposition characterizes when such a resolution terminates after a
prescribed number of steps; it will be the main link between weak
global dimension and finite projective resolutions.
For dg rings, its projective part is the version of Minamoto's
criterion using sup-projective resolutions
\cite[Theorem~2.22]{Min21} with finitely generated projective terms.
The proof below also treats dg categories.

\begin{Prop}\label{prop:fp-flat-projective}
Let $\A$ be a right coherent connective dg category and let
$M\in\D_{\fp}^{\le0}(\A)$.  For every $n\ge0$, the following conditions are equivalent.
\begin{enumerate}
\item $\fd_{\A}M\le n$.
\item $\pd_{\A}M\le n$.
\item We have
\[
M\in\add(\A)*\add(\A)[1]*\cdots *\add(\A)[n].
\]
\end{enumerate}
\end{Prop}
\begin{proof}
Since $\A$ is connective, we may replace $\A$ by its nonpositive truncation $\tau^{\le 0}\A$. Set $\B=H^0\A$, and
write $\overline X=X\otimes_{\A}^{\mathbb L}\B$.
Associativity and adjunction give
\begin{equation}\label{eq:base-change-fd-pd}
\fd_{\A}X=\fd_{\B}\overline X,
\qquad
\pd_{\A}X=\pd_{\B}\overline X.
\end{equation}
Restriction from $\B$ to $\A$ is $t$-exact and identifies the
hearts. Its left adjoint is therefore right $t$-exact, and
adjunction against modules placed in degree $m$ gives
\[
H^m(\overline X)\cong H^mX
\qquad\bigl(X\in\D^{\le m}(\A)\bigr).
\]
Thus reduction is conservative on $\D^-(\A)$: the highest nonzero
cohomology of a nonzero bounded-above object remains nonzero.

The implication $(3)\Rightarrow(2)$ follows because
$\RHom_{\A}(P,N)$ is concentrated in degree zero for
$P\in\add\A$ and $N\in\Mod\B$.
The implication $(2)\Rightarrow(1)$ follows from the base-change
identities~\eqref{eq:base-change-fd-pd} and the usual inequality
between flat and projective dimensions over $\B$.

Assume (1). By Lemma~\ref{projrescoh}, take triangles
\[
M_{i+1}\longrightarrow P_i\longrightarrow M_i\dashrightarrow
\qquad(0\le i<n),
\]
where $M_0=M$, $P_i\in\add\A$, and
$M_i\in\D_{\fp}^{\le0}(\A)$. Since each $P_i\in\add\A$ satisfies $\fd_{\A}P_i\le 0$, dimension shifting gives
$\fd_{\A}T\le0$ for $T:=M_n$.
Then $\overline T\in\D^{\le0}(\B)$ has flat dimension at most zero,
so it is a flat $\B$-module concentrated in degree zero.
Indeed, tensoring with left representables detects the vanishing
of its negative cohomology, and tensoring with arbitrary left
modules gives flatness.
Since $H^0\overline T\cong H^0T$ is finitely presented,
$\overline T\in\add\B$.

By dg Yoneda and splitting idempotents in $\D(\A)$, $H^0$
induces an equivalence $\add\A\xrightarrow{\sim}\proj\B$ and isomorphisms
\[
\D(\A)(P,T)\xrightarrow{\sim}
\Hom_{\B}(H^0P,H^0T)\qquad(P\in\add\A).
\]
Choose $P\in\add\A$ and $f\colon P\to T$ inducing an isomorphism
on $H^0$. Both $\overline P$ and $\overline T$ belong to
$\add\B$, so $\overline f$ is an isomorphism.
Conservativity on $\D^-(\A)$ gives $T\cong P\in\add\A$;
splicing yields (3).
\end{proof}

For modules concentrated in degree zero, the preceding proposition gives
the following criterion for the weak global dimension.

\begin{Cor}\label{cor:weak-global-dimension-fp}
For a right coherent connective dg category $\A$ and $n\ge0$, the following conditions are equivalent.
\begin{enumerate}
\item $\wgl\A\le n$.
\item $\pd_{\A}M\le n$ for every $M\in\mod H^0\A$.
\end{enumerate}
\end{Cor}
\begin{proof}
The implication $(1)\Rightarrow(2)$ follows immediately from Proposition~\ref{prop:fp-flat-projective}. Conversely, assume $(2)$, so that $\fd_{\A}M\le n$ for every $M\in\mod H^0\A$ by Proposition~\ref{prop:fp-flat-projective}. Since every module in $\Mod H^0\A$ is a filtered colimit of modules in $\mod H^0\A$, and both the derived tensor product and cohomology commute with filtered colimits, we obtain $\fd_{\A}M\le n$ for all $M\in\Mod H^0\A$, hence $\wgl\A\le n$.
\end{proof}

We next extend the notion of a noetherian ring to connective dg rings,
using the term \emph{right noetherian} for Minamoto's
\emph{right piecewise noetherian} condition \cite{Min21}.
As in the ordinary case, right noetherianity will imply equality of the
right and weak global dimensions, and noetherianity will imply equality
with the left global dimension as well.

\begin{samepage}
\begin{Def}\label{def:noetherian-dg-ring}
A connective dg ring $A$ is called
\begin{enumerate}
\item {\it right noetherian} if $H^0A$ is right noetherian and
$H^{-n}A$ is finitely generated as a right $H^0A$-module for every $n\geq0$.
\item {\it left noetherian} if $A^{\op}$ is right noetherian.
\item {\it noetherian} if $A$ is both right noetherian and left noetherian.
\end{enumerate}
\end{Def}
\end{samepage}

\begin{Cor}[Right and weak global dimensions]
\label{cor:noetherian-global-dimension}
Let $A$ be a right noetherian connective dg ring. Then
\[
\rgl A=\wgl A.
\]
\end{Cor}
\begin{proof}
Since $A$ is connective, we may replace $A$ by its nonpositive truncation $\tau^{\le 0}A$ and set $B=H^0A$.
Since $A$ is right noetherian, it is right coherent.
The inequality $\wgl A\leq\rgl A$ follows from
\eqref{eq:base-change-fd-pd} and the inequality between flat and
projective dimensions over $B$.
For the converse, assume $\wgl A\leq n$ for some $n\geq0$.
By Corollary~\ref{cor:weak-global-dimension-fp},
\[
\pd_A M\leq n\qquad(M\in\mod B).
\]
We extend this bound to all $B$-modules using Baer's criterion.

Fix $N\in\Mod B$ and put $X=\RHom_A(B,N)\in\D^{\geq0}(B)$.
For every $M\in\Mod B$, adjunction gives
\[
\RHom_A(M,N)\simeq\RHom_B(M,X).
\]
Choose a bounded-below complex $I^\bullet$ of injective right
$B$-modules representing $X$, with $I^{<0}=0$.
Put $Z^nI^\bullet=\Ker(I^n\to I^{n+1})$.
Since $B\in\mod B$, we have $H^{>n}X=0$.
For every right ideal $\mathfrak a$ of $B$, the cyclic module
$B/\mathfrak a$ is finitely presented by right noetherianity.
Consequently, dimension shifting in $I^\bullet$ gives
\[
\Ext_B^1(B/\mathfrak a,Z^n I^\bullet)
\cong\Ext_B^{n+1}(B/\mathfrak a,X)
\cong\Ext_A^{n+1}(B/\mathfrak a,N)=0.
\]
Baer's criterion shows that $Z^n I^\bullet$ is injective.
Thus $X$ is represented by the complex of injectives
\[
I^0\longrightarrow\cdots\longrightarrow I^{n-1}
\longrightarrow Z^n I^\bullet
\]
concentrated in degrees $0,\ldots,n$; for $n=0$ this is just
$Z^0I^\bullet$ in degree zero.
It follows that $\Ext_A^{>n}(M,N)=0$ for every $M\in\Mod B$.
As $N$ was arbitrary, $\rgl A\leq n$.
\end{proof}

\begin{Def}[Global dimension]\label{def:global-dimension}
Let $A$ be a noetherian connective dg ring.
By Corollary~\ref{cor:noetherian-global-dimension} applied to $A$ and
$A^{\op}$ and the left-right symmetry of weak global dimension,
the right, left, and weak global dimensions coincide.
We call their common value the \emph{global dimension} of $A$, as in
\cite[Theorem~5.1]{Min21}, and write
\[
\gl A:=\rgl A=\lgl A=\wgl A.
\]
In particular, this notation applies to every locally finite
connective dg algebra over a field.
\end{Def}

We now pass from these module-theoretic preparations to higher kernels.
Bounded one-sided twisted complexes provide the appropriate dg analogue
of finite complexes in an additive category.

\begin{Def}\label{def:homotopy-d-kernel}
Let $\A$ be an additive dg category and let $d\ge1$.
\begin{enumerate}
\item A bounded one-sided twisted complex in $\A$ is a pair
$X^\bullet=((X^i)_{i\in\mathbb Z},(d_X^{i,j})_{i,j\in\mathbb Z})$, where
$X^i\in\A$ vanishes for all but finitely many $i$, and
\[
d_X^{i,j}\in\A(X^i,X^j)^{i-j+1},\qquad d_X^{i,j}=0\quad(i\ge j),
\]
such that
\[
d_{\A}(d_X^{i,j})+(-1)^j\sum_{k\in\mathbb Z}
d_X^{k,j}\circ d_X^{i,k}=0.
\]
A degree-$m$ morphism $f^\bullet\colon X^\bullet\to Y^\bullet$ is a
family $f^{i,j}\in\A(X^i,Y^j)^{i-j+m}$, with differential
\[
(df)^{i,j}=(-1)^jd_{\A}(f^{i,j})+
\sum_{k\in\mathbb Z}\left(d_Y^{k,j}\circ f^{i,k}
-(-1)^mf^{k,j}\circ d_X^{i,k}\right).
\]
Together with matrix composition, these form a dg category
$\operatorname{tw}(\A)$.  We write
$\operatorname{htw}(\A)=H^0\operatorname{tw}(\A)$ and regard an object
$A\in\A$ as the twisted complex concentrated in degree zero; see
\cite[Definition~2.3]{MN26}.

\item Let $X^\bullet\in\operatorname{tw}(\A)$ be concentrated in
degrees $0,\ldots,d+1$.  Following
\cite[Definition~3.1 and Corollary~3.21]{MN26}, it is called \emph{left
$d$-exact} if
\[
\operatorname{htw}(\A)(A[-i],X^\bullet)=0
\]
for every $A\in\A$ and every $i\le d$.  For
$f\colon X^d\to X^{d+1}$ in $Z^0\A$, such a twisted complex is called a
\emph{homotopy $d$-kernel} of $f$ if $d_X^{d,d+1}=f$.

\item Right $d$-exactness and homotopy $d$-cokernels are defined
dually, by passing to $\A^{\op}$ and reversing the indices.
Explicitly, $X^\bullet$ as in (2) is \emph{right $d$-exact} if
\[
\operatorname{htw}(\A)(X^\bullet,A[j])=0
\qquad(A\in\A,\ j<0).
\]
It is a \emph{homotopy $d$-cokernel} of
$f\colon X^0\to X^1$ if it is right $d$-exact and
$d_X^{0,1}=f$.
\end{enumerate}
\end{Def}

The following proposition is the dg counterpart of the corresponding
characterization for ordinary additive categories: an idempotent
complete additive category $\C$ has $d$-kernels if and only if $\C$ is
right coherent and the abelian category $\mod\C$ has global dimension
at most $d+1$; see
\cite[Proposition~5]{Gul25}.  It translates the existence of homotopy
$d$-kernels into the homological conditions established above.

\begin{Prop}\label{prop:homotopy-kernels}
Let $d\ge1$ be an integer. For an additive connective dg category $\A$ such that $H^0\A$ is idempotent complete, the following conditions are equivalent.
\begin{enumerate}
\item Each morphism in $Z^0\A$ has a homotopy $d$-kernel.
\item $\A$ is right coherent and $\wgl\A\le d+1$ holds.
\end{enumerate}
\end{Prop}
\begin{proof}
Put $\B=H^0\A$ and write $h_A=\A(-,A)$ for $A\in\A$.  For a twisted
complex $X^\bullet$ concentrated in degrees $0,\ldots,d+1$, its Yoneda
totalization is
\[
\operatorname{Tot}(h_X):=\bigoplus_{i=0}^{d+1}h_{X^i}[-i],
\]
whose differential is induced by the $d_X^{i,j}$.  The dg Yoneda lemma
gives
\[
H^i\operatorname{Tot}(h_X)(A)
 \cong\operatorname{htw}(\A)(A[-i],X^\bullet).
\]
Since $\A$ is connective, $\operatorname{Tot}(h_X)\in\D^{\le d+1}(\A)$,
and its cohomology in degree $d+1$ is
$\Cok\B(-,d_X^{d,d+1})$.  Consequently, $X^\bullet$ is left $d$-exact
if and only if
\begin{equation}\label{eq:homotopy-kernel-resolution}
\operatorname{Tot}(h_X)[d+1]
 \simeq\Cok\bigl(\B(-,d_X^{d,d+1})\bigr)
\end{equation}
in $\D(\A)$. Thus a homotopy
$d$-kernel is equivalently such a resolution whose first differential
is the prescribed morphism.

Suppose first that (1) holds.  Since $\A$ is additive and $\B$ is
idempotent complete, every morphism between finitely generated
projective $\B$-modules is isomorphic to $\B(-,f)$ for some
$f\in Z^0\A$.  Put
\[
M_f=\Cok\B(-,f),\qquad
C_f=\operatorname{Cone}(h_X\xrightarrow{h_f}h_Y).
\]
A homotopy $d$-kernel of $f$ and
\eqref{eq:homotopy-kernel-resolution} give an exact triangle
\begin{equation}\label{eq:two-term-remainder}
C_f\longrightarrow M_f\longrightarrow E_f\dashrightarrow
\end{equation}
with
\[
E_f\in\add(\A)[2]*\add(\A)[3]*\cdots*\add(\A)[d+1].
\]
In particular, $E_f\in\D^{\le-2}(\A)$, and $H^{-2}E_f$ is a finitely
presented $\B$-module.  Indeed, the finite projective filtration yields a
triangle
\[
P_2[2]\longrightarrow E_f\longrightarrow E_f'\dashrightarrow
\]
with $P_2\in\add(\A)$ and, when $d\ge2$,
$E_f'\in\add(\A)[3]*\cdots*\add(\A)[d+1]$; for $d=1$, take
$E_f'=0$.  Hence
$H^0P_2\to H^{-2}E_f$ is surjective and its kernel is the image of
$H^{-3}E_f'$.  The latter module is finitely generated, since it is a
quotient of $H^0P_3$ for some $P_3\in\add(\A)$ when $d\ge2$, and it
is zero when $d=1$.  Thus $H^{-2}E_f$ is finitely presented.

Apply this construction to the zero morphism $0\to A$.  Then
$C_f=h_A$ and $M_f=H^0h_A$, so \eqref{eq:two-term-remainder} becomes
the exact triangle
\begin{equation}\label{eq:negative-cohomology-triangle}
h_A\longrightarrow H^0h_A\longrightarrow E_A\dashrightarrow
\end{equation}
and gives $H^{-1}h_A\cong H^{-2}E_A$.  Thus $H^{-1}h_A$ is finitely
presented for every $A\in\A$.

For a general $f\colon X\to Y$, the long exact sequence of
\eqref{eq:two-term-remainder} gives
$H^{-1}C_f\cong H^{-2}E_f$.  On the other hand, the triangle
\[
h_X\xrightarrow{h_f}h_Y\longrightarrow C_f\dashrightarrow
\]
gives an exact sequence
\[
H^{-1}h_Y\longrightarrow H^{-1}C_f\longrightarrow
\Ker\B(-,f)\longrightarrow0.
\]
It follows that $\Ker\B(-,f)$ is finitely presented.  Hence kernels of
morphisms between finitely generated projective $\B$-modules are
finitely presented, and therefore $\B$ is right coherent.

It remains to show that $H^{-r}h_A$ is finitely presented for every
$A\in\A$ and $r\ge2$.  We argue by induction on $r$, the cases $r=0,1$
having already been established.  Assume that $H^{-s}h_C$ is finitely
presented for all $C\in\A$ and $0\le s<r$.
Choose a finite extension filtration
\[
E_A=E_A^{(2)},\quad E_A^{(d+2)}=0,
\]
together with exact triangles
\[
P_i[i]\longrightarrow E_A^{(i)}\longrightarrow
E_A^{(i+1)}\dashrightarrow
\qquad(2\le i\le d+1),
\]
where $P_i\in\add(\A)$.  We prove by descending induction on $i$ the stronger assertion
\[
H^jE_A^{(i)}\in\mod\B
\qquad(j\geq1-r-i).
\]
It holds for $i=d+2$, since $E_A^{(d+2)}=0$.
For $2\leq i\leq d+1$ and $j\geq1-r-i$, use the five-term
exact sequence
\[
\begin{aligned}
H^{j-1}E_A^{(i+1)}&\longrightarrow H^j(P_i[i])
\longrightarrow H^jE_A^{(i)}\\
&\longrightarrow H^jE_A^{(i+1)}
\longrightarrow H^{j+1}(P_i[i]).
\end{aligned}
\]
The two terms involving $E_A^{(i+1)}$ are finitely presented by
descending induction. Since $j+i\geq1-r$, the two terms involving
$P_i$ are either zero by connectivity or finitely presented by the
induction on $r$. Their cokernel and kernel are therefore finitely
presented, since $\mod\B$ is abelian. The middle term is an extension
of these modules and is finitely presented as well.
Taking $i=2$ and $j=-r-1$ gives $H^{-r-1}E_A\in\mod\B$.
The long exact
sequence of
\eqref{eq:negative-cohomology-triangle} gives
\[
H^{-r}h_A\cong H^{-r-1}E_A,
\]
so the induction closes.  Therefore $\A$ is right coherent.

Now let $M\in\mod\B$.  Choose a projective presentation of $M$ and,
using the idempotent completeness of $\B$, write it as
\[
\B(-,X)\xrightarrow{\B(-,f)}\B(-,Y)\longrightarrow M\longrightarrow0
\]
for some $f\in Z^0\A$.  A homotopy $d$-kernel of $f$ gives, by
\eqref{eq:homotopy-kernel-resolution},
\[
M\in\add(\A)*\add(\A)[1]*\cdots*\add(\A)[d+1].
\]
Proposition~\ref{prop:fp-flat-projective} yields $\pd_{\A}M\le d+1$.
Corollary~\ref{cor:weak-global-dimension-fp} now gives $\wgl\A\le d+1$, proving (2).

Conversely, suppose that (2) holds, and let $f\colon X\to Y$ be a
morphism in $Z^0\A$.  Set
\[
M=\Cok\bigl(\B(-,f)\colon\B(-,X)\to\B(-,Y)\bigr).
\]
Then $M\in\mod\B$, and Corollary~\ref{cor:weak-global-dimension-fp} gives
$\pd_{\A}M\le d+1$.  In the construction of Lemma \ref{projrescoh}, we may take
$P_0=h_Y$ and the canonical epimorphism $h_Y\to M$.  If
\[
M_1\longrightarrow h_Y\longrightarrow M\dashrightarrow
\]
is the resulting triangle, the morphism $h_X\to h_Y$ induced by $f$
lifts to $h_X\to M_1$, and its map on zeroth cohomology is surjective.
Choose the lift so that its composite with $M_1\to h_Y$ is $h_f$.
Thus we may take $P_1=h_X$ and continue the construction of Lemma \ref{projrescoh}.
We obtain triangles
\[
M_{i+1}\longrightarrow P_i\longrightarrow M_i\dashrightarrow
\qquad(0\le i\le d)
\]
with $M_0=M$, $P_0=h_Y$, and $P_1=h_X$.  Dimension shifting gives
$\fd_{\A}M_{d+1}\le0$: indeed,
Proposition~\ref{prop:fp-flat-projective} first gives
$\fd_{\A}M\le d+1$. Thus the same proposition gives
$M_{d+1}\in\add(\A)$.

By the idempotent completeness of $\B$, all the objects in
$\add(\A)$ occurring above may be represented by objects of $\A$.
Convolving these triangles then produces a twisted complex $X^\bullet$
concentrated in degrees $0,\ldots,d+1$, with $X^d=X$,
$X^{d+1}=Y$, and $d_X^{d,d+1}=f$.  Its Yoneda totalization shifted by
$d+1$ is isomorphic to $M$.  By
\eqref{eq:homotopy-kernel-resolution}, $X^\bullet$ is left $d$-exact.
Hence it is a homotopy $d$-kernel of $f$, proving (1).
\end{proof}

\section{Auslander correspondence for higher stable dg categories}\label{sec:Aus-corr}

The Auslander correspondence in this section characterizes higher
stability by coherence, weak global dimension, and a duality on
finitely presented modules
(Theorem~\ref{thm:auslander-correspondence}).
Section~\ref{sec:loc-fin} will specialize this characterization to
locally finite dg algebras and their Yoneda algebras.

We recall the notion of a $d$-stable dg category introduced by Mochizuki and Nakaoka \cite[Definition~4.12]{MN26}, generalizing Chen's stable dg categories \cite[Definition~6.1]{Che26}.

\begin{Def}[{\cite[Definition~4.12]{MN26}}]\label{def:d-stable}
Let $d\ge1$. An additive
connective dg category $\M$ is called \emph{$d$-stable} if the following
conditions hold.
\begin{enumerate}
\item Every morphism in $Z^0\M$ has a homotopy $d$-kernel and a
homotopy $d$-cokernel.
\item A twisted complex $X^\bullet\in\operatorname{tw}(\M)$ concentrated
in degrees $0,\ldots,d+1$ is left $d$-exact if and only if it is right
$d$-exact.
\end{enumerate}
\end{Def}

For $d=1$, this is the notion of a stable dg category in
\cite[Definition~6.1]{Che26}; see also
\cite[Proposition~4.11, Definition~4.12, and Proposition~4.13]{MN26}.
The next lemma explains how a stable connective dg category retains
the positive-degree morphisms of its pretriangulated realization.

\begin{Lem}[Recovery from connective truncation]
\label{lem:connective-recovery}
A stable connective dg category $\A$, equipped with its canonical
exact structure, satisfies
\[
\A\simeq\tau^{\leq0}\D^b_{\dg}(\A).
\]
If $\C$ is pretriangulated, then
\[
\D^b_{\dg}(\tau^{\leq0}\C)\simeq\C,
\]
where the truncation has the canonical exact structure.
Consequently, a quasi-equivalence between the connective truncations
of pretriangulated dg categories induces a quasi-equivalence between
their pretriangulated realizations.
\end{Lem}
\begin{proof}
The first two assertions are
\cite[Proposition~3.26 and the following paragraph]{Che24}.
A quasi-equivalence between stable connective dg categories
preserves homotopy short exact sequences, hence their canonical
exact structures. It therefore induces a quasi-equivalence on their
bounded dg derived categories, proving the last assertion.
\end{proof}

\begin{Def}\label{def:d-cluster-tilting}
Let $d\ge 1$ and let $\C$ be a pretriangulated dg category with $\T=H^0\C$.
A full additive subcategory $\M\subseteq\tau^{\leq0}\C$ is called
\emph{$d$-cluster tilting} if $H^0\M$ is functorially finite in $\T$
(every object has both a right and a left $H^0\M$-approximation) and
\[
\begin{aligned}
H^0\M
&=\{X\in\T\mid \T(H^0\M,X[i])=0\text{ for }1\leq i<d\}\\
&=\{X\in\T\mid \T(X,H^0\M[i])=0\text{ for }1\leq i<d\}.
\end{aligned}
\]
For a stable connective dg category, we define $d$-cluster tilting
subcategories via its pretriangulated realization from
Lemma~\ref{lem:connective-recovery}.
For $d=1$, the vanishing conditions are vacuous, so $H^0\M=H^0\C$
and the inclusion $\M\hookrightarrow\tau^{\leq0}\C$ is a
quasi-equivalence.
\end{Def}

The basic examples come from cluster tilting theory: a $d$-cluster
tilting subcategory of a stable dg category inherits $d$-stability
from its ambient category.

\begin{Prop}[{\cite[Corollary~7.23]{MN26}}]\label{prop:cluster-tilting-d-stable}
Let $\A$ be a stable dg category and let $\M\subseteq\A$ be a $d$-cluster tilting subcategory. Then $\M$ is $d$-stable.
\end{Prop}
\begin{proof}
The proof of \cite[Corollary~7.23]{MN26} passes through $d$-exact dg
categories.  We instead verify the two conditions in
Definition~\ref{def:d-stable} directly for the convenience of the reader.

By Lemma~\ref{lem:connective-recovery}, realize $\A$ as
$\tau^{\leq0}\C$ for a pretriangulated dg category $\C$, and put
$\T=H^0\C$. Conflations in $\A$ are represented by triangles
in $\T$; cluster-tilting orthogonality is measured in $\T$.
We may assume $d\ge2$. Let $f\colon M^d\to M^{d+1}$ be a morphism in $Z^0\M$. Since $\A$ is
stable, $f$ is a deflation and hence occurs in a conflation
\[
K^{d-1}\longrightarrow M^d\xrightarrow{f}M^{d+1}.
\]
For $j=d-1,\ldots,1$, choose a right $H^0\M$-approximation
$M^j\to K^j$.  It is again a deflation, so it occurs in a conflation
\[
K^{j-1}\longrightarrow M^j\longrightarrow K^j,
\qquad M^j\in\M.
\]
For $M\in\M$ and $1\le i\le d-1$,
these conflations give
\[
\T(M,K^0[i])\cong\cdots\cong
\T(M,K^{i-1}[1])=0.
\]
Here the last vanishing follows from the approximation property and
$\T(\M,\M[1])=0$.  Since $\M\subseteq\A$ is a $d$-cluster tilting subcategory, we have $K^0\in\M$. Put $M^0=K^0$.

Splicing these conflations yields a twisted complex
$Y_f^\bullet\in\operatorname{tw}(\M)$ with terms
$M^0,\ldots,M^{d+1}$ and $d_{Y_f}^{d,d+1}=f$
(see \cite[Proposition~7.3]{MN26}). By
\cite[Propositions~7.3 and 7.14]{MN26}, it is both left and right
$d$-exact.  In particular, it is a homotopy $d$-kernel of $f$.
The dual construction gives a homotopy $d$-cokernel of every morphism
in $Z^0\M$, proving condition~(1) of
Definition~\ref{def:d-stable}.

Finally, let $X^\bullet\in\operatorname{tw}(\M)$ be left $d$-exact and
put $f=d_X^{d,d+1}$.  Both $X^\bullet$ and $Y_f^\bullet$ are homotopy
$d$-kernels of $f$, so they are isomorphic in the homotopy category of
twisted complexes by
\cite[Corollary~3.13 and Remark~3.14]{MN26}.  Hence $X^\bullet$ is
right $d$-exact.  The dual argument proves the converse, and therefore
condition~(2) holds.
\end{proof}

We now establish an Auslander correspondence for $d$-stable dg categories, which is our main theorem in this section.

\begin{Thm}[Auslander correspondence for $d$-stable dg categories]\label{thm:auslander-correspondence}
Let $\M$ be an additive connective dg category. Let $d\ge1$ be an integer. Then $\M$ is $d$-stable if and only if the following conditions are satisfied.
\begin{enumerate}
\item $\M$ is coherent.
\item $\wgl\M\le d+1$.
\item The duality
\[\RHom_\M(-,\M)\colon\per\M\xrightarrow[\simeq]{}\per\M^{\op}\]
restricts to a duality
\[\mod H^0\M\xrightarrow[\simeq]{}\mod H^0\M^{\op}[-d-1].\]
\end{enumerate}
\end{Thm}
\begin{proof}
Put $\B=H^0\M$ and let
\[
(-)^*=\RHom_{\M}(-,\M)\colon
\per\M\longrightarrow\per\M^{\op}
\]
be the standard duality; we use the same notation for its
quasi-inverse.

We first relate the two exactness conditions in
Definition~\ref{def:homotopy-d-kernel} to modules.  Let $X^\bullet$ be
a twisted complex concentrated in degrees $0,\ldots,d+1$, and put
\[
Q_X=\operatorname{Tot}(h_X),\qquad T_X=Q_X[d+1].
\]
By \eqref{eq:homotopy-kernel-resolution},
\begin{equation}\label{eq:left-exact-module}
X^\bullet\text{ is left $d$-exact}
\quad\Longleftrightarrow\quad
T_X\in\mod\B.
\end{equation}
On the other hand, for $A\in\M$ and $j\in\mathbb Z$, dg Yoneda gives
\begin{equation}\label{eq:dual-totalization-cohomology}
H^j(Q_X^*)(A)
 \cong \operatorname{htw}(\M)(X^\bullet,A[j]).
\end{equation}
Therefore \eqref{eq:dual-totalization-cohomology} and our convention for right
$d$-exactness imply
\[
X^\bullet\text{ is right $d$-exact}
\quad\Longleftrightarrow\quad
Q_X^*\in\mod\B^{\op}.
\]
As $(Q_X[d+1])^*\simeq Q_X^*[-d-1]$, we have obtained
\begin{equation}\label{eq:right-exact-module}
X^\bullet\text{ is right $d$-exact}
\quad\Longleftrightarrow\quad
T_X^*\in\mod\B^{\op}[-d-1].
\end{equation}

We shall use one further consequence of Lemma~\ref{projrescoh} and
Proposition~\ref{prop:fp-flat-projective}.  If an additive connective dg category
$\A$ is right coherent and $\wgl\A\le d+1$, then every
$F\in\mod H^0\A$ is isomorphic to $T_X$ for some $X^\bullet$
concentrated in degrees $0,\ldots,d+1$.  Indeed,
$\fd_{\A}F\le d+1$, so Proposition~\ref{prop:fp-flat-projective}, applied to
the resolution in Lemma~\ref{projrescoh}, gives
$F\in\add(\A)*\cdots*\add(\A)[d+1]$.  By our convention, the terms
in $\add(\A)$ are represented by objects of $\A$, and the Yoneda
correspondence for one-sided twisted complexes \cite[Definition~2.3]{MN26}
therefore gives $F\simeq T_X$.  Equation~\eqref{eq:left-exact-module}
then shows that $X^\bullet$ is left $d$-exact.

Suppose first that $\M$ is $d$-stable. Reversing the indices turns
a homotopy $d$-cokernel over $\M$ into a homotopy $d$-kernel over
$\M^{\op}$. Applying Proposition~\ref{prop:homotopy-kernels} to
$\M$ and $\M^{\op}$ therefore shows that $\M$ is coherent and
$\wgl\M\le d+1$.

Let $F\in\mod\B$.  By the preceding realization, there is a left
$d$-exact twisted complex $X^\bullet$ such that $F\simeq T_X$.
Since $\M$ is $d$-stable, $X^\bullet$ is right $d$-exact, and
\eqref{eq:right-exact-module} gives
\[
F^*\in\mod\B^{\op}[-d-1].
\]
Applying the same argument to $\M^{\op}$ shows that the quasi-inverse
duality sends $G\in\mod\B^{\op}$ to
$G^*\in\mod\B[-d-1]$, and therefore sends $G[-d-1]$ to
$G^*[d+1]\in\mod\B$.  Hence the standard duality restricts to the
duality in condition~(3).

Conversely, suppose that conditions~(1)--(3) hold.  By
Proposition~\ref{prop:homotopy-kernels}, applied to $\M$ and
$\M^{\op}$ and using $\wgl\M^{\op}=\wgl\M$, every morphism in
$Z^0\M$ has a homotopy $d$-kernel and a homotopy $d$-cokernel.

Let $X^\bullet$ be concentrated in degrees $0,\ldots,d+1$.  By
condition~(3) and the full faithfulness of the ambient duality,
\[
T_X\in\mod\B
\quad\Longleftrightarrow\quad
T_X^*\in\mod\B^{\op}[-d-1].
\]
Together with \eqref{eq:left-exact-module} and
\eqref{eq:right-exact-module}, this says precisely that left and right
$d$-exactness coincide. Thus $\M$ is $d$-stable.
\end{proof}

By the condition (3), we can show that the equality in the condition (2) holds unless $H^0\M=0$.

\begin{Cor}
Let $\M$ be a $d$-stable dg category with $H^0\M\ne0$. Then we have
\[\wgl\M=d+1.\]
\end{Cor}
\begin{proof}
By Theorem~\ref{thm:auslander-correspondence}, we have $\wgl\M\le d+1$. Suppose for contradiction that $\wgl\M\le d$. Then Proposition~\ref{prop:fp-flat-projective} implies that every $M\in\mod H^0\M$ lies in $\add(\M)*\cdots*\add(\M)[d]$, and hence $M^*=\RHom_\M(M,\M)\in\D^{\le d}(\M^{\op})$. On the other hand, condition~(3) of Theorem~\ref{thm:auslander-correspondence} gives $M^*\in\mod(H^0\M)^{\op}[-d-1]\subseteq\D^{\ge d+1}(\M^{\op})$. Thus $M^*=0$, so $M=0$ for every $M\in\mod H^0\M$, contradicting $H^0\M\ne0$.
\end{proof}

\section{Cluster categories of higher stable dg categories}\label{sec:cluster-cat}

For a $d$-stable dg category $\M$, the shift $[-d-1]$ in the
perfect-module duality of Theorem~\ref{thm:auslander-correspondence}
determines a fundamental domain for its $d$-cluster dg category.
We use this domain to reconstruct a pretriangulated dg category
from a $d$-cluster tilting subcategory.

The definition below extends the construction of Amiot, Guo, and
Keller, who take the quotient of a perfect derived category by its
finite-dimensional derived category
\cite[Theorem~2.1]{Ami09}\cite[Theorem~2.2]{Guo}\cite[Section~6.11]{Kel11}.

\begin{Def}\label{def:cluster-dgcat}
For a $d$-stable dg category $\M$, we define the {\it $d$-cluster dg category} $\C_{d,\dg}(\M)$ of $\M$ as
\[\C_{d,\dg}(\M):=\per_{\dg}\M/^\mathbb{L}\D^b_{\fp,\dg}(\M).\]
We define the {\it $d$-cluster category} $\C_d(\M)$ of $\M$ as
\[\C_d(\M):=H^0\C_{d,\dg}(\M)=\per\M/\D^b_{\fp}(\M).\]
\end{Def}

The next result generalizes the fundamental-domain equivalences of
Amiot \cite[Theorem~2.9]{Ami09} and Guo
\cite[Theorem~2.15]{Guo}.  Here,
the fundamental domain is expressed intrinsically as an extension
closure of shifts of $\M$. For $d=1$, the statement is due to Chen
\cite[Proposition~3.26]{Che24}.

\begin{Thm}\label{thm:fundamental-domain}
Let $\M$ be a $d$-stable dg category. Then the composition
\[\M*\M[1]*\cdots*\M[d-1]\hookrightarrow\per_{\dg}\M\to\C_{d,\dg}(\M)\]
is a connective equivalence, that is, the functor
\[\tau^{\le0}(\M*\M[1]*\cdots*\M[d-1])\to\tau^{\le0}\C_{d,\dg}(\M)\]
is a quasi-equivalence. Moreover, $\C_d(\M)$ is idempotent complete,
and the quasi-fully faithful functor
\[\M\to\tau^{\le0}\C_{d,\dg}(\M)\]
gives a $d$-cluster tilting subcategory.
\end{Thm}
\begin{proof}
Put
\[
\T=\per\M,\qquad \S=\D^b_{\fp}(\M),\qquad
\pi\colon\T\longrightarrow\T/\S=\C_d(\M).
\]
By Theorem~\ref{thm:auslander-correspondence}, the dg category $\M$
is coherent and $\wgl\M\le d+1$; hence $\S$ is a thick subcategory
of $\T$. The standard duality
\[
(-)^*=\RHom_{\M}(-,\M)\colon\T\xrightarrow{\sim}\per\M^{\op}
\]
restricts to $\S$ and, since the canonical $t$-structure on $\S$ is
bounded, satisfies
\[
 \bigl(\S\cap\D^{\le i}(\M)\bigr)^*
 =\D^b_{\fp}(\M^{\op})\cap
   \D^{\ge d+1-i}(\M^{\op})
 \qquad(i\in\mathbb Z).
\]

We seek representatives in the shift window $0,\ldots,d-1$.
The standard truncation will give one boundary of this window;
the Auslander duality gives the other. Set
\[
\X=\S\cap\D^{\le1}(\M),\qquad
\Y=\S\cap\D^{\ge2}(\M).
\]
Then $\S=\X\perp\Y$ is a $t$-structure. We claim that it extends on
the two sides to $t$-structures
\begin{equation}\label{eq:two-t-structures}
 \T=\X\perp\X^\perp={}^\perp\Y\perp\Y.
\end{equation}
Indeed, for $P\in\T$, the canonical truncation triangle
\[
 \tau^{\le1}P\longrightarrow P\longrightarrow\tau^{\ge2}P
 \dashrightarrow
\]
has $\tau^{\ge2}P\in\Y$: a perfect module over a connective coherent
dg category is bounded above with finitely presented cohomology, so
this truncation belongs to $\S$. This proves the second $t$-structure
in \eqref{eq:two-t-structures}. Applying the same argument to
$\M^{\op}$ at degree $d$ and then applying $(-)^*$ gives the first one.

The orthogonals can first be read in the standard $t$-structures:
\[
\begin{aligned}
{}^\perp(\Y[1])
  &=\T\cap\D^{\leq0}(\M),\\
\X^\perp
  &=\{P\in\T\mid P^*\in\D^{\leq d-1}(\M^{\op})\}.
\end{aligned}
\]
For the first equality, $\tau^{\geq1}P$ belongs to $\Y[1]$
for $P\in\T$, so orthogonality forces the canonical map
$P\to\tau^{\geq1}P$, and hence its positive cohomology, to vanish.
The converse is $t$-structure orthogonality.
For the second, duality takes $\X$ onto
$\D^b_{\fp}(\M^{\op})\cap\D^{\geq d}(\M^{\op})$;
apply the same truncation argument to $P^*$.

Every perfect object has finite projective dimension.
Proposition~\ref{prop:fp-flat-projective} therefore identifies
the first orthogonal with finite extensions of nonnegative
shifts of $\M$. Applying the same assertion to $P^*[d-1]$
identifies the second:
\begin{align*}
 {}^\perp(\Y[1])
   &=\bigcup_{r\geq0}\M*\M[1]*\cdots*\M[r],\\
 \X^\perp
   &=\bigcup_{r\geq0}\M[d-1-r]*\cdots*\M[d-1].
\end{align*}
These are the two bounds of a weight interval for the canonical
bounded co-$t$-structure with coheart $\M$ \cite{AI}.
Its decomposition triangles identify their intersection with
the extension closure of the shifts $0,\ldots,d-1$.
Consequently,
\begin{equation}\label{eq:fundamental-domain-identification}
 \Z:=\X^\perp\cap{}^\perp(\Y[1])
 =\M*\M[1]*\cdots*\M[d-1].
\end{equation}

Since $\X[1]\subseteq\X$, we have $\X[1]\cap\Y=0$. Thus
the quotient theorem of Iyama--Yang \cite[Theorem~1.1]{IYang20},
applied to \eqref{eq:two-t-structures}, gives an additive equivalence
\[
 \Z\hookrightarrow\T\xrightarrow{\pi}\T/\S.
\]
The two orthogonals defining $\Z$ are closed under direct summands in
the idempotent-complete category $\T$, so $\T/\S$ is idempotent complete.

Let $P,Q\in\Z$ and $i\ge0$. We have $P\in{}^\perp(\Y[1])$ and
$Q[-i]\in\X^\perp$, the latter since $\X[i]\subseteq\X$. Hence
\cite[Proposition~A.8(1)]{Tom25b} gives an isomorphism
\[
 \T(P,Q[-i])\xrightarrow{\sim}
 (\T/\S)(\pi P,\pi Q[-i]).
\]
These are precisely the maps on the $(-i)$-th cohomology of the morphism
complexes. Thus the induced quasi-functor
\[
 \tau^{\le0}\Z\longrightarrow\tau^{\le0}\C_{d,\dg}(\M)
\]
is quasi-fully faithful, and the preceding equivalence gives its
quasi-essential surjectivity. It is therefore a quasi-equivalence.

Write $\U$ for the essential image of $H^0\M$ in $\C_d(\M)$.
It is closed under direct summands by full faithfulness and
idempotent completeness of $H^0\M$. For $d=1$, it is the whole
category. For $d\geq2$, both $M$ and $N[i]$ lie in $\Z$ for
$M,N\in\M$ and $1\leq i<d$, so connectivity and full faithfulness
of $\pi|_\Z$ give $\C_d(\M)(\U,\U[i])=0$.
The fundamental domain and its shift by $1-d$ give triangles
\[
U_0\longrightarrow X\longrightarrow Y\dashrightarrow,\qquad
W\longrightarrow X\longrightarrow U_1\dashrightarrow
\]
for each $X\in\C_d(\M)$, with $U_0,U_1\in\U$,
$Y\in\U[1]*\cdots*\U[d-1]$, and
$W\in\U[1-d]*\cdots*\U[-1]$. The vanishing just proved makes
$U_0\to X$ a right approximation and $X\to U_1$ a left approximation.
If $\C_d(\M)(X,\U[i])=0$ for $1\leq i<d$, then $X\to Y$ vanishes,
so $X$ is a direct summand of $U_0$. If
$\C_d(\M)(\U,X[i])=0$ for $1\leq i<d$, then $W\to X$ vanishes,
so $X$ is a direct summand of $U_1$. Thus $\U$ is $d$-cluster tilting.
\end{proof}

As a corollary, every $d$-stable dg category can be realized as a $d$-cluster tilting subcategory of a stable dg category. Moreover, we can show that this ambient realization is unique up to quasi-equivalence. We call this reconstruction result the cluster Morita theorem.

\begin{Thm}[Cluster Morita theorem]\label{thm:morita-cluster-dgcat}
Let $\C$ be a pretriangulated dg category with $H^0\C$ idempotent
complete, and let $\M\subseteq\tau^{\le0}\C$ be a $d$-cluster
tilting subcategory. Then we have a quasi-equivalence
\[\C\simeq\C_{d,\dg}(\M)\]
which commutes with the inclusions from $\M$.
\end{Thm}
\begin{proof}
Since $\M$ is $d$-cluster tilting in
a stable dg category $\tau^{\le0}\C$, it is $d$-stable by Proposition \ref{prop:cluster-tilting-d-stable}; in particular, $\C_{d,\dg}(\M)$ is defined.

The inclusion $\M\hookrightarrow\C$ extends to a quasi-functor
\[
R\colon\per_{\dg}\M\longrightarrow\C.
\]
Set
\[
\Z:=\M*\M[1]*\cdots*\M[d-1]\subseteq\per_{\dg}\M.
\]
By Theorem~\ref{thm:fundamental-domain} and
Lemma~\ref{lem:connective-recovery}, it is enough to show that
$R|_{\Z}$ is a connective equivalence: this identifies the stable
connective truncations, from which their pretriangulated
realizations are recovered. For $X,Y\in\Z$, consider the comparison map
\[
 R_{X,Y}\colon\per_{\dg}\M(X,Y)\longrightarrow\C(RX,RY).
\]

First suppose that $X=M[a]$ and $Y=N[b]$, where $M,N\in\M$ and
$0\le a,b\le d-1$. The map $R_{X,Y}$ induces an isomorphism on $H^j$
for $j\le d-1+a-b$. Indeed, if $j+b-a\le0$, this follows from the
fullness of $\M\subseteq\tau^{\le0}\C$; if
$1\le j+b-a\le d-1$, both sides vanish by connectivity and the
$d$-cluster-tilting orthogonality.

Fix $X=M[a]$. We show by induction on $b$ that $R_{X,Y}$ induces an
isomorphism on $H^j$ for $j\le d-1+a-b$ whenever
$Y\in\M*\cdots*\M[b]$. Let
\[
 Y'\longrightarrow Y\longrightarrow N[b]\dashrightarrow
\]
be a triangle with $Y'\in\M*\cdots*\M[b-1]$ and $N\in\M$. For
brevity, write
\[
 E^j(U,V)=H^j\per_{\dg}\M(U,V),\qquad
 F^j(U,V)=H^j\C(RU,RV).
\]
The two long exact cohomology sequences form the following
commutative diagram with exact rows:
\[
\begin{tikzcd}[column sep=small]
 E^{j-1}(X,N[b]) \arrow[r] \arrow[d]
   & E^j(X,Y') \arrow[r] \arrow[d]
   & E^j(X,Y) \arrow[r] \arrow[d,"H^jR_{X,Y}"]
   & E^j(X,N[b]) \arrow[r] \arrow[d]
   & E^{j+1}(X,Y') \arrow[d] \\
 F^{j-1}(X,N[b]) \arrow[r]
   & F^j(X,Y') \arrow[r]
   & F^j(X,Y) \arrow[r]
   & F^j(X,N[b]) \arrow[r]
   & F^{j+1}(X,Y').
\end{tikzcd}
\]
Here all the vertical arrows are induced by $R$.
For $j\le d-1+a-b$, the four outer vertical maps are isomorphisms by
the preceding computation and the induction hypothesis. The five
lemma therefore shows that $H^jR_{X,Y}$ is an isomorphism. Taking
$b=d-1$, we conclude that $R_{M[a],Y}$ induces an isomorphism on
$H^{\le a}$ for every $Y\in\Z$.

For arbitrary $X,Y\in\Z$, induct on the extension length of $X$.
For a triangle $X'\to X\to M[a]\dashrightarrow$, with $a\geq1$,
the relevant five-term sequence has outer terms
\[
E^{j-1}(X',Y),\quad E^j(M[a],Y),\quad
E^j(X',Y),\quad E^{j+1}(M[a],Y).
\]
Their comparison maps are isomorphisms for $j\leq0$:
the first and third by induction, and the other two because
$j+1\leq a$. This degree margin closes the induction and proves
that $R_{X,Y}$ is an isomorphism on $H^{\leq0}$.

It remains to prove essential surjectivity. For $0\le a\le d-1$, put
\[
\Z_a=\M*\M[1]*\cdots *\M[a]\subseteq\Z.
\]
Since $\M$ is $d$-cluster tilting in $\C$, we have
\[
\C=\M*\M[1]*\cdots *\M[d-1].
\]
We show inductively that every object
$C\in\M*\cdots *\M[a]\subseteq\C$
is isomorphic to $RY$ for some $Y\in\Z_a$.
For $a\ge1$, take a triangle
\[
C'\to C\to M[a]\xrightarrow{\delta}C'[1]
\]
with $C'\in\M*\cdots *\M[a-1]\subseteq\C$. By induction, $C'\cong RY'$ for some $Y'\in\Z_{a-1}$. Since $Y'[1]\in\Z$, $H^0R_{M[a],Y'[1]}$ is an isomorphism by the preceding argument. Thus $\delta$
lifts to a morphism $M[a]\to Y'[1]$ in $\Z$. Completing this lift to a triangle
yields $Y\in\Z_a$, and exactness of $R$ gives $RY\cong C$.
Hence $R|_\Z$ is a connective equivalence. Together with
Theorem~\ref{thm:fundamental-domain}, it gives
\[
\tau^{\leq0}\C_{d,\dg}(\M)
\simeq\tau^{\leq0}\Z\simeq\tau^{\leq0}\C.
\]
Lemma~\ref{lem:connective-recovery} now yields the asserted
quasi-equivalence, compatibly with $\M$.
\end{proof}

\section{Calabi--Yau structures and Amiot-type realization}\label{sec:Amiot-conj}

Throughout this section, $k$ is a field.
We detect nondegeneracy on $d$-cluster dg categories and use the
cluster Morita theorem (Theorem~\ref{thm:morita-cluster-dgcat}) to
lift right Calabi--Yau structures. In characteristic zero,
Proposition~\ref{prop:direct-minimal-tensor-model} then allows us
to remove the pseudo-compact enhancement hypothesis from
Keller--Liu's realization theorem.

\subsection{Calabi--Yau structures on \texorpdfstring{$d$}{d}-cluster dg categories}

We recall the definition of right Calabi--Yau structures. Observe that for an integer $n\in\mathbb{Z}$ and a dg category $\A$, we have
\[\D(\A^e)(\A[n],D\A)\cong DH^0(\A\otimes_{\A^e}^\mathbb{L}\A[n])\cong D\!\HH_{-n}(\A).\]

\begin{Def}
Let $n\in\mathbb{Z}$, and let $\A$ be a dg category.
\begin{enumerate}
\item $\A$ is called {\it weak right $n$-Calabi--Yau} if we have an isomorphism
\[D\A\cong\A[n]\]
in $\D(\A^e)$.
\item An element $[x]\in D\!\HH_{-n}(\A)$ is called {\it nondegenerate} if the corresponding morphism $\A[n]\to D\A$ in $\D(\A^e)$ is an isomorphism.
\item An element $[\widetilde x]\in D\!\HC_{-n}(\A)$ is called a {\it right $n$-Calabi--Yau structure} if its image $[x]$ under the canonical map $D\!\HC_{-n}(\A)\to D\!\HH_{-n}(\A)$ is nondegenerate.
\end{enumerate}
\end{Def}

Weak right Calabi--Yau dg categories enhance Calabi--Yau triangulated categories: if $\A$ is a weak right $n$-Calabi--Yau pretriangulated dg category, then the triangulated category $H^0\A$ is $n$-Calabi--Yau.

Let $A$ be a locally finite connective dg algebra with $\gl A<\infty$.
Every finite-dimensional $H^0A$-module is perfect over $A$ by
Proposition~\ref{prop:fp-flat-projective}, so bounded truncation towers
give $\pvd_{\dg}A\subseteq\per_{\dg}A$. Set
\[
\A:=\pvd_{\dg}A,\qquad \B:=\per_{\dg}A,\qquad \C:=\B/^{\mathbb L}\A.
\]
Keller's localization theorem \cite{Kel99} gives an exact triangle
\[
\HH(\A)\longrightarrow\HH(\B)\longrightarrow\HH(\C)\dashrightarrow
\]
in $\D(k)$. Since $\HH_{<0}(\B)\cong\HH_{<0}(A)=0$, this yields an isomorphism
\[\delta\colon D\!\HH_{-d-1}(\A)\xrightarrow[\cong]{}D\!\HH_{-d}(\C)\]
for $d\ge1$.

If $[x]\in D\!\HH_{-d-1}(\A)$ is nondegenerate, then so is
$\delta([x])$ \cite{HL25,KL23}. This lifts Amiot's construction
of a Serre functor on Verdier quotients \cite[Section~1]{Ami09}.
In this case $\proj_{\dg}A$ is $d$-stable by
Theorem~\ref{thm:aus-cor-for-loc-fin}, so $\C=\C_{d,\dg}(A)$ and
the image of $A$ is $d$-cluster tilting by
Theorem~\ref{thm:fundamental-domain}.

Keller--Liu prove nondegeneracy detection for the model obtained
from a pseudo-compact enhancement
\cite[Proposition~6.4.1]{KL23}. The proof below works in ordinary dg
categories and combines the dg enhancement of Amiot's construction
\cite[Section~4]{HL25} with the cohomological gap from cluster tilting.
For a locally finite $d$-stable dg category $\M$, the localization
$\D^b_{\fp,\dg}(\M)\subseteq\per_{\dg}\M$ gives the same connecting
isomorphism because $\HH_{<0}(\per_{\dg}\M)\cong\HH_{<0}(\M)=0$.

\begin{Thm}\label{thm:delta-detects-nondegeneracy}
Let $\M$ be a locally finite $d$-stable dg category. Consider the isomorphism
\[\delta\colon D\!\HH_{-d-1}(\D^b_{\fp,\dg}(\M))\xrightarrow[\cong]{}D\!\HH_{-d}(\C_{d,\dg}(\M)).\]
For $[x]\in D\!\HH_{-d-1}(\D^b_{\fp,\dg}(\M))$, it is nondegenerate if and only if $\delta([x])$ is nondegenerate.
\end{Thm}
\begin{proof}
First, we prove that $\delta$ preserves nondegeneracy. Set
\[
\A=\D^b_{\fp,\dg}(\M),\qquad
\B=\per_{\dg}\M,\qquad
\C=\B/^{\mathbb L}\A=\C_{d,\dg}(\M).
\]
Since $\M$ is locally finite, the evaluation map $\B\to DD\B$ is an
isomorphism in $\D(\B^e)$. Thus by \cite[Theorem~4.3.1(b)]{HL25}, it
suffices to show that the canonical map $\B\to\RHom_\A(\B,\B)$ is an
isomorphism in $\D(\B^e)$. By the dual of \cite[Lemma~4.3.4]{HL25},
we have only to prove that for any $X,Y\in H^0\B$, $Y$ has a local
$H^0\A$-envelope relative to $X$ in the sense of
\cite[Definition~1.2]{Ami09}. Since $X\in\B=\per_{\dg}\M$, there exists
$a\in\mathbb{Z}$ such that $\D(\M)(X,Z)=0$ holds for any
$Z\in\D^{\le a}(\M)$. Consider the truncation
\[\tau^{\le a}Y\to Y\to\tau^{>a}Y\dashrightarrow.\]
Since $Y\in\D^-_{\fp}(\M)$, we have $\tau^{>a}Y\in H^0\A$.
Since $\D(\M)(X,\tau^{\le a}Y)=0$, the morphism
$Y\to\tau^{>a}Y$ gives a local $H^0\A$-envelope relative to $X$.

For the converse, let $[x]\in D\!\HH_{-d-1}(\A)$ and let
\[
 f_x\colon\A[d+1]\longrightarrow D\A\quad\text{and}\quad
 g_x\colon\C[d]\longrightarrow D\C
\]
be the corresponding morphisms to $[x]$ and $\delta([x])$ respectively.

Put $\K=\B\otimes_{\A}^{\mathbb L}\B$.
The adjunctions of \cite[Lemma~4.2.1]{HL25} associate to $f_x$ maps
\[
\ell\colon\K[d+1]\longrightarrow D\B,\qquad
m\colon\B[d+1]\longrightarrow D\K.
\]
By \cite[Proposition~4.2.2 and Theorem~4.3.1(a)]{HL25}, they fit into a morphism of triangles
in $\D(\B^e)$
\[
\begin{tikzcd}[column sep=large]
 \K[d+1] \arrow[r] \arrow[d,"{\ell}"']
   & \B[d+1] \arrow[r] \arrow[d,"{m}"]
   & \C[d+1] \arrow[r] \arrow[d,"{\overline{g_x[1]}}"]
   & \K[d+2] \arrow[d,"{\ell[1]}"] \\
 D\B \arrow[r]
   & D\K \arrow[r]
   & D\C[1] \arrow[r]
   & D\B[1].
\end{tikzcd}
\]
Here $\overline{g_x[1]}$ is the restriction of $g_x[1]\in\D(\C^e)(\C[d+1],D\C[1])$ along
the quotient functor.

Evaluating the localization
triangle $\K\to\B\to\C\dashrightarrow$ at $(M,N)$ for $M,N\in\M$ gives
\[
 \K(M,N)\longrightarrow \M(M,N)\longrightarrow\C(M,N)\dashrightarrow.
\]
The second morphism is an isomorphism on $H^{\le0}$ by
Theorem~\ref{thm:fundamental-domain}. Moreover, $\M$ is connective and
\[
 H^i\C(M,N)=0\qquad(1\le i\le d-1),
\]
because $H^0\M$ is $d$-rigid in $H^0\C$. Hence
\begin{equation}\label{eq:localization-gap}
 \K(M,N)\in\D^{\ge d+1}(k).
\end{equation}

Assume that $\delta([x])$ is nondegenerate, so that
$\overline{g_x[1]}$ is an isomorphism. By \eqref{eq:localization-gap}
and the connectivity of $\M$, at $(M,N)$ the first terms of the two
triangles belong to $\D^{\ge0}(k)$ and the middle terms belong to
$\D^{\le-d-1}(k)$. For $j\ge0$, the middle terms have zero
cohomology in degrees $j-1$ and $j$, since $d\ge1$. Thus the horizontal
connecting maps in
\[
\begin{tikzcd}[column sep=large]
 H^{j-1}(\C[d+1])(M,N) \arrow[r,"\sim"]
   \arrow[d,"{H^{j-1}\overline{g_x[1]}}"']
   & H^j(\K[d+1])(M,N) \arrow[d,"{H^j\ell}"] \\
 H^{j-1}(D\C[1])(M,N) \arrow[r,"\sim"]
   & H^j(D\B)(M,N)
\end{tikzcd}
\]
are isomorphisms. The left vertical map is an isomorphism by assumption,
so $H^j\ell(M,N)$ is an isomorphism. For $j<0$, the source and target
of $H^j\ell(M,N)$ both vanish. Hence $\ell(M,N)$ is an isomorphism for all
$M,N\in\M$. Evaluation of dg $\B$-bimodules is exact in each variable
on $H^0\B=\thick(H^0\M)$, so $\ell$ is an isomorphism in $\D(\B^e)$.

Finally,
under the adjunction of \cite[Lemma~4.2.1]{HL25}, $f_x$ is the
restriction of $\ell$ to $\A^e$, using
\[
 \K|_{\A^e}\simeq\A,\qquad (D\B)|_{\A^e}\simeq D\A.
\]
Therefore $f_x$ is an isomorphism. This means that $[x]$ is nondegenerate.
\end{proof}

As an immediate corollary, we obtain a bijection between the sets of right Calabi--Yau structures. Note that since $\HC_{<0}(\B)\cong\HC_{<0}(\M)=0$, we also have a connecting isomorphism
\[\delta\colon D\!\HC_{-d-1}(\D^b_{\fp,\dg}(\M))\xrightarrow[\cong]{}D\!\HC_{-d}(\C_{d,\dg}(\M)).\]

\begin{Cor}\label{cor:bij-CYstr}
Let $\M$ be a locally finite $d$-stable dg category. Then the isomorphism
\[\delta\colon D\!\HC_{-d-1}(\D^b_{\fp,\dg}(\M))\xrightarrow[\cong]{}D\!\HC_{-d}(\C_{d,\dg}(\M))\]
restricts to a bijection
\[\{\text{right $(d+1)$-Calabi--Yau structures on $\D^b_{\fp,\dg}(\M)$}\}\xrightarrow[\cong]{}\{\text{right $d$-Calabi--Yau structures on $\C_{d,\dg}(\M)$}\}.\]
\end{Cor}
\begin{proof}
This follows from Theorem~\ref{thm:delta-detects-nondegeneracy} and the following commutative diagram:
\[
\begin{tikzcd}[column sep=large]
D\!\HC_{-d-1}(\D^b_{\fp,\dg}(\M)) \arrow[r,"{\cong}"] \arrow[d] & D\!\HC_{-d}(\C_{d,\dg}(\M)) \arrow[d] \\
D\!\HH_{-d-1}(\D^b_{\fp,\dg}(\M)) \arrow[r,"{\cong}"'] & D\!\HH_{-d}(\C_{d,\dg}(\M)).
\end{tikzcd}
\]
\end{proof}

\subsection{Realization with Calabi--Yau structures}

The correspondence combines with cluster Morita theory over an
arbitrary field.

\begin{Thm}[Morita realization with Calabi--Yau structures]\label{thm:many-Amiot-conj}
Let $\C$ be a locally finite pretriangulated dg category with
$H^0\C$ idempotent complete, having a $d$-cluster tilting subcategory $\M\subseteq\tau^{\le0}\C$.
\begin{enumerate}
\item $\M$ is $d$-stable, and we have a quasi-equivalence
\[\C\simeq\C_{d,\dg}(\M).\]
\item If $\C$ is weak right $d$-Calabi--Yau, then $\D^b_{\fp,\dg}(\M)$ is weak right $(d+1)$-Calabi--Yau.
\item If $\C$ has a right $d$-Calabi--Yau structure, then $\D^b_{\fp,\dg}(\M)$ has a right $(d+1)$-Calabi--Yau structure.
\end{enumerate}
\end{Thm}
\begin{proof}
(1) follows from Proposition \ref{prop:cluster-tilting-d-stable} and Theorem \ref{thm:morita-cluster-dgcat}. (2) follows from Theorem \ref{thm:delta-detects-nondegeneracy}. (3) follows from Corollary \ref{cor:bij-CYstr}.
\end{proof}

\begin{Cor}[The one-object case]\label{thm:Amiot-conj}
Let $\C$ be a locally finite pretriangulated dg category with
$H^0\C$ idempotent complete, having a $d$-cluster tilting object $M\in\C$. Put $A:=\tau^{\le0}\End_\C(M)$.
\begin{enumerate}
\item $\proj_{\dg}A$ is $d$-stable, and we have a quasi-equivalence
\[\C\simeq\C_{d,\dg}(A).\]
\item If $\C$ is weak right $d$-Calabi--Yau, then $\pvd_{\dg}A$ is weak right $(d+1)$-Calabi--Yau.
\item If $\C$ has a right $d$-Calabi--Yau structure, then $\pvd_{\dg}A$ has a right $(d+1)$-Calabi--Yau structure.
\end{enumerate}
\end{Cor}
\begin{proof}
Apply Theorem~\ref{thm:many-Amiot-conj} to
$\M=\proj_{\dg}A\simeq\add M$.
Derived Morita theory identifies
\[
\D^b_{\fp,\dg}(\M)\simeq\pvd_{\dg}A.\qedhere
\]
\end{proof}

Amiot asked whether every $2$-Calabi--Yau tilted algebra is Jacobian
\cite[Question~2.20(1)]{Ami11}. The categorical realization problem
asks, over an algebraically closed field of characteristic zero,
whether a Hom-finite, idempotent-complete algebraic $2$-Calabi--Yau
triangulated category with a chosen cluster tilting object is
triangle equivalent to $\C_{Q,W}$ for a Jacobi-finite quiver with
potential, with the chosen object corresponding to the canonical one.
The following is the realization direction of Keller--Liu's theorem.
For deformed dg preprojective algebras and their cluster categories,
see \cite[Sections~4.5--4.6]{KL23}.

\begin{Thm}[{\cite[Theorem~6.2.1]{KL23}}]
Let $k$ be a field of characteristic zero and let $d\geq1$.
Let $\E$ be a small Frobenius exact pseudo-compact $k$-category with
full subcategory $\P$ of projective-injective objects. Suppose that
its stable category $\T=\underline{\E}$ is Hom-finite and idempotent
complete and contains a $d$-cluster tilting object. If the dg enhancement
\[
\D^b_{\dg}(\E)/^{\mathbb L}\D^b_{\dg}(\P)
\]
carries a right $d$-Calabi--Yau structure, then there exist a
$(d+1)$-dimensional deformed dg preprojective algebra
$\Pi=\Pi_{d+1}(l,V_c,\eta,w)$ and a triangle equivalence
\[
\C_\Pi\xrightarrow{\sim}\T.
\]
Here $\C_\Pi$ is the idempotent completion of the Verdier quotient
$\per^{\mathrm{pc}}\Pi/\pvd^{\mathrm{pc}}\Pi$, formed using
pseudo-compact dg modules.
\end{Thm}

In the proof of \cite[Theorem~6.2.1]{KL23}, the existence of a pseudo-compact enhancement is used to apply the structure theorem of Calabi--Yau dg algebras \cite{VdB15}.

To obtain a concrete dg algebra model, we use the following
existence result, proved directly in
Proposition~\ref{prop:direct-minimal-tensor-model}.
Every locally finite connective dg algebra $A$ with basic separable
quotient $R=H^0A/\rad H^0A$ is quasi-isomorphic to
$(\widehat T_R(V),\partial)$, where $V$ is nonpositive, finite
dimensional in each degree, and
$\partial(V)\subseteq\widehat T_R(V)_{\geq2}$.
For $n\geq1$, its generators in degree $1-n$ form the bimodule
$H^{-n}(R\otimes_A^{\mathbb L}R)$.

The comparison with ordinary dg modules is recorded separately.

\begin{Lem}[Forgetting the topology]\label{lem:pc-ordinary-comparison}
Let $\Gamma=(\widehat T_R(V),\partial)$, where $R$ is finite
dimensional, basic, semisimple, and separable, $V$ is finite
dimensional and nonpositive, and
$\partial(V)\subseteq\Gamma_{\geq2}$.
If $H^0\Gamma$ is finite dimensional, then $\Gamma$ is smooth in
the pseudo-compact sense, and forgetting topology induces
quasi-equivalences
\[
\per^{\mathrm{pc}}_{\dg}\Gamma\simeq\per_{\dg}\Gamma,
\qquad
\pvd^{\mathrm{pc}}_{\dg}\Gamma\simeq\pvd_{\dg}\Gamma.
\]
\end{Lem}
\begin{proof}
The resolution in Lemma~\ref{lem:species-cofibrant-resolution}
is a finite cofibrant resolution of $R$ by pseudo-compact
projective modules. Thus
$\pvd^{\mathrm{pc}}\Gamma=\thick R\subseteq
\per^{\mathrm{pc}}\Gamma$, by the bounded Postnikov filtration.
Smoothness follows from \cite[Corollary~4.2.3]{KL23}.
The forgetful functor is a quasi-equivalence on perfect dg
categories, since it identifies the dg endomorphisms of the
generator $\Gamma$; see \cite[Lemma~4.2.5]{KL23}.
Its restriction to perfectly valued objects is fully faithful.
Its image is thick and contains $R$, hence equals
$\pvd\Gamma=\thick R$, proving essential surjectivity.
\end{proof}

We can now remove the pseudo-compact enhancement hypothesis
from Keller--Liu's realization theorem.

\begin{Thm}[Realization without a pseudo-compact enhancement]
\label{thm:Amiot-without-pseudocompact}
Let $k$ be a field of characteristic zero and let $d\geq1$.
Let $\C$ be a locally finite pretriangulated dg $k$-category such
that $H^0\C$ is idempotent complete, and suppose that $\C$ carries
a right $d$-Calabi--Yau structure. For any basic $d$-cluster tilting
object $M$, put $R=\End_{H^0\C}(M)/\rad\End_{H^0\C}(M)$. There is
a $(d+1)$-dimensional deformed dg preprojective algebra
\[
\Pi=\Pi_{d+1}(R,V_c,\eta,w)
\]
with $H^0\Pi$ finite dimensional, quasi-isomorphic to
$\tau^{\leq0}\End_\C(M)$ as an ordinary dg algebra. There is a
quasi-equivalence
\[
\C_{d,\dg}(\Pi)\simeq\C
\]
sending the image of the free module to $M$.
\end{Thm}
\begin{proof}
Put $A=\tau^{\leq0}\End_\C(M)$, so that $R=H^0A/\rad H^0A$.
By Corollary~\ref{thm:Amiot-conj}, $A$ is $d$-stable,
$\C\simeq\C_{d,\dg}(A)$, and $\pvd_{\dg}A$ carries a right
$(d+1)$-Calabi--Yau structure. In particular,
$\gl A\leq d+1$ by Theorem~\ref{thm:auslander-correspondence}.

Since $k$ has characteristic zero, $R$ is separable over $k$.
Proposition~\ref{prop:direct-minimal-tensor-model} gives a
quasi-isomorphic ordinary dg algebra
\[
\Gamma=(\widehat T_R(V),\partial)
\]
satisfying {\rm (P1)--(P4)}. Quasi-isomorphisms preserve global
dimension, so Corollary~\ref{cor:species-finiteness}(2) gives
$V^{<-d}=0$. Thus $V$ is finite dimensional and concentrated in
degrees $-d,\ldots,0$.

Lemma~\ref{lem:pc-ordinary-comparison} shows that $\Gamma$ is
smooth in the pseudo-compact sense and gives a quasi-equivalence
\[
\pvd^{\mathrm{pc}}_{\dg}\Gamma
\simeq\pvd_{\dg}\Gamma\simeq\pvd_{\dg}A.
\]
Transport the right $(d+1)$-Calabi--Yau structure along this
equivalence. Under the Morita equivalence between
$\pvd^{\mathrm{pc}}_{\dg}\Gamma$ and the Koszul dual of $\Gamma$, the
left--right correspondence
\cite[Proposition~4.4.1]{KL23} gives a left
$(d+1)$-Calabi--Yau structure on $\Gamma$.

Now apply Van den Bergh's structure theorem
\cite[Theorem~9.4]{VdB15}, with the case $d=1$ (dimension two) treated as in
\cite[Section~6.4]{KL23}. It yields a $(d+1)$-dimensional
deformed dg preprojective algebra
$\Pi=\Pi_{d+1}(R,V_c,\eta,w)$ weakly equivalent to $\Gamma$.
These weak equivalences are quasi-isomorphisms of underlying
ordinary dg algebras. They identify perfect and perfectly valued
dg categories, and hence
\[
\C\simeq\C_{d,\dg}(A)\simeq\C_{d,\dg}(\Gamma)
\simeq\C_{d,\dg}(\Pi).
\]
The equivalence sends $\Pi$ to $M$, and
$H^0\Pi\cong H^0A=\End_{H^0\C}(M)$ is finite dimensional.
\end{proof}

\begin{Cor}[Realization by a quiver with potential]
\label{cor:Amiot-quiver-potential}
Let $k$ be an algebraically closed field of characteristic zero.
Let $\C$ be a locally finite pretriangulated dg $k$-category with
$H^0\C$ idempotent complete, a $2$-cluster tilting object, and a
right $2$-Calabi--Yau structure. Then there is a Jacobi-finite
quiver with potential $(Q,W)$ and a quasi-equivalence
\[
\C\simeq\C_{2,\dg}(\Gamma(Q,W)).
\]
In particular, $H^0\C\simeq\C_{Q,W}$.
\end{Cor}
\begin{proof}
For $d=2$ and $R=k^{Q_0}$, the deformed dg preprojective algebra
in Theorem~\ref{thm:Amiot-without-pseudocompact} is the complete
Ginzburg dg algebra of a finite quiver with a formal potential
\cite[Section~9.3]{VdB15}. Its zeroth cohomology is the completed
Jacobian algebra and is finite dimensional by the theorem.
Thus $(Q,W)$ is Jacobi-finite.
\end{proof}

\section{Koszul duality and the Yoneda--Frobenius criterion}\label{sec:loc-fin}

We specialize the Auslander correspondence to locally finite connective
dg algebras. Koszul duality first expresses higher stability as shifted
self-injectivity of the Koszul dual. We then show that this condition
is detected by its ordinary cohomology algebra, namely the Yoneda
algebra. This is the link to the quadratic recognition theorem in
Section~\ref{sec:quadratic-recognition}.

Throughout this section, let $k$ denote a field. All dg categories are assumed to be defined over $k$.

\subsection{Koszul duality}

We recall the form of Koszul duality used to pass from the
Auslander condition to a condition on the Yoneda algebra.  Following
\cite[Section~2.3]{Fus}, for a dg algebra $C$ we denote by $\pvd C$
the full subcategory of $\D(C)$ consisting of the objects whose total
cohomology is finite dimensional.

\begin{Def}\label{def:koszul-dual}
Let $A$ be a locally finite connective dg $k$-algebra and put
$S_A=H^0A/\rad H^0A$. Following \cite[Section~4]{Fus}, its
\emph{Koszul dual} is
\[
A^!:=\REnd_A(S_A).
\]
The associated contravariant Koszul duality functor is
\[
\Phi_A:=\RHom_A(-,S_A)\colon
\D(A)\longrightarrow\D((A^!)^{\op})^{\op}.
\]
\end{Def}

For a locally finite positive dg algebra $B$, let
$S_B=\sigma^{\leq0}B\in\pvd B$ be the distinguished module given by the
weight truncation of the right regular module. It is characterized up
to isomorphism in $\D(B)$ by $H^i(S_B)=0$ for $i\ne0$ and
$H^0(S_B)\cong H^0B$ as right $H^0B$-modules
\cite[Definition~3.3 and Remark~3.4]{Fus}.

Write $\D^-_{\mathrm{fd},\dg}(A)$ (respectively, $\D^+_{\mathrm{fd},\dg}(A^!)$) for the full dg subcategory of dg modules with bounded-above (respectively, bounded-below) cohomology that is finite dimensional in each degree. In \cite{Fus}, the Koszul duality functor $\Phi_A$ is shown to restrict to triangle equivalences.

\begin{Thm}[{\cite[Theorem~4.4(1)]{Fus}}]\label{thm:koszul-duality}
In the setting of Definition~\ref{def:koszul-dual}, the dg algebra
$A^!$ is locally finite and positive. Moreover,
$\Phi_A(A)\simeq S_{(A^!)^{\op}}$, and $\Phi_A$ induces the following
three contravariant triangle equivalences, where the upper and lower
rows are the restrictions of the middle row:
\[
\begin{tikzcd}[column sep=large,row sep=small]
\per_{\dg}A
  \arrow[rr,"\sim"] \arrow[d,hook]
&& (\pvd_{\dg}((A^!)^{\op}))^{\op}
  \arrow[d,hook] \\
\D^-_{\mathrm{fd},\dg}(A)
  \arrow[rr,"\sim"]
&& (\D^+_{\mathrm{fd},\dg}((A^!)^{\op}))^{\op} \\
\pvd_{\dg}A
  \arrow[rr,"\sim"] \arrow[u,hook]
&& (\per_{\dg}((A^!)^{\op}))^{\op}
  \arrow[u,hook].
\end{tikzcd}
\]
\end{Thm}

We next record how this duality interacts with passage to the opposite
algebra and with the object obtained by dualizing $DS_A$.

\begin{Lem}\label{lem:koszul-duality}
Let $A$ be a locally finite connective dg algebra and put $B=A^!$. Then
\[
\add DS_A=\add S_{A^{\op}},
\qquad
(A^{\op})^!\simeq B^{\op}.
\]
Assume that
$DS_A\in\per A^{\op}$ and put
\[
U_A:=\RHom_{A^{\op}}(DS_A,A)\in\per A.
\]
Then
\[
\Phi_A(U_A)\cong DB
\quad\text{in }\D(B^{\op}),
\]
where $DB$ is regarded as a right $B^{\op}$-module via the right
regular action of $B$.
\end{Lem}
\begin{proof}
The first two assertions follow from \cite[Remark~4.1]{Fus}. Applying tensor--Hom adjunction to
$A^{\op}$ gives
\[
\begin{split}
\Phi_A(U_A)
&=\RHom_A\bigl(\RHom_{A^{\op}}(DS_A,A),S_A\bigr)\\
&\cong D\bigl(\RHom_{A^{\op}}(DS_A,A)\otimes_{A}^{\mathbb L}DS_A\bigr)\\
&\cong D\REnd_{A^{\op}}(DS_A)
 \cong DB.\qedhere
\end{split}
\]
\end{proof}

The following lemma reads the global dimension of $A$ from
the cohomological amplitude of its Koszul dual.

\begin{Lem}\label{lem:global-dimension-koszul-dual}
Let $A$ be a locally finite connective dg algebra and assume that
$S_A\in\per A$. Then
\[
\gl A=\sup\{i\mid H^i(A^!)\neq0\}.
\]
\end{Lem}
\begin{proof}
Put $C=H^0A$ and $Q=S_A\otimes_A^{\mathbb L}C\in\per C$.
Choose a minimal bounded complex of finitely generated projective
$C$-modules representing $Q$, concentrated in nonpositive degrees.
Adjunction and minimality give
\[
\Ext_A^i(S_A,S_A)
\cong H^i\RHom_C(Q,S_A)
=\Hom_C(Q^{-i},S_A).
\]
Every nonzero finite projective $C$-module maps nontrivially to $S_A$.
Thus the largest $i$ with nonzero $\Ext_A^i(S_A,S_A)$ is
$\pd_CQ=\pd_A S_A$, where the last equality is
\eqref{eq:base-change-fd-pd}.
Since every finite $C$-module has a composition series with factors
in $\add S_A$, Corollary~\ref{cor:weak-global-dimension-fp} and
Definition~\ref{def:global-dimension} give
$\pd_A S_A=\gl A$.
\end{proof}

\subsection{The homological and Koszul-dual criteria}

First, we interpret conditions in Theorem \ref{thm:auslander-correspondence} in the case of locally finite connective dg algebras. In this setting, the homological conditions can be tested on the semisimple module
$S=H^0A/\rad H^0A$, and condition (3) admits a Koszul-dual
reformulation.

\begin{Thm}\label{thm:aus-cor-for-loc-fin}
Let $A$ be a locally finite connective dg
algebra. Set $S:=H^0A/\operatorname{rad}H^0A$, viewed as a right or left $H^0A$-module. Then the following
conditions are equivalent.
\begin{enumerate}
\item $\proj_{\dg}A$ is $d$-stable.
\item $\gl A\le d+1$, and
\[
\Ext_A^i(S,A)=0\quad(i\ne d+1),
\qquad
\Ext_A^{d+1}(S,A)\in\add S\subseteq\per A^{\op}.
\]
\item In $\D(A^!)$, we have
\[\add DA^!=\add A^![d+1].\]
\end{enumerate}
\end{Thm}
\begin{proof}
Put $B=A^!$ and $T=\RHom_A(S,A)\in\D(A^{\op})$.
Local finiteness of $A$
implies that $\proj_{\dg}A$ is coherent. Since $H^0A$ is finite
dimensional, every finitely generated $H^0A$-module has a finite
composition series with simple factors in $\add S$.
Corollary~\ref{cor:weak-global-dimension-fp} and
Proposition~\ref{prop:fp-flat-projective}, together with
Definition~\ref{def:global-dimension}, therefore give
\[
\wgl(\proj_{\dg}A)=\fd_A S=\pd_A S=\gl A.
\]
Thus Theorem~\ref{thm:auslander-correspondence} reduces
$(1)\Longleftrightarrow(2)$ to the following observation. Under
$\gl A\le d+1$, the standard duality restricts to
\[
\mod H^0A\xrightarrow{\sim}\mod H^0A^{\op}[-d-1]
\]
if and only if
$T\in\add S[-d-1]$. Indeed, the forward implication follows by applying
the restricted duality to $S$. Conversely, assume this membership and
choose representatives $S_1,\ldots,S_r$ of the simple
$H^0A$-modules. The modules $\RHom_A(S_i,A)[d+1]$ are semisimple, and
full faithfulness shows that they are nonzero, simple, and pairwise
non-isomorphic. They form a complete collection of simple
$H^0A^{\op}$-modules because the two algebras have the same number of
simple modules. Taking extension closures gives the restricted
duality. This argument also shows that
\begin{equation}\label{eq:add-gorenstein-dual}
\add T=\add S[-d-1]
\quad\text{in }\D(A^{\op}).
\end{equation}
The membership of $T$ above is precisely the Ext condition in (2), so
$(1)\Longleftrightarrow(2)$ follows.

We prove $(2)\Longleftrightarrow(3)$. On the
$A^{\op}$-side, $DS$ and $S$ are semisimple generators, so
$\add DS=\add S$. Put
\[
U:=\RHom_{A^{\op}}(DS,A)\in\D(A).
\]
Whenever $S\in\per A$ and $DS\in\per A^{\op}$, the standard duality identifies
\eqref{eq:add-gorenstein-dual} with
\begin{equation}\label{eq:add-gorenstein-dual-opposite}
\add U=\add S[-d-1]
\quad\text{in }\D(A).
\end{equation}
Indeed, perfectness of $S$ gives the biduality isomorphism
$\RHom_{A^{\op}}(T,A)\cong S$. Applying $\RHom_{A^{\op}}(-,A)$ to
$\add T=\add DS[-d-1]$ therefore gives $\add S=\add U[d+1]$.

Assume (2). Then $S\in\per A$ and $DS\in\per A^{\op}$, and
Lemma~\ref{lem:koszul-duality} gives
\[
\Phi_A(S)\cong B,
\qquad
\Phi_A(U)\cong DB
\quad\text{in }\D(B^{\op}).
\]
Applying $\Phi_A$ to
\eqref{eq:add-gorenstein-dual-opposite} yields
\begin{equation}\label{eq:frobenius-opposite}
\add DB=\add B[d+1]
\quad\text{in }\D(B^{\op}).
\end{equation}
Moreover, Theorem~\ref{thm:koszul-duality} identifies $\per A$
with $(\pvd B^{\op})^{\op}$. Since $S\in\per A$ and
$\Phi_A(S)\cong B$, it follows that $B$ is proper.
Applying $D$ to \eqref{eq:frobenius-opposite} and shifting by $d+1$
gives $\add DB=\add B[d+1]$ in $\D(B)$, which is (3).

Conversely, assume (3). The dg algebra $B$ is positive and $DB$ is
nonpositive. Since $B[d+1]\in\add DB$, we have $H^{>d+1}B=0$, so $B$ is
proper. Applying $D$ to (3) and shifting by $d+1$ gives
\eqref{eq:frobenius-opposite}. By the equivalences in
Theorem~\ref{thm:koszul-duality}, properness of
$\Phi_A(S)\cong B$ implies $S\in\per A$. Since
$(A^{\op})^!\simeq B^{\op}$ by Lemma~\ref{lem:koszul-duality},
the same theorem applied to $A^{\op}$ gives
$DS\in\per A^{\op}$. Hence
Lemma~\ref{lem:global-dimension-koszul-dual} gives
\[
\gl A=\sup\{i\mid H^iB\neq0\}\le d+1.
\]
By Theorem~\ref{thm:koszul-duality}, the functor $\Phi_A$ is fully
faithful on $\per A$. Pulling \eqref{eq:frobenius-opposite} back and
using Lemma~\ref{lem:koszul-duality} gives
\eqref{eq:add-gorenstein-dual-opposite}. Applying
$\RHom_A(-,A)$ then gives
\[
\add T=\add DS[-d-1]=\add S[-d-1].
\]
Thus \eqref{eq:add-gorenstein-dual} holds, and hence so does (2).
\end{proof}

\subsection{Detection on the Yoneda algebra}

For a graded algebra $E$, write $\Mod^{\mathbb{Z}}E$ for the category
of graded right modules with degree-zero morphisms, and use the
cohomological shift $M[N]^i=M^{i+N}$. A finite-dimensional nonnegatively
graded algebra $E$ is \emph{graded Frobenius of parameter $N$} if
\[
DE\cong E[N]\qquad\text{in }\Mod^{\mathbb{Z}}E.
\]
Equivalently, there is a functional $\lambda\colon E^N\to k$,
extended by zero in other degrees, for which
$(a,b)\mapsto\lambda(ab)$ is nondegenerate. No symmetry of this
pairing is required.

\begin{Prop}[Cohomological detection]\label{prop:cohomological-frobenius}
Let $B$ be a locally finite positive dg algebra, put $E=H^*B$, and
let $N\geq0$. Then
\[
\add DB=\add B[N]\text{ in }\D(B)
\quad\Longleftrightarrow\quad
\add DE=\add E[N]\text{ in }\Mod^{\mathbb{Z}}E.
\]
Under either condition $B$ is proper. If $E^0$ is basic, both
conditions are equivalent to $DB\cong B[N]$ in $\D(B)$ and to $E$
being graded Frobenius of parameter $N$.
\end{Prop}
\begin{proof}
Exactness of $D$ over $k$ gives $H^*(DB)\cong DE$, so the dg condition
implies the graded one. Since $E[N]\in\add DE$ and $DE$ is nonpositive,
the graded condition forces $E^{>N}=0$ and hence properness.

For $X\in\D(B)$, the cohomology map
\[
\Hom_{\D(B)}(B[N],X)\longrightarrow
\Hom_{\Mod^{\mathbb{Z}}E}(E[N],H^*X)
\]
identifies with the identity on $H^{-N}X$ by evaluation at the
shifted unit \cite[Section~3.2, equation~(3)]{Kel06}. This extends to sources
$P\in\add B[N]$ and gives an equivalence $\add B[N]\simeq\add E[N]$
by splitting idempotents. When the graded condition holds, choose
$P\in\add B[N]$ with $H^*P\cong DE$. This isomorphism lifts to
$P\cong DB$, so the equivalence gives $\add DB=\add B[N]$.

If $E^0$ is basic, the degree-zero endomorphism algebras of $E[N]$
and $DE$ are both isomorphic to $E^0$, the latter by graded duality.
Each module thus has pairwise nonisomorphic indecomposable summands.
Equality of their additive closures gives $E[N]\cong DE$ by
Krull--Schmidt, and the cohomology map above lifts this to $B[N]\cong DB$.
\end{proof}

\begin{Thm}[Yoneda--Frobenius criterion]\label{thm:yoneda-frobenius}
Let $A$ be a locally finite connective dg algebra, and put
\[
S=H^0A/\rad H^0A,\qquad E_A=\Ext_A^*(S,S).
\]
Then $\proj_{\dg}A$ is $d$-stable if and only if $E_A$ is finite
dimensional and
\[
\add DE_A=\add E_A[d+1]
\qquad\text{in }\Mod^{\mathbb{Z}}E_A.
\]
If $S$ is basic, this is equivalent to $E_A$ being graded
Frobenius of parameter $d+1$.
\end{Thm}
\begin{proof}
Theorem~\ref{thm:aus-cor-for-loc-fin} identifies $d$-stability with the
dg condition in Proposition~\ref{prop:cohomological-frobenius} for
$B=A^!$ and $N=d+1$. The proposition applies by
Theorem~\ref{thm:koszul-duality} and gives the claim since $H^0B=S$.
\end{proof}

\begin{Rem}
This criterion uses the ordinary Yoneda multiplication, not a choice
of higher multiplications on $E_A$. It detects higher stability, not
the quasi-isomorphism class of $A$. The additive-closure formulation
also avoids a basicness assumption: for a nonbasic semisimple $S$,
shifted self-injectivity should not be replaced without justification
by an isomorphism $DE_A\cong E_A[d+1]$.

Theorem~\ref{thm:aus-cor-for-loc-fin}(2) is a Gorenstein condition,
whereas (3) is a one-sided shifted self-injectivity condition. A
bimodule Calabi--Yau structure requires compatibility with both
actions, which the Frobenius criterion does not assert. The distinction
is analogous to that between self-injective and symmetric algebras.
\end{Rem}

\subsection{Koszul duality and singularity categories}

Koszul duality identifies locally finite connective dg algebras $A$ for which $\proj_{\dg}A$ is $d$-stable with proper pvd-finite positive dg algebras satisfying the shifted self-injectivity condition below. Recall from \cite[Theorem~4.5(3)]{Fus} that a locally finite positive dg algebra $B$ is {\it pvd-finite} if and only if $\pvd B$ is Hom-finite.

\begin{Cor}\label{cor:koszul-dual-bijection}
Taking Koszul duals gives a bijection between the quasi-isomorphism classes of the following dg algebras.
\begin{enumerate}
\item Locally finite connective dg algebras $A$ such that $\proj_{\dg}A$ is $d$-stable.
\item Proper pvd-finite positive dg algebras $B$ such that $\add DB=\add B[d+1]$ holds in $\D(B)$.
\end{enumerate}
\end{Cor}
\begin{proof}
This is a restriction of the bijection \cite[Corollary~4.17]{Fus} to $d$-stable ones and $(d+1)$-shifted self-injective ones.
\end{proof}

Next, using Koszul duality, we interpret the cluster dg categories as singularity categories. Here, for a proper dg algebra $B$, set
\[
\sg_{\dg}B:=\pvd_{\dg}B/^\mathbb{L}\per_{\dg}B,\qquad
\sg B:=H^0(\sg_{\dg}B).
\]

\begin{Cor}\label{cor:koszul-singularity}
Let $A$ be a locally finite connective dg algebra such that $\proj_{\dg}A$ is $d$-stable. Put $B:=A^!$. Then $B$ is proper and we have a quasi-equivalence
\[\C_{d,\dg}(A)\simeq\sg_{\dg}B.\]
\end{Cor}
\begin{proof}
By Theorem \ref{thm:koszul-duality}, we have quasi-equivalences
\[
\begin{tikzcd}[column sep=large,row sep=small]
\per_{\dg}A
  \arrow[rr,"\sim"]
&& (\pvd_{\dg}B^{\op})^{\op} \\
\pvd_{\dg}A
  \arrow[rr,"\sim"] \arrow[u,hook]
&& (\per_{\dg}B^{\op})^{\op}
  \arrow[u,hook].
\end{tikzcd}
\]
Since $\add DB=\add B[d+1]$, the duality $D\colon\pvd_{\dg}B^{\op}\xrightarrow[\simeq]{}(\pvd_{\dg}B)^{\op}$ restricts to
\[
\begin{tikzcd}[column sep=large,row sep=small]
\pvd_{\dg}B^{\op}
  \arrow[rr,"\sim"]
&& (\pvd_{\dg}B)^{\op} \\
\per_{\dg}B^{\op}
  \arrow[rr,"\sim"] \arrow[u,hook]
&& (\per_{\dg}B)^{\op}
  \arrow[u,hook]
\end{tikzcd}
\]
Combining these, we have quasi-equivalences
\[
\begin{tikzcd}[column sep=large,row sep=small]
\per_{\dg}A
  \arrow[rr,"\sim"]
&& \pvd_{\dg}B \\
\pvd_{\dg}A
  \arrow[rr,"\sim"] \arrow[u,hook]
&& \per_{\dg}B
  \arrow[u,hook].
\end{tikzcd}
\]
Therefore by taking dg quotients, we obtain a quasi-equivalence
\[\C_{d,\dg}(A)\simeq\sg_{\dg}B.\qedhere\]
\end{proof}

We end this section by interpreting Theorem \ref{thm:fundamental-domain} and Corollary \ref{cor:bij-CYstr} purely in terms of positive dg algebras.

\begin{Thm}
Let $B$ be a proper pvd-finite positive dg algebra such that $\add DB=\add B[d+1]$ holds in $\D(B)$.
\begin{enumerate}
\item Let $\pi\colon\pvd B\to\sg B$ be the quotient functor. Then
$\pi(S_B)$ is a $d$-cluster tilting object of $\sg B$.
\item There is a bijection
\[\{\text{right $(d+1)$-Calabi--Yau structures on $B$}\}\xrightarrow[\sim]{}\{\text{right $d$-Calabi--Yau structures on $\sg_{\dg}B$}\}.\]
\end{enumerate}
\end{Thm}
\begin{proof}
By Corollary~\ref{cor:koszul-dual-bijection}, choose a locally finite
connective dg algebra $A$ such that $\proj_{\dg}A$ is $d$-stable and
$A^!\simeq B$. The proof of Corollary~\ref{cor:koszul-singularity}
constructs an equivalence $F=D\Phi_A\colon\per_{\dg}A\simeq\pvd_{\dg}B$
which restricts to $\pvd_{\dg}A\simeq\per_{\dg}B$ and induces
$\C_{d,\dg}(A)\simeq\sg_{\dg}B$.

By Theorem~\ref{thm:koszul-duality}, $\Phi_A(A)$ corresponds to
$S_{B^{\op}}$ under the equivalence induced by $A^!\simeq B$, since
the characterization by cohomology is preserved by this equivalence.
Thus $F(A)\simeq D S_{B^{\op}}\simeq S_B$, where the second isomorphism
follows from \cite[Remark~4.1]{Fus}.
Theorem~\ref{thm:fundamental-domain} now gives (1).

For (2), local finiteness identifies
$\D^b_{\fp,\dg}(\proj_{\dg}A)$ with $\pvd_{\dg}A$ up to Morita
equivalence. Under the restriction of $F$, this is Morita equivalent
to $\per_{\dg}B$, and hence to $B$. Morita invariance of cyclic
homology and of nondegeneracy therefore identifies its right
$(d+1)$-Calabi--Yau structures with those on $B$.
Apply Corollary~\ref{cor:bij-CYstr} and transport the structures on
the quotient along $\C_{d,\dg}(A)\simeq\sg_{\dg}B$.
\end{proof}

\section{Quadratic recognition and completed tensor models}
\label{sec:quadratic-recognition}

Throughout this section, $k$ is a field. We call a locally finite
connective dg algebra $A$ \emph{$d$-stable} if $\proj_{\dg}A$ is
$d$-stable. For a minimal completed tensor model, the Yoneda--Frobenius
criterion gives conditions on the generators and the quadratic
differential. These conditions detect $d$-stability when $H^0A$ is
finite dimensional.
For $d=1$, a Dynkin criterion detects finite dimensionality of $H^0A$,
and the quadratic part of the differential is identified with the
mesh tensor of the Auslander--Reiten species. For $d=2$, the normal
form uses a twisted hyperpotential. Both give concrete realizations
of the cluster Morita theorem.

We work with a degreewise completed tensor dg algebra
\[
A=(\widehat T_R(V),\partial),\qquad
A^j=\prod_{\ell\geq0}(V^{\otimes_R\ell})^j,
\qquad V^{\otimes_R0}=R,
\]
where $\partial$ is a continuous degree-one derivation with $\partial^2=0$ and
$\partial|_R=0$. Our standing hypotheses are
\begin{enumerate}
\item[(P1)] $R=\prod_{v\in I}D_v$ is a finite-dimensional basic
semisimple $k$-algebra;
\item[(P2)] $V^{>0}=0$;
\item[(P3)] $\dim_kV^{-j}<\infty$ for every $j\geq0$;
\item[(P4)] $\partial(V)\subseteq A_{\geq2}$.
\end{enumerate}
Here $A_{\geq r}$ denotes the tensor series of length at least $r$,
and $D_v$ are division $k$-algebras.
No separability is assumed unless explicitly stated.
The table records the additional hypotheses of the results below and
of the characteristic-zero realization theorem.
\begin{center}
\small
\begin{tabular}{@{}p{.39\textwidth}p{.43\textwidth}p{.11\textwidth}@{}}
\hline
Result & Extra hypotheses & $\ch k$\\
\hline
Quadratic recognition ($d\geq1$)\newline
and twisted Ginzburg normal form ($d=2$)
  & No separability or finiteness of $H^0$ & arbitrary\\[3pt]
Stability under {\rm (M1)--(M3)}
  & Finite-dimensional $H^0$ & arbitrary\\[3pt]
Minimal completed tensor models ($d\geq1$)\newline
and realizations for $d=1,2$
  & Basic separable semisimple quotient & arbitrary\\[3pt]
Deformed dg preprojective realization\newline
  {\rm (Theorem~\ref{thm:Amiot-without-pseudocompact})}
  & A right $d$-Calabi--Yau structure & $0$\\
\hline
\end{tabular}
\end{center}
Write $\partial_r\colon V\to V^{\otimes_R r}$
for the tensor-length-$r$ part of the differential, and set
\[
V^j(u,v)=e_vV^je_u,\qquad S_v=e_vR,\qquad P_v=e_vA.
\]
Throughout this section, dg modules and derived categories are
ordinary ones. For a finite
$R$-bimodule $W$, denote its right $R$-dual by
\[
W^\vee:=\Hom_R(W,R),\qquad
(rfs)(w)=r f(sw)\quad(r,s\in R).
\]
Homogeneous components of $V$ are regarded as ungraded bimodules
when taking this dual. By Corollary~\ref{cor:mesh-ext-dimension},
\[
E:=\Ext_A^*(R,R),\qquad E^0=R,\qquad
E^p=(V^{1-p})^\vee\quad(p\geq1),
\]
with these bimodule actions. Thus $V^{-i}$ corresponds to $E^{i+1}$,
and the lowest generators $V^{-d}$ correspond to $E^{d+1}$.
The Yoneda-product formula of
Proposition~\ref{prop:yoneda-quadratic} is
\[
(gf)(v)=(-1)^{p+1}\sum_j g\bigl(f(a_j)b_j\bigr),
\qquad
(\partial_2v)_{1-p,1-q}=\sum_j a_j\otimes b_j,
\]
for $f\in E^p$, $g\in E^q$, and $v\in V^{1-p-q}$.
Products are ordered as composition. The placement of the
coefficient $f(a_j)$ is essential over noncommutative $R$.

In the split case $R=k^{Q_0}$ and $V=kQ_1$, we recover the complete
dg path algebra $A=\widehat{kQ}$. Paths are composed from right to
left, so an arrow $\alpha\colon u\to v$ lies in $e_vVe_u$.
The right $R$-dual above becomes the $k$-dual, with the vertex
idempotents reversed. An arrow of degree $1-p$ determines its dual
class $\alpha^\vee\in\Ext_A^p(S_v,S_u)$.
Corollary~\ref{cor:quiver-yoneda-quadratic} identifies the coefficient
of $\gamma^\vee$ in $\beta^\vee\alpha^\vee$ with $(-1)^{p+1}$ times
the coefficient of $\alpha\beta$ in $\partial_2\gamma$.
The leftmost- and rightmost-arrow coefficients are denoted by
$\alpha\backslash q$ and $q/\alpha$, as in
Appendix~\ref{app:dg-species}.

\subsection{Quadratic recognition of higher stability}

The calculation in Appendix~\ref{app:dg-species} translates the
Frobenius condition into isomorphisms of $R$-bimodules. This
formulation includes automorphism twists of the coefficient division
algebras and requires no choice of arrow bases.

Fix $d\geq1$ and put $T:=V^{-d}$. Consider the following conditions.
\begin{enumerate}
\item[(M1)] $V^{<-d}=0$.
\item[(M2)] The $R$-bimodule $T$ is invertible.
Equivalently, there are a permutation $\tau_d$ of $I$, $k$-algebra
isomorphisms $\sigma_v\colon D_{\tau_d(v)}\xrightarrow{\sim}D_v$,
and generators $t_v\in e_vTe_{\tau_d(v)}$ such that
\[
T=\bigoplus_{v\in I}D_vt_v
 =\bigoplus_{v\in I}t_vD_{\tau_d(v)},\qquad
t_va=\sigma_v(a)t_v\quad(a\in D_{\tau_d(v)}).
\]
\item[(M3)] For every $0\leq i<d$, contraction with the
$(-i,-d+i+1)$-component of $\partial_2|_T$ is an isomorphism of
$R$-bimodules
\begin{equation}\label{eq:mesh-perfect-tensor}
\begin{aligned}
\kappa_i\colon (V^{-i})^\vee\otimes_R T
&\xrightarrow{\ \sim\ }V^{-d+i+1},\\
f\otimes t&\longmapsto\sum_j f(a_j)b_j,
\qquad (\partial_2t)_{-i,-d+i+1}=\sum_j a_j\otimes b_j.
\end{aligned}
\end{equation}
\end{enumerate}
The dual actions and $R$-bilinearity of $\partial_2$ make this contraction
well defined over $R$. After choosing the generators in {\rm (M2)},
condition {\rm (M3)} is equivalently the assertion that
\[
(e_vV^{-i})^\vee\longrightarrow V^{-d+i+1}e_{\tau_d(v)},
\qquad f\longmapsto(f\otimes1)(\partial_2t_v)_{-i,-d+i+1}
\]
is an isomorphism of left $R$-modules for every $v$ and $i$.
In the split case, {\rm (M2)} says that there is one degree-$-d$
arrow $t_v\colon\tau_d(v)\to v$ for each $v$, and {\rm (M3)} is
exactly the perfect-tensor condition
$D(e_vV^{-i})\xrightarrow{\sim}V^{-d+i+1}e_{\tau_d(v)}$.
All intermediate vertices are included in this contraction.

\begin{Thm}[Quadratic recognition]\label{thm:quadratic-recognition}
Let $A=\widehat T_R(V)$ be a completed tensor dg algebra of a
graded species satisfying {\rm (P1)--(P4)}, and let $d\geq1$.
The following conditions are equivalent.
\begin{enumerate}
\item $\Ext_A^*(R,R)$ is graded Frobenius of parameter $d+1$.
\item Conditions {\rm (M1)--(M3)} hold.
\end{enumerate}
If moreover $H^0A$ is finite dimensional, these conditions are
equivalent to $\proj_{\dg}A$ being $d$-stable.
\end{Thm}
\begin{proof}
Put $E=\Ext_A^*(R,R)$ and $P=E^{d+1}=T^\vee$. By the Ext
formula, {\rm (M1)} is equivalent to $E^{>d+1}=0$ and implies that
$E$ is finite dimensional. A Frobenius functional of parameter $d+1$
forces this vanishing and identifies $P$ with $DR$ on each side by
the perfect pairings between $R$ and $P$. Since $R$ is a product of
division algebras, $DR$ is regular on each side. The faithful
right action on ${}_RP\cong {}_RR$ is an isomorphism
$R^{\op}\to\End_R({}_RP)$ by dimension. Thus $P$ and its left
$R$-dual $T$ are invertible, giving {\rm (M2)}.

Under {\rm (M1)--(M2)}, $P$ is invertible and $P\cong DR$ as a left
$R$-module. A Frobenius functional induces such an isomorphism, and
any such isomorphism gives a functional on $P$ by evaluation at $1$.
Via $\Hom_{R^{\op}}(-,DR)\cong D(-)$, the Frobenius pairing criterion
therefore amounts to bijectivity of the right multiplication maps
$E^p\to\Hom_{R^{\op}}(E^{d+1-p},P)$. The maps for $p=0,d+1$ are
isomorphisms by invertibility of $P$. For $p=i+1$ with $0\leq i<d$,
right $R$-duality and the Yoneda product formula identify the map with
\begin{equation}\label{eq:yoneda-top-pairing}
(V^{-i})^\vee\longrightarrow\Hom_R(T,V^{1-d+i}),\qquad
f\longmapsto(-1)^i\kappa_i(f\otimes-).
\end{equation}
Since $-\otimes_R T$ is an equivalence, these adjoints are
isomorphisms exactly when the $\kappa_i$ are.

If $H^0A$ is finite dimensional, Corollary~\ref{cor:species-finiteness}(2)
and \eqref{eq:species-semisimple-quotient} give the hypotheses of
Theorem~\ref{thm:yoneda-frobenius} with $S=R$.
\end{proof}

Combining Theorem~\ref{thm:quadratic-recognition} with the direct
construction in Proposition~\ref{prop:direct-minimal-tensor-model}
gives the following structure result.

\begin{Cor}[Structure of $d$-stable dg algebras]
\label{cor:d-stable-species-model}
Let $A$ be a locally finite connective dg $k$-algebra such that
$R:=H^0A/\rad H^0A$ is basic and separable over $k$, and let $d\geq1$.
The following conditions are equivalent.
\begin{enumerate}
\item $A$ is $d$-stable.
\item $A$ is quasi-isomorphic to a completed tensor dg algebra
$B=(\widehat T_R(V),\partial)$ satisfying {\rm (P1)--(P4)} and
{\rm (M1)--(M3)}, with $V$ finite dimensional and concentrated in
degrees $-d,\ldots,0$.
\end{enumerate}
\end{Cor}
\begin{proof}
Proposition~\ref{prop:direct-minimal-tensor-model} supplies a model
$B=(\widehat T_R(V),\partial)$ satisfying {\rm (P1)--(P4)}.
Theorem~\ref{thm:yoneda-frobenius} makes $d$-stability invariant
under quasi-isomorphisms, and Theorem~\ref{thm:quadratic-recognition}
characterizes it for $B$ by {\rm (M1)--(M3)} since $H^0B\cong H^0A$
is finite dimensional. By {\rm (P3)}, the bounded generator space
in {\rm (M1)} is finite dimensional. The same criteria apply to
any model in {\rm (2)}, giving the converse.
\end{proof}

\begin{Cor}\label{thm:higher-mesh-d-stable}
Let $A=\widehat T_R(V)$ satisfy {\rm (P1)--(P4)} and
{\rm (M1)--(M3)}. If $H^0A$ is finite dimensional, then $A$ is $d$-stable,
with $\gl A=d+1$ when $R\ne0$, and there are isomorphisms
in $\D(A^{\op})$
\[
\RHom_A(S_v,A)\cong{}_{\tau_d(v)}S[-d-1]
\qquad(v\in I),
\]
where ${}_uS=Re_u$ denotes the left vertex simple.
\end{Cor}
\begin{proof}
Theorem~\ref{thm:quadratic-recognition} gives $d$-stability.
For $R\ne0$ the invertible $T$ is nonzero, so
Corollary~\ref{cor:species-finiteness}(2) gives $\gl A=d+1$.
By Theorem~\ref{thm:aus-cor-for-loc-fin},
$\RHom_A(S_v,A)$ is a left simple shifted by $[-d-1]$.
In the finite resolution of Proposition~\ref{prop:species-minimal-resolution},
the term in shift $[d+1]$ is $e_vT\otimes_RA\cong e_{\tau_d(v)}A$.
The dual resolution is still minimal, since its coefficients
lie in the augmentation ideal. Mapping it to ${}_uS[-d-1]$
therefore gives a nonzero $\Hom_{\D(A^{\op})}$-space exactly for
$u=\tau_d(v)$, identifying the left simple.
\end{proof}

Conditions {\rm (M1)--(M3)} are intrinsic to the species. In the
split case, a tensor $\sum_{a,b}c_{ab}\alpha_a\beta_b$ is
nondegenerate precisely when its coefficient matrix $(c_{ab})$ is
invertible; scalar twists and signs are therefore allowed.
No symmetry or cyclicity is needed for higher stability. Such
compatibilities concern additional Calabi--Yau structures.

For fixed $R$, $V$, and $\partial_2$, finite dimensionality of $H^0A$
can depend on the higher-length terms
(Example~\ref{ex:twisted-cycle-family}).

\subsection{\texorpdfstring{$d=1\colon$}{d=1:} deformed derived mesh algebras}
\label{subsec:derived-mesh-algebras}

For $d=1$, conditions {\rm (M1)--(M3)} give $V=M\oplus T$ in degrees
$0,-1$, with $T$ invertible and quadratic contraction
$\kappa\colon M^\vee\otimes_RT\xrightarrow{\sim}M$.
We call $\mathfrak S=(R,M,T,\kappa)$ a \emph{finite stable
translation species} with \emph{polarization} $\kappa$.
Its valued translation quiver $\Gamma$ has an arrow $e\colon u\to v$
for each nonzero $e_vMe_u$, valued by $(\dim_{D_u}(e_vMe_u),\dim_{D_v}(e_vMe_u))$.
Write $\tau=\tau_1$. The block
$(e_vMe_u)^\vee\otimes_{D_v}e_vT\cong e_uMe_{\tau(v)}$ of $\kappa$
gives the mate $\sigma(e)\colon\tau(v)\to u$ with the ranks
interchanged, so $\sigma^2=\tau$ preserves the valuations.
The \emph{valued tree class} is the tree $\Delta$ in Riedtmann's
repetition cover $\mathbb Z\Delta$ \cite[Section~1.6]{Rie80}.
It is obtained from sectional paths starting at a vertex. For components
with loops these paths lift uniquely to the bipartite double, where
arrows reverse parity and $\tau$ preserves it.

Set $B=\widehat T_R(M)$ and $J=B_{\geq1}$. An $R$-bimodule map
$\delta\colon T\to J^3$ defines the dg algebra
\[
A(\mathfrak S,\delta)=(\widehat T_R(M\oplus T),\partial),\qquad
\partial(M)=0,\qquad \partial|_T=\rho_\kappa+\delta,
\]
where $\rho_\kappa(t)=\sum_jm_j\otimes\kappa(f_j\otimes t)$ for a
right dual basis $(m_j,f_j)$ of $M_R$.

\begin{Def}\label{def:derived-mesh-algebra}
Let $\mathfrak S$ be a finite stable translation species with Dynkin tree
classes. We call $A(\mathfrak S,\delta)$ a \emph{deformed derived mesh algebra}
of $\mathfrak S$, with the deformation given by the higher-length terms
$\delta$ in the mesh differential. We call $A(\mathfrak S,0)$ the
\emph{derived mesh algebra} of $\mathfrak S$.
\end{Def}

Set $\Lambda=H^0A(\mathfrak S,\delta)$, $\Lambda_0=H^0A(\mathfrak S,0)$
and $P=T_R(M)/(\rho_\kappa(T))$, graded by length. Then
\[
\Lambda=B/\overline{(\partial(T))},\qquad
\Lambda_0=\prod_{n\geq0}P_n,\qquad
P\twoheadrightarrow\operatorname{gr}_J\Lambda.
\]
In every such tensor model the words with one $V^{-1}$-factor span
a dense subspace of $A^{-1}$, whose differential has closed image by
linear compactness. Thus $\partial(A^{-1})=\overline{(\partial(V^{-1}))}$ and the
quotient is separated.

At the level of ordinary algebras, this modification of the mesh relations
by higher-length terms also appears in the deformed mesh algebras of
Bia{\l}kowski, Erdmann, and Skowro\'nski \cite[p.~221]{BES12Mesh}.

The isomorphisms $\sigma_v^{-1}\colon D_v\to D_{\tau(v)}$ and the
$D_u$-duals of the blocks of $\kappa(-\otimes t_v)$ give a $\tau$-species
whose mesh at $v$ is the coevaluation sum $\rho_\kappa(t_v)$
\cite[Definition~7.9]{INP}. Thus $\Lambda_0$ is its complete mesh
algebra, and on components with arrows $\proj\Lambda_0$ is a
$\tau$-category with Auslander--Reiten quiver $\Gamma$
\cite[Definition~7.12 and Proposition~7.13]{INP}.

In Dynkin tree type, $\Lambda_0$ is finite dimensional. Isolated
vertices contribute only $D_v$. Every other component, symmetrized by
$\dim_kD_v$, is also the Auslander--Reiten quiver of a derived orbit
category $\D^b(\mod H)/F$ for a Dynkin species tensor algebra $H$
\cite[Remark~7.3 and Example~8.6(3)]{INP}. Its Hom spaces are finite dimensional
since $F$ acts freely on indecomposables and a module maps only to
shifts $0,1$ of finitely many indecomposable modules. Since the quiver
is finite, all sufficiently high radical layers vanish. Their ranks over
the source residue division rings depend only on the valued translation
quiver \cite[Section~1.3.3]{Iya05}. In $\proj\Lambda_0$ these layers
are the vertex components of $P_n$, as $\rad\Lambda_0=J\Lambda_0$.
Thus $P_n=0$ for $n\gg0$ and $\Lambda_0=P$.

\begin{Thm}[Dynkin criterion]\label{thm:derived-mesh-stable}
Let $\mathfrak S$ be a finite stable translation species and
$\delta\colon T\to J^3$ an $R$-bimodule map. The following are equivalent.
\begin{enumerate}
\item $H^0A(\mathfrak S,\delta)$ is finite dimensional.
\item $A(\mathfrak S,\delta)$ is $1$-stable.
\item Every component of $\mathfrak S$ has finite valued Dynkin tree class.
\end{enumerate}
\end{Thm}
\begin{proof}
Theorem~\ref{thm:quadratic-recognition} gives $(1)\Leftrightarrow(2)$.
For $A=A(\mathfrak S,\delta)$ and $\mathfrak a=\partial(A^{-1})$,
splitting off the first factor by \eqref{eq:species-first-last-factor}
gives $\mathfrak a=M\mathfrak a+\partial(T)B$, hence an exact sequence
\begin{equation}\label{eq:mesh-simple-presentation}
T\otimes_R\Lambda\longrightarrow M\otimes_R\Lambda
\longrightarrow\Lambda\longrightarrow R\longrightarrow0.
\end{equation}
Under {\rm (1)}, let $h_v$ be the $\tau$-orbit average of
$\dim_k e_v\Lambda$ and $a_{vu}$ the right rank of $e_vMe_u$.
Since $e_vT$ has right rank one, multiplying \eqref{eq:mesh-simple-presentation}
on the left by $e_v$ gives
$\sum_u a_{vu}\dim_k(e_u\Lambda)\le\dim_k(e_{\tau(v)}\Lambda)+\dim_k(e_v\Lambda)-\dim_k D_v$;
averaging over $\tau$-orbits yields
\begin{equation}\label{eq:mesh-positive-weight}
\sum_u a_{vu}h_u\leq 2h_v-\dim_kD_v<2h_v.
\end{equation}
Along an infinite ray in $\Delta$ the successive differences of the
lifted $h$ would strictly decrease in a finite set. Thus $\Delta$ is
finite, and \eqref{eq:mesh-positive-weight} gives {\rm (3)} by the
positive vector criterion for its Cartan matrix
\cite[Corollary~4.3 and Theorem~4.8]{Ka}.

Under {\rm (3)}, $\Lambda_0=P$ is finite dimensional by the preceding
comparison. The surjection $P\twoheadrightarrow\operatorname{gr}_J\Lambda$
and separatedness give {\rm (1)}.
\end{proof}

The Dynkin criterion characterizes $1$-stability for a given mesh
presentation. For a $1$-stable tensor model $A$, the next theorem
identifies the species in its mesh presentation with the
Auslander--Reiten species of $\C_1(A)$.

\begin{Thm}[Auslander--Reiten mesh presentation]
\label{thm:ar-mesh-presentation}
Let $A=(\widehat T_R(V),\partial)$ be a $1$-stable completed tensor dg algebra
satisfying {\rm (P1)--(P4)}, and let
$\mathfrak S_{\mathrm{AR}}=(R,M,T,\kappa_{\mathrm{AR}})$ be the
Auslander--Reiten species of $\C_1(A)$. Then a linear change of
generators puts $A$ in the form $A(\mathfrak S_{\mathrm{AR}},\delta)$
for an $R$-bimodule map $\delta\colon T\to J^3$.
\end{Thm}
\begin{proof}
By quadratic recognition, $V=M\oplus T'$ with $T'=V^{-1}$ invertible.
By \eqref{eq:mesh-E-differential}, $e_vA_{\geq1}\cong\operatorname{Cone}(f_v)$,
where $f_v$ extends $\ell\partial|_{e_vT'}$ $A$-linearly and $T'$ is placed in
degree zero. The inclusion $e_vA_{\geq1}\to e_vA$ has cone $e_vR$,
so it gives the following $D_v$-equivariant triangle in $\C_1(A)$:
\[
e_vT'\otimes_R A\xrightarrow{f_v}e_vM\otimes_R A
\xrightarrow{\mu_v}e_vA\longrightarrow(e_vT'\otimes_R A)[1].
\]
For $\Lambda=H^0A$, \eqref{eq:species-first-radical-layer} identifies
$M=V^0$ with $\rad\Lambda/\rad^2\Lambda$. Under
$\C_1(A)\simeq\proj\Lambda$ (Theorem~\ref{thm:fundamental-domain}),
multiplication $\mu_v$ therefore gives the projective cover of
$e_v\rad\Lambda$ followed by inclusion. Thus this is an Auslander--Reiten
triangle.

The $D_v$-action on the full triangle induces $\sigma_v^{-1}$ on the
residue ring of its first term, identifying $e_vT'\cong e_vT$ as bimodules
\cite[Proposition~2.6(3)]{INP}. The irreducible components of $f_v$
are the length-one coefficients of $\ell\partial$, so the mesh map under
this identification is $\partial_2|_{e_vT'}$
\cite[Example~7.11 and Definition~7.12]{INP}. This gives the required form.
\end{proof}

In the split case $R=k^I$, the Dynkin tree classes are simply laced
and the arrow spaces are one-dimensional. Process the vertices in
decreasing order of $h_v$. By \eqref{eq:mesh-positive-weight} at most
one arrow $u\to v$ has $h_u\geq h_v$, and we normalize its coefficient
by rescaling $t_v$ if it exists. The other coefficients are normalized
by rescaling the mates, whose other occurrences are at vertices of
strictly smaller weight. In these generators,
\[
\partial(t_v)=\sum_{t(\alpha)=v}\alpha\sigma(\alpha)+\delta_v,
\qquad \delta_v\in e_vJ^3e_{\tau(v)}.
\]

\begin{Cor}[Realization by the Auslander--Reiten species]
\label{cor:translation-species-realization}
Let $\C$ be a locally finite pretriangulated dg $k$-category whose
homotopy category $H^0\C$ has a basic additive generator $M$. Put
$\Lambda=\End_{H^0\C}(M)$, and suppose that $R=\Lambda/\rad\Lambda$
is separable over $k$. The Auslander--Reiten species
$\mathfrak S_{\mathrm{AR}}=(R,M_{\mathrm{AR}},T,\kappa_{\mathrm{AR}})$ of $H^0\C$
is a finite stable translation species with Dynkin tree classes.
There is a deformed derived mesh algebra $A(\mathfrak S_{\mathrm{AR}},\delta)$
such that there is a quasi-equivalence
\[
\C_{1,\dg}(A(\mathfrak S_{\mathrm{AR}},\delta))\simeq\C
\]
sending the free module to $M$.
\end{Cor}
\begin{proof}
Since $M$ is $1$-cluster tilting, Corollary~\ref{thm:Amiot-conj}(1)
gives $1$-stability of $A=\tau^{\leq0}\End_\C(M)$ and identifies
$\C_{1,\dg}(A)$ with $\C$, sending the free module to $M$.
Corollary~\ref{cor:d-stable-species-model} gives a minimal completed
tensor model, to which Theorem~\ref{thm:ar-mesh-presentation} applies.
The tree classes are Dynkin by Theorem~\ref{thm:derived-mesh-stable}.
\end{proof}

\subsection{\texorpdfstring{$d=2\colon$}{d=2:} twisted Ginzburg dg algebras}
\label{subsec:twisted-mesh-presentations}

We construct complete twisted Ginzburg dg algebras over species,
starting with quivers in the split case. The invertible bimodule $T$
in degree $-2$ records the vertex translation and the coefficient
twist. A continuous automorphism $\theta$ extends this twist to the
completed degree-zero algebra, and a $\theta$-twisted hyperpotential
specifies the degree-zero relations.

The resulting normal form applies to completed tensor models satisfying
{\rm (M1)--(M3)} for $d=2$, with possibly nonseparable coefficient algebras.
Corollary~\ref{cor:twisted-ginzburg-realization} gives a twisted analogue of
the categorical realization in Amiot's conjecture: a locally finite
pretriangulated dg category with a $2$-cluster tilting object $M$ is
quasi-equivalent to the $2$-cluster dg category of a complete twisted
Ginzburg dg algebra, under the separability of the semisimple
endomorphism quotient of $M$.

\emph{The split case.}
Let $Q$ be a finite quiver, $B=\widehat{kQ}$ and $J$ its arrow ideal.
Let $\theta$ be a continuous $k$-algebra automorphism of $B$ with
$\theta(e_v)=e_{\tau_2(v)}$ for a vertex permutation $\tau_2$.
For each arrow $\alpha\colon u\to v$,
adjoin $a_\alpha\colon\tau_2(v)\to u$ of degree $-1$ and choose
$\rho_\alpha\in e_uJ^2e_{\tau_2(v)}$. At each vertex adjoin
$t_v\colon\tau_2(v)\to v$ of degree $-2$. With $s(\alpha),t(\alpha)$
denoting source and target, the differential is
\begin{equation}\label{eq:split-twisted-ginzburg-differential}
\begin{aligned}
\partial\alpha&=0,\qquad \partial a_\alpha=\rho_\alpha,\\
\partial t_v&=\sum_{t(\alpha)=v}\alpha a_\alpha
              -\sum_{s(\beta)=v}a_\beta\theta(\beta).
\end{aligned}
\end{equation}
The condition $\partial^2=0$ amounts to
\begin{equation}\label{eq:split-twisted-ginzburg-compatibility}
\sum_{t(\alpha)=v}\alpha\rho_\alpha
=\sum_{s(\beta)=v}\rho_\beta\theta(\beta)
\qquad(v\in Q_0).
\end{equation}
A family satisfying this identity is a \emph{reduced formal
$\theta$-twisted hyperpotential}. Write $\Gamma(Q,\theta,\rho)$ for
the resulting complete dg algebra. Splitting off the first arrow
identifies relation families with tensors $w_\rho=\sum_\alpha\alpha\rho_\alpha$
in $\bigoplus_v e_vJ^3e_{\tau_2(v)}$. Equation
\eqref{eq:split-twisted-ginzburg-compatibility} says that $w_\rho$
is fixed by the rotation $\sum_\alpha\alpha q_\alpha\mapsto
\sum_\alpha q_\alpha\theta(\alpha)$.

For $\theta=\mathrm{id}_B$ the condition is $\sum_\alpha[\alpha,\rho_\alpha]=0$.
The cyclic derivatives $\rho_\alpha=\partial_\alpha W$ of a potential
of length at least three give the usual complete Ginzburg algebra
\cite{Gin06,Kel11}. In this untwisted quiver case every reduced
hyperpotential arises this way in characteristic zero, and this can
fail in positive characteristic \cite[Section~1.1]{Lad14}.

\emph{The species construction.}
Let $M$ be a finite $R$-bimodule and $T$ an invertible $R$-bimodule.
Choose $t=\sum_vt_v$ as in {\rm (M2)}, a basis of $T$ on either side.
The automorphism $\sigma_R$ given by $rt=t\sigma_R(r)$ has vertex
permutation $\tau_2$. Set $B=\widehat T_R(M)$, $J=B_{\geq1}$ and
$N=M^\vee\otimes_R T$.
For a finite $R$-bimodule $X$, let $\operatorname{coev}_X\colon
R\to X\otimes_R X^\vee$ be the left action via
$\End_R(X_R)\cong X\otimes_R X^\vee$. By projectivity of $X_R$,
contraction $f\otimes x\otimes y\mapsto f(x)y$ has natural inverse
\[
\Hom_{R^e}(X^\vee\otimes_R T,Y)\xrightarrow{\sim}\Hom_{R^e}(T,X\otimes_R Y),
\qquad F\longmapsto(1_X\otimes F)(\operatorname{coev}_X\otimes 1_T).
\]
Here $R^e=R\otimes_kR^{\op}$, and for $X=M$ this sends $1_N$ to
$c\colon T\to M\otimes_R N$.

Fix a symmetrizing form $\lambda\colon R\to k$, obtained by summing the
surjective reduced traces of the $D_v$ composed with nonzero $k$-linear
functionals on their centres. The trace pairing $R\cong DR$ and
$\Hom_R(X,DR)\cong DX$ identify $X^\vee\cong DX$ as bimodules by
$f\mapsto\lambda\circ f$ \cite[Remark~2.2 and Lemma~2.1(b)]{Ngu10}.
The resulting biduality
$m\mapsto u_m\in M^{\vee\vee}$ satisfies $\lambda(u_m(f)r)=\lambda(f(rm))$.
In the coordinates $f\otimes t$ on $N$ the right action is twisted by
$\sigma_R^{-1}$, so $j(m)=(\sigma_R\circ u_m)\otimes t$ is a
$k$-linear bijection $M\to N^\vee\otimes_R T$.
Using $u_{rmr'}=r u_m r'$ and $r't=t\sigma_R(r')$ gives
$j(rmr')=\sigma_R(r)j(m)\sigma_R(r')$. For $k$-linear
$\phi\colon M\to Y$ with the same semilinearity, $\phi j^{-1}$ is
$R$-bilinear. Its inverse contraction for $X=N$ gives a natural
bijection $\phi\mapsto K_\phi$ onto $\Hom_{R^e}(T,N\otimes_R Y)$.

Let $\theta$ be a continuous $k$-algebra automorphism of $B$ extending
$\sigma_R$. It preserves $J=\rad B$. A \emph{reduced formal
$\theta$-twisted hyperpotential} is an $R$-bimodule map $\rho\colon N\to J^2$
satisfying
\begin{equation}\label{eq:twisted-hyperpotential}
\mu(1_M\otimes\rho)c=\mu(\rho\otimes1_B)K_{\theta|_M},
\end{equation}
where $\mu$ denotes multiplication. Since $M\otimes_R J^2\cong J^3$,
the left side $w_\rho$ identifies all relation maps with $\Hom_{R^e}(T,J^3)$,
or, by evaluation at $t$, with tensors $w\in J^3$ satisfying
$rw=w\sigma_R(r)$. The right side defines $\operatorname{rot}_\theta(w_\rho)$
in these coordinates, so hyperpotentials are precisely the fixed tensors.

The \emph{complete twisted Ginzburg dg algebra} $\Gamma(M,T,\theta,\rho)$
is $\widehat T_R(M\oplus N\oplus T)$ in generator degrees $0,-1,-2$ with
\begin{equation}\label{eq:twisted-mesh-differential}
\partial|_M=0,\qquad \partial|_N=\rho,\qquad
\partial|_T=\mu c-\mu K_{\theta|_M}.
\end{equation}
Here $t,\lambda$ are understood. When $R=k^{Q_0}$, $M=kQ_1$ and
$\lambda(e_v)=1$, the arrow basis gives
\eqref{eq:split-twisted-ginzburg-differential}.
As in Subsection~\ref{subsec:derived-mesh-algebras},
$H^0\Gamma=B/\overline{(\rho(N))}$.

\begin{Thm}[Twisted Ginzburg construction]\label{thm:twisted-ginzburg-construction}
The differential \eqref{eq:twisted-mesh-differential} makes
$\Gamma(M,T,\theta,\rho)$ a completed tensor dg algebra satisfying
{\rm (P1)--(P4)} and {\rm (M1)--(M3)} for $d=2$.
\end{Thm}
\begin{proof}
The derivation raises tensor length. Its square vanishes on $M,N$
by degree and on $T$ by \eqref{eq:twisted-hyperpotential}.
The generator degrees and invertibility of $T$ give {\rm (M1)--(M2)}.
The quadratic contractions are $1_N$ and $-\theta_1j^{-1}$, where
$\theta_1$ is the invertible linear part of $\theta$, giving {\rm (M3)}.
\end{proof}

\emph{Normal form.}
Let $A=(\widehat T_R(V),\partial)$ satisfy {\rm (P1)--(P4)} and
{\rm (M1)--(M3)} for $d=2$. Put $M=V^0$,
$N=V^{-1}$, $T=V^{-2}$, $B=A^0$, and $J=B_{\geq1}$.
Choose $t,\lambda$ as above.
The first contraction in {\rm (M3)} identifies $N$ with
$M^\vee\otimes_RT$. With $\kappa_1$ the second contraction,
$\Theta=-\kappa_1 j$ is bijective and $\sigma_R$-semilinear, and
$\partial_2|_T=c-K_\Theta$.

\begin{Thm}[Normal form relative to fixed generators]
\label{thm:twisted-mesh-basis-characterization}
With the generators and duality data above fixed, there is a unique pair
$(b,\theta)$ consisting of an $R$-bimodule map $b\colon N\to A^{-1}$
and a continuous $k$-algebra automorphism $\theta$ of $B$ extending
$\sigma_R$ such that
\begin{equation}\label{eq:mesh-formal-normal-form}
\partial|_T=\mu(1_M\otimes b)c-\mu(b\otimes1_B)K_{\theta|_M}.
\end{equation}
Their linear parts are $1_N$ and $\Theta$, respectively.
The map $\rho=\partial b$ is a reduced formal
$\theta$-twisted hyperpotential, and the substitution $n\mapsto b(n)$,
fixing $R,M,T$, induces a continuous dg algebra isomorphism
$\Gamma(M,T,\theta,\rho)\xrightarrow{\sim}A$.
If $\partial|_T=\partial_2|_T$, then $b(n)=n$ and $\theta|_M=\Theta$.
\end{Thm}
\begin{proof}
Write $W_{p,q}$ and $b_{p,q}$ for the components of $\partial|_T$ and $b$ in
\[
A^{-1}=\prod_{p,q\geq0}
\bigl(M^{\otimes_R p}\otimes_R N\otimes_R M^{\otimes_R q}\bigr).
\]
The contractions above give the natural bijections
$h\mapsto(1_M\otimes h)c$ and $\phi\mapsto K_\phi$ for
$\sigma_R$-semilinear $\phi$. Any twist preserves $J=\rad B$, so we seek
$\theta|_M=\sum_{j\geq1}\theta_j$ by length, with $K_j=K_{\theta_j}$.
At $q=0$, minimality gives $W_{0,0}=0$ and
\eqref{eq:mesh-formal-normal-form} gives $(1_M\otimes b_{p,0})c=W_{p+1,0}$.
This determines all $b_{p,0}$, with $b_{0,0}=1_N$ since $W_{1,0}=c$.
We now determine $\theta_q$ and all $b_{p,q}$ by induction on $q\geq1$.
The $(0,q)$ and $(p+1,q)$ components give
\[
\begin{aligned}
K_q &= -W_{0,q}-\sum_{j=1}^{q-1}(b_{0,q-j}\otimes1)K_j,\\
(1_M\otimes b_{p,q})c
    &= W_{p+1,q}+\sum_{j=1}^q(b_{p+1,q-j}\otimes1)K_j
       \qquad(p\geq0).
\end{aligned}
\]
The first line determines $K_q$ and hence $\theta_q$.
In the second line all $K_j$ are now known and $q-j<q$, so it
determines every $b_{p,q}$. In particular, $\theta_1=\Theta$ by
$W_{0,1}=-K_\Theta$.

Only finitely many pairs have a given length $p+1+q$, so these equations
give the unique solution in the products $A^{-1}$ and
$J=\prod_{q\geq1}M^{\otimes_R q}$. The generator maps extend uniquely to
continuous substitutions over $\mathrm{id}_R$ and $\sigma_R$, respectively.
Their linear parts are invertible, so Lemma~\ref{lem:formal-tensor-inverse}(1)
makes both substitutions automorphisms.
Minimality gives $\rho=\partial b\colon N\to J^2$. Applying $\partial$
to \eqref{eq:mesh-formal-normal-form} gives \eqref{eq:twisted-hyperpotential},
so $n\mapsto b(n)$ is the claimed dg isomorphism. For quadratic
$\partial|_T$ the linear pair already solves \eqref{eq:mesh-formal-normal-form}.
\end{proof}

In $\Gamma$, the pair $(N\hookrightarrow\Gamma^{-1},\theta)$ solves
\eqref{eq:mesh-formal-normal-form}, so uniqueness and
\eqref{eq:twisted-mesh-differential} give $\partial|_T=\partial_2|_T$
exactly when $\theta(M)=M$.

\begin{Cor}[Koszul-dual characterization]
\label{cor:twisted-ginzburg-characterization}
Let $A=\widehat T_R(V)$ satisfy {\rm (P1)--(P4)} and set
\[
A^!=\RHom_A(R,R).
\]
Then $A$ admits an $R$-fixing continuous
twisted Ginzburg presentation if and only if
$D(A^!)\cong A^![3]$ in $\D(A^!)$. Equivalently, one can require
$\add D(A^!)=\add A^![3]$. If $H^0A$ is finite dimensional, these
conditions are also equivalent to $2$-stability.
\end{Cor}
\begin{proof}
By Corollary~\ref{cor:mesh-ext-dimension}, $B=A^!$ is locally finite
and positive with $H^0B=R$ basic. Proposition~\ref{prop:cohomological-frobenius}
therefore identifies both dg conditions with $E=H^*B$ being graded
Frobenius of parameter $3$.

Theorem~\ref{thm:quadratic-recognition} identifies this with
{\rm (M1)--(M3)} and, for finite-dimensional $H^0A$, with $2$-stability.
Theorems~\ref{thm:twisted-ginzburg-construction} and~\ref{thm:twisted-mesh-basis-characterization}
give the equivalence with twisted Ginzburg presentations, since a dg
isomorphism preserves the Yoneda algebra of the semisimple quotient
\eqref{eq:species-semisimple-quotient}.
\end{proof}

\emph{Realization.}
Corollaries~\ref{cor:d-stable-species-model} and~\ref{cor:twisted-ginzburg-characterization}
characterize $2$-stability of locally finite connective dg algebras
with basic separable semisimple quotient by the existence of a
twisted Ginzburg model up to ordinary dg quasi-isomorphism.

\begin{Cor}[Twisted Ginzburg realization]
\label{cor:twisted-ginzburg-realization}
Let $\C$ be a locally finite pretriangulated dg $k$-category with
$H^0\C$ idempotent complete. Let $M$ be a basic $2$-cluster
tilting object and suppose that
\[
R=\End_{H^0\C}(M)/\rad\End_{H^0\C}(M)
\]
is separable over $k$. There is a complete twisted Ginzburg dg
algebra $\Gamma$ over $R$, quasi-isomorphic
as an ordinary dg algebra to $\tau^{\leq0}\End_\C(M)$, and a
quasi-equivalence
\[
\C_{2,\dg}(\Gamma)\simeq\C
\]
sending $\Gamma$ to $M$.
\end{Cor}
\begin{proof}
By Corollary~\ref{thm:Amiot-conj}(1), $A_M=\tau^{\leq0}\End_\C(M)$ is
$2$-stable and $\C_{2,\dg}(A_M)\simeq\C$ sends $A_M$ to $M$.
Its semisimple quotient $R$ is basic since $M$ is basic. The preceding
criterion thus gives a twisted Ginzburg model $\Gamma$ for $A_M$.
The quasi-isomorphisms identify the $2$-cluster dg categories,
sending the free module of $\Gamma$ to $A_M$.
\end{proof}

For $R=k^I$ this gives the quiver presentation $\Gamma(Q,\theta,\rho)$ of
\eqref{eq:split-twisted-ginzburg-differential}.
Conversely, finite-dimensional $H^0\Gamma$ makes $\Gamma$ locally
finite and $2$-stable by Corollaries~\ref{cor:species-finiteness}(2)
and~\ref{cor:twisted-ginzburg-characterization}. The fundamental
domain then gives a locally finite $\C_{2,\dg}(\Gamma)$ with canonical
$2$-cluster tilting object $\Gamma$ (Theorem~\ref{thm:fundamental-domain}).

\emph{Examples.}
For $\theta=\mathrm{id}_B$, $t$ identifies $T$ with $R$ and
$K_{\mathrm{id}_M}(t)$ is the left coevaluation. Thus
$\operatorname{rot}_{\mathrm{id}_B}$ is Nguefack's left cyclic
permutation. By \cite[Definition~3.8 and Proposition~3.2]{Ngu10},
the cyclic derivative $\rho$ of a length-$n$ species potential
represented by an $R$-central tensor $W$ satisfies
\[
w_\rho(t)=\sum_{i=0}^{n-1}\operatorname{rot}_{\mathrm{id}_B}^{\,i}(W)
\qquad(n\geq3).
\]

\begin{Ex}[Relation length determines the vertex translation]
\label{ex:twisted-cycle-family}
Let $C_n$ be the oriented cycle with arrows
$\alpha_i\colon i\to i+1$, indexed by $\mathbb Z/n\mathbb Z$,
and let $J$ be its arrow ideal. For $L\geq2$, take the translation
$\tau_2(i)=i-L-1$, the twist $\theta(\alpha_i)=\alpha_{i-L-1}$,
and the relation $\rho_{\alpha_i}=\alpha_{i-1}\cdots\alpha_{i-L}$.
The tensor $w_\rho=\sum_i\alpha_i\cdots\alpha_{i-L}$ is fixed by
rotation, which sends the $i$th summand to the $(i-1)$st.
The relations generate $J^L$, so the resulting $\Gamma_{n,L}$ has
$H^0\Gamma_{n,L}\cong kC_n/J^L$
of dimension $nL$ and is $2$-stable.
Figure~\ref{fig:twisted-cycle-family} shows the three patterns for
$n=3$, with $b_i=a_{\alpha_i}$.
The twist is the identity when $L+1\equiv0\pmod n$.

Fix $n$ and let $L\geq3$ vary within a residue class modulo $n$.
The graded quiver and quadratic differential remain unchanged,
and hence so does the Yoneda algebra. The different dimensions
of $H^0$ distinguish the quasi-isomorphism classes. Taking $\rho=0$
with the same twist gives the same quadratic differential and
makes $H^0$ infinite dimensional.
\end{Ex}

In the split case, a unit $u\in 1+J$ commuting with the vertices gives
new generators $a_\alpha u,t_vu$. Equation~\eqref{eq:split-twisted-ginzburg-differential}
then replaces $\theta$ by $x\mapsto u^{-1}\theta(x)u$ and $\rho_\alpha$ by
$\rho_\alpha u$. In the cycle family with $n=2,L=3$, take
$u=1+\alpha_1\alpha_0$. The new twist sends $\alpha_0$ to
$\alpha_0+\alpha_0\alpha_1\alpha_0$, so the same $2$-stable dg algebra
admits a nonlinear twist with identity linear part.

\begin{figure}[htbp]
\centering
\begin{tikzpicture}[every node/.style={font=\small}]
\begin{scope}[xshift=0.0cm]
\node at (1.05,3.15) {$L=2$};
\node[dgq vertex] (c0v0) at (0,0) {$0$};
\node[dgq vertex] (c0v1) at (2.1,0) {$1$};
\node[dgq vertex] (c0v2) at (1.05,1.8) {$2$};
\draw[dgq zero] (c0v0) -- node[pos=.27,auto=left,dgq label] {$\alpha_0$} (c0v1);
\draw[dgq zero] (c0v1) -- node[pos=.27,auto=left,dgq label] {$\alpha_1$} (c0v2);
\draw[dgq zero] (c0v2) -- node[pos=.27,auto=left,dgq label] {$\alpha_2$} (c0v0);
\draw[dgq minus one] (c0v1) to[bend left=22] node[pos=.6,auto=left,dgq label] {$b_0$} (c0v0);
\draw[dgq minus one] (c0v2) to[bend left=22] node[pos=.6,auto=left,dgq label] {$b_1$} (c0v1);
\draw[dgq minus one] (c0v0) to[bend left=22] node[pos=.6,auto=left,dgq label] {$b_2$} (c0v2);
\draw[dgq minus two] (c0v0) edge[loop left,min distance=8mm] node[left,dgq label] {$t_0$} (c0v0);
\draw[dgq minus two] (c0v1) edge[loop right,min distance=8mm] node[right,dgq label] {$t_1$} (c0v1);
\draw[dgq minus two] (c0v2) edge[loop above,min distance=8mm] node[above,dgq label] {$t_2$} (c0v2);
\end{scope}
\begin{scope}[xshift=5.1cm]
\node at (1.05,3.15) {$L=3$};
\node[dgq vertex] (c1v0) at (0,0) {$0$};
\node[dgq vertex] (c1v1) at (2.1,0) {$1$};
\node[dgq vertex] (c1v2) at (1.05,1.8) {$2$};
\draw[dgq zero] (c1v0) -- node[pos=.27,auto=left,dgq label] {$\alpha_0$} (c1v1);
\draw[dgq zero] (c1v1) -- node[pos=.27,auto=left,dgq label] {$\alpha_1$} (c1v2);
\draw[dgq zero] (c1v2) -- node[pos=.27,auto=left,dgq label] {$\alpha_2$} (c1v0);
\draw[dgq minus one] (c1v0) edge[loop left,min distance=8mm] node[left,dgq label] {$b_0$} (c1v0);
\draw[dgq minus one] (c1v1) edge[loop right,min distance=8mm] node[right,dgq label] {$b_1$} (c1v1);
\draw[dgq minus one] (c1v2) edge[loop above,min distance=8mm] node[above,dgq label] {$b_2$} (c1v2);
\draw[dgq minus two] (c1v2) to[bend right=22] node[pos=.55,auto=right,dgq label] {$t_0$} (c1v0);
\draw[dgq minus two] (c1v0) to[bend right=22] node[pos=.55,auto=right,dgq label] {$t_1$} (c1v1);
\draw[dgq minus two] (c1v1) to[bend right=22] node[pos=.55,auto=right,dgq label] {$t_2$} (c1v2);
\end{scope}
\begin{scope}[xshift=10.2cm]
\node at (1.05,3.15) {$L=4$};
\node[dgq vertex] (c2v0) at (0,0) {$0$};
\node[dgq vertex] (c2v1) at (2.1,0) {$1$};
\node[dgq vertex] (c2v2) at (1.05,1.8) {$2$};
\draw[dgq zero] (c2v0) -- node[midway,dgq label] {$\alpha_0$} (c2v1);
\draw[dgq zero] (c2v1) -- node[midway,dgq label] {$\alpha_1$} (c2v2);
\draw[dgq zero] (c2v2) -- node[midway,dgq label] {$\alpha_2$} (c2v0);
\draw[dgq minus one] (c2v2) to[bend left=22] node[pos=.6,auto=left,dgq label] {$b_0$} (c2v0);
\draw[dgq minus one] (c2v0) to[bend left=22] node[pos=.6,auto=left,dgq label] {$b_1$} (c2v1);
\draw[dgq minus one] (c2v1) to[bend left=22] node[pos=.6,auto=left,dgq label] {$b_2$} (c2v2);
\draw[dgq minus two] (c2v1) to[bend left=22] node[pos=.55,auto=left,dgq label] {$t_0$} (c2v0);
\draw[dgq minus two] (c2v2) to[bend left=22] node[pos=.55,auto=left,dgq label] {$t_1$} (c2v1);
\draw[dgq minus two] (c2v0) to[bend left=22] node[pos=.55,auto=left,dgq label] {$t_2$} (c2v2);
\end{scope}
\end{tikzpicture}

\smallskip
\quiverdegreelegend
\caption{The cycle family for $n=3$, with all generator degrees shown on the same vertices in each panel. The endpoints of $b_i$ determine those of $t_i$ and the vertex translation; the three patterns repeat modulo three.}
\label{fig:twisted-cycle-family}
\end{figure}

\begin{Ex}[An arrow twist on the Kronecker quiver]
\label{ex:kronecker-arrow-twist}
Let $Q$ have arrows $a,b\colon1\to2$, and let $\theta$ fix the
vertices and exchange $a,b$. Set $\Gamma=\Gamma(Q,\theta,0)$ and
write $a^*,b^*$ for the reverse arrows
(Figure~\ref{fig:kronecker-arrow-twist}).
Here $H^0\Gamma=kQ$, and Corollary~\ref{thm:higher-mesh-d-stable}
gives $2$-stability and $\gl\Gamma=3$.

The product formula gives $a^\vee(b^*)^\vee=t_1^\vee$ and
$(b^*)^\vee a^\vee=0$. A graded cyclic functional in degree
three therefore vanishes on $t_1^\vee$. This makes the pairing
degenerate, since $t_1^\vee$ is supported at vertex $1$ and
pairs only with $E^0$.
For the ordinary zero-potential algebra the functional taking
value one on both $t_i^\vee$ is graded cyclic Frobenius.
Thus the two algebras have the same graded quiver and $H^0$
but distinct Yoneda algebras, ruling out a quasi-isomorphism.
\end{Ex}

\begin{figure}[htbp]
\centering
\begin{tikzpicture}[every node/.style={font=\small}]
\node[dgq vertex] (v1) at (0,0) {$1$};
\node[dgq vertex] (v2) at (3.6,0) {$2$};
\draw[dgq zero] (v1) to[bend left=14]
  node[above,dgq label] {$a$} (v2);
\draw[dgq zero] (v1) to[bend right=14]
  node[below,dgq label] {$b$} (v2);
\draw[dgq minus one] (v2) to[bend right=42]
  node[above,dgq label] {$a^*$} (v1);
\draw[dgq minus one] (v2) to[bend left=42]
  node[below,dgq label] {$b^*$} (v1);
\draw[dgq minus two] (v1) edge[loop left,min distance=9mm]
  node[left,dgq label] {$t_1$} (v1);
\draw[dgq minus two] (v2) edge[loop right,min distance=9mm]
  node[right,dgq label] {$t_2$} (v2);
\end{tikzpicture}

\smallskip
\quiverdegreelegend
\caption{The full graded quiver of the Kronecker example.
The formal twist exchanges $a,b$; the crossed pairing occurs in $\partial t_1$.}
\label{fig:kronecker-arrow-twist}
\end{figure}

\begin{Ex}[A coefficient twist over a quadratic Galois extension]
\label{ex:galois-twisted-ginzburg}
Let $l/k$ be a quadratic Galois extension with nontrivial
automorphism $\iota$. For $m\geq2$, take left $l$-bases
$z,a,t$ of degrees $0,-1,-2$ with right actions
\[
zc=\iota(c)z,\qquad
ac=\iota^m(c)a,\qquad
tc=\iota^{m+1}(c)t\qquad(c\in l).
\]
Define
\[
\partial z=0,\qquad \partial a=z^m,\qquad \partial t=za-az.
\]
The summands of each differential have the required coefficient
actions, and $\partial^2t=z^{m+1}-z^{m+1}=0$.
For the symmetrizing form $\lambda=\operatorname{Tr}_{l/k}$,
the two contractions have rank one and are invertible.
The trace pairing is nondegenerate in every characteristic
\cite[Tag~0BIL]{stacks-project}.
This defines a twisted Ginzburg dg algebra $\Gamma_m$ with
\[
\theta|_l=\iota^{m+1},\qquad \theta(z)=z.
\]
Here $H^0\Gamma_m\cong l[z;\iota]/(z^m)$ has dimension $2m$, so
$\Gamma_m$ is $2$-stable. The single vertex carries the coefficient
twist $\iota$ for even $m$ and the full identity twist for odd $m$.

The splitting $l\otimes_kl\cong l\times l$ identifies
$\Gamma_m\otimes_k l$ with the cycle model $\Gamma_{2,m}$ of
Example~\ref{ex:twisted-cycle-family} over $l$.
The coefficient twist becomes the vertex translation of the
split model; Figure~\ref{fig:galois-twisted-ginzburg} shows $m=2$. For even $m$ it cannot be removed by rescaling
the generators over $l$, since $l$ is commutative and $\iota$
is not an inner automorphism.
\end{Ex}

\begin{figure}[htbp]
\centering
\begin{tikzpicture}[every node/.style={font=\small}]
\node[dgq vertex] (D) at (0,0) {$l$};
\draw[dgq zero] (D) edge[loop above,min distance=9mm]
  node[above,dgq label] {$z$} (D);
\draw[dgq minus one] (D) edge[loop left,min distance=9mm]
  node[left,dgq label] {$a$} (D);
\draw[dgq minus two] (D) edge[loop below,min distance=9mm]
  node[below,dgq label] {$t$} (D);
\draw[-{Stealth[length=5pt]},double,line width=.4pt]
  (1.7,0)--node[above] {$l\otimes_k-$} (4,0);
\node[dgq vertex] (v0) at (5.5,0) {$0$};
\node[dgq vertex] (v1) at (8.7,0) {$1$};
\draw[dgq zero] (v0) to[bend left=13]
  node[above,dgq label] {$z_0$} (v1);
\draw[dgq zero] (v1) to[bend left=13]
  node[below,dgq label] {$z_1$} (v0);
\draw[dgq minus two] (v0) to[bend left=43]
  node[above,dgq label] {$t_1$} (v1);
\draw[dgq minus two] (v1) to[bend left=43]
  node[below,dgq label] {$t_0$} (v0);
\draw[dgq minus one] (v0) edge[loop left,min distance=8mm]
  node[left,dgq label] {$a_0$} (v0);
\draw[dgq minus one] (v1) edge[loop right,min distance=8mm]
  node[right,dgq label] {$a_1$} (v1);
\end{tikzpicture}

\smallskip
\quiverdegreelegend
\caption{The full graded species in the Galois example for $m=2$
and its split quiver after scalar extension. The coefficient
automorphism becomes the exchange of the two vertices.}
\label{fig:galois-twisted-ginzburg}
\end{figure}

\appendix
\section{Completed tensor dg algebras of graded species}
\label{app:dg-species}

Throughout this appendix, $k$ is a field. We develop the resolutions
and Yoneda-product calculations needed in
Section~\ref{sec:quadratic-recognition} over a finite-dimensional
semisimple $k$-algebra $R$. The resolutions and Yoneda-product
formulas require no separability over $k$. In
Subsection~\ref{subsec:direct-tensor-model}, we assume separability
to construct such a tensor model for a given dg algebra.
All derived categories and extension groups refer to ordinary dg
modules over the underlying dg algebra.

\subsection{Graded species and degreewise completion}

A \emph{graded species} consists of a finite product
$R=\prod_{v\in I}D_v$ of finite-dimensional division $k$-algebras
and a graded $R$-bimodule $V$. Its component $e_vVe_u$ is the
space of arrows from $u$ to $v$. We allow, more generally, any
finite-dimensional semisimple $R$ in the constructions below; the
split case $R=k^I$ recovers graded quivers.

Consider a completed tensor dg algebra
\[
A=(\widehat T_R(V),\partial),\qquad
A^j=\prod_{\ell\geq0}(V^{\otimes_R\ell})^j,
\qquad V^{\otimes_R0}=R,
\]
where $\partial$ is a continuous degree-one derivation, $\partial^2=0$, and
$\partial|_R=0$. The differential may have nonlinear terms on $V$.
We use conditions {\rm (P1)--(P4)} stated in
Section~\ref{sec:quadratic-recognition}.
In the first three subsections, basicness in {\rm (P1)} may be
omitted: the resolutions and multiplication formula use only
semisimplicity. Vertexwise notation is used when
$R=\prod_vD_v$. Here $A_{\geq r}$ is the closed subspace of tensor
series of length at least $r$. Projection to length zero is an augmentation
$\varepsilon\colon A\to R$, and we set $J=\ker\varepsilon=A_{\geq1}$.
Minimality {\rm (P4)} implies $\partial(A_{\geq r})\subseteq A_{\geq r+1}$
for $r\geq1$. We write
\[
\partial|_V=\sum_{r\geq2}\partial_r,\qquad \partial_r\colon V\longrightarrow V^{\otimes_R r},
\]
as a tensor-length-convergent sum of $R$-bimodule maps.

Multiplication gives isomorphisms of graded bimodules
\begin{equation}\label{eq:species-first-last-factor}
V\otimes_R A\xrightarrow{\sim}J
\xleftarrow{\sim}A\otimes_R V.
\end{equation}
In each degree only finitely many degrees of $V$ can occur as
the first or last factor. Their components are finite projective
on both sides over $R$, so tensoring commutes with the degreewise
products over tensor length. The same argument with $r$ factors
gives $J^r=A_{\geq r}$ for the ordinary ideal powers.
Denote the inverse of the first map in
\eqref{eq:species-first-last-factor} by $\ell\colon J\to V\otimes_R A$,
and put
\[
\delta(v):=\ell(\partial v)\in V\otimes_R J\qquad(v\in V).
\]

When $R=k^{Q_0}$ and $V=kQ_1$, these maps give the unique expansions
\[
p=\sum_{\alpha\in Q_1}\alpha(\alpha\backslash p)
 =\sum_{\beta\in Q_1}(p/\beta)\beta\qquad(p\in J).
\]
For homogeneous $p$, only finitely many arrows occur, although their
coefficients may be infinite series. In particular,
\[
\delta(\alpha)=\sum_\beta\beta\otimes(\beta\backslash \partial\alpha).
\]
This is the coefficient notation used in
Section~\ref{sec:quadratic-recognition}.

Completeness also identifies the semisimple quotient of $H^0A$.
Set $J_0=J\cap A^0$. For every $y\in J_0$, the geometric series
$\sum_{r\geq0}y^r$ converges to $(1-y)^{-1}$. The same applies to
all multiples of $y$, so $J_0\subseteq\rad A^0$.
Since $A^0/J_0=R$ is semisimple, $\rad A^0=J_0$.
Moreover, $\partial(A^{-1})\subseteq J_0^2$ by {\rm (P4)}. Hence
\begin{equation}\label{eq:species-semisimple-quotient}
H^0A/\rad H^0A\cong R.
\end{equation}
The inclusion $\partial(A^{-1})\subseteq J_0^2$ also gives
\begin{equation}\label{eq:species-first-radical-layer}
\rad H^0A/\rad^2 H^0A\cong J_0/J_0^2\cong V^0.
\end{equation}

\begin{Lem}[Comparison by tensor length]\label{lem:formal-tensor-inverse}
Let $C=\widehat T_R(U)$ and $D=\widehat T_R(W)$ for nonpositive graded
bimodules over a finite-dimensional semisimple algebra $R$. Let
$f\colon C\to D$ be a continuous graded algebra map preserving the
augmentation ideals and restricting to an automorphism of $R$. Write
$f_1\colon U\to W$ for its linear part.
\begin{enumerate}
\item If $f_1$ is bijective, then $f$ is a continuous graded algebra
isomorphism.
\item If $C,D$ have continuous differentials preserving the augmentation
ideals and $f$ is a dg map, then $f$ is a quasi-isomorphism whenever
$f_1$ is one for the induced linear differentials.
\end{enumerate}
\end{Lem}
\begin{proof}
On the length-$r$ graded pieces, $f$ induces $f_1^{\otimes_R r}$,
with coefficients twisted by $f|_R$. For {\rm (1)}, induction on $r$
gives isomorphisms $C/C_{\geq r}\to D/D_{\geq r}$, whose compatible
inverses pass to the limit.

For {\rm (2)}, the linear differentials induce the differentials on
the graded pieces. By semisimplicity the tensor powers of $f_1$ are
quasi-isomorphisms, so induction gives quasi-isomorphisms on the length
quotients. The towers are degreewise surjective, so their limits compute
homotopy limits and $f$ is a quasi-isomorphism.
\end{proof}

\subsection{Cofibrant resolutions and extension groups}

Give $V\otimes_R A$ the differential transported from $J$ by
multiplication. For homogeneous $v\in V$ and $a\in A$, it is
\begin{equation}\label{eq:mesh-E-differential}
\partial(v\otimes a)=\delta(v)a+(-1)^{|v|}v\otimes \partial a.
\end{equation}
Thus $V\otimes_R A$ is a dg $R$--$A$-bimodule. Write $sv$ for
the copy of $v\otimes1$ in $(V\otimes_R A)[1]$, of degree $|v|-1$.

\begin{Lem}[Cofibrant resolutions of semisimple modules]
\label{lem:species-cofibrant-resolution}
Let $A$ satisfy {\rm (P1)--(P4)}, and let $L$ be a finite-dimensional
right $R$-module, viewed as a right dg $A$-module by augmentation.
Multiplication gives a short exact sequence of dg $A$-modules
\begin{equation}\label{eq:mesh-simple-short-exact}
0\longrightarrow L\otimes_R V\otimes_R A
\xrightarrow{\mu_L}L\otimes_R A\longrightarrow L\longrightarrow0.
\end{equation}
The cone
\[
C_L:=\operatorname{Cone}(\mu_L)\longrightarrow L
\]
is a cofibrant resolution in the category of ordinary dg $A$-modules.
\end{Lem}
\begin{proof}
By \eqref{eq:species-first-last-factor}, the sequence is
$L\otimes_R(0\to J\to A\to R\to0)$ and is exact by semisimplicity.
Its cone therefore resolves $L$. For $L=R$, the cone differential is
\[
\partial_{C_R}(1)=0,\qquad \partial_{C_R}(sv)=v-s\delta(v),
\]
where $s(w\otimes a)=(sw)a$. Every first factor $w$ of $\delta(v)$
has $|w|\geq |v|+1$. Thus the filtration by descending generator
degree attaches finite projective right $A$-modules and gives
cofibrancy. Since $L_R$ is a summand of a finite free module,
$C_L=L\otimes_R C_R$ is cofibrant as well.
\end{proof}

We describe these resolutions as minimal one-sided twisted complexes.
We denote by $\operatorname{tw}^{-}(\proj_{\dg}A)$ the bounded-above
version of the one-sided twisted complexes in
Definition~\ref{def:homotopy-d-kernel}: each term is a finitely
generated projective dg module, and totalization uses direct sums.
Since $A$ is nonpositive, the matrices of homogeneous morphisms are
automatically column-finite. Thus their compositions and
sum-totalizations are well defined.

For $1\leq m<n$, let
\[
\delta_{m,n}\colon V^{1-n}\longrightarrow
V^{1-m}\otimes_R J^{m-n+1}
\]
be the component of $\delta(x)=\ell(\partial x)$ with first factor in
$V^{1-m}$, where $J^i=J\cap A^i$.
All other components vanish, since $V^{>0}=A^{>0}=0$.
These maps are $R$-bimodule maps. When forming the induced dg modules
below, each homogeneous component $V^j$ is regarded as an
$R$-bimodule concentrated in degree zero.

\begin{Prop}[Minimal projective resolutions of semisimple modules]
\label{prop:species-minimal-resolution}
Let $L\in\mod R$, viewed as a right dg $A$-module by augmentation.
Set
\begin{equation}\label{eq:species-twisted-terms}
X_L^{>0}=0,\qquad X_L^0=L\otimes_R A,\qquad
X_L^{-n}=L\otimes_R V^{1-n}\otimes_R A\quad(n\geq1).
\end{equation}
Its structure maps are
\begin{equation}\label{eq:species-twisted-maps}
\begin{aligned}
d_{X_L}^{-n,0}(\lambda\otimes x\otimes a)&=\lambda\otimes xa,\\
d_{X_L}^{-n,-m}(\lambda\otimes x\otimes a)
&=-(1_L\otimes\delta_{m,n})(\lambda\otimes x)a
\qquad(1\leq m<n),
\end{aligned}
\end{equation}
with all remaining components zero. Then
$X_L^\bullet\in\operatorname{tw}^{-}(\proj_{\dg}A)$, and its
sum-totalization is the cofibrant resolution $C_L$ of
Lemma~\ref{lem:species-cofibrant-resolution}.
\end{Prop}
\begin{proof}
The terms are induced from finite projective right $R$-modules.
The maps in \eqref{eq:species-twisted-maps} have degrees $1-n$ and
$m-n+1$, as required. In the sum-totalization they give the
multiplication to $L\otimes_R A$ and the components of $-s\delta$
on $(L\otimes_R V\otimes_R A)[1]$. This is the cone differential
of Lemma~\ref{lem:species-cofibrant-resolution}, so its square
gives the Maurer--Cartan equations.
\end{proof}

For the vertexwise formulas, take $R=\prod_{u\in I}D_u$,
$S_v=e_vR$, and $P_u=e_uA$, and write $X_v^\bullet=X_{S_v}^\bullet$.
Choose for every
$u,v$ and $j\leq0$ a
\emph{right $D_u$-basis} $\mathcal B_{v,u}^j$ of $e_vV^je_u$.
Write $s(\alpha)=u$, $t(\alpha)=v$, and $|\alpha|=j$ for
$\alpha\in\mathcal B_{v,u}^j$.
No basis as a bimodule is needed. These choices give
\begin{equation}\label{eq:species-vertex-projective-multiplicities}
X_v^{-n}\cong\bigoplus_{u\in I}
P_u^{\oplus\dim_{D_u}(e_vV^{1-n}e_u)}
=\bigoplus_{u\in I}\ \bigoplus_{\alpha\in\mathcal B_{v,u}^{1-n}}P_u.
\end{equation}
The first-factor isomorphism
\eqref{eq:species-first-last-factor} gives unique expansions
\[
p=\sum_{\beta\in\mathcal B_v}\beta(\beta\backslash p)
\qquad(p\in e_vJ),\qquad
\mathcal B_v:=\bigcup_{u,j}\mathcal B_{v,u}^j.
\]
Here $\beta\backslash p\in e_{s(\beta)}A$ and only finitely many
$\beta$ occur for homogeneous $p$. Thus, for
$\alpha\in\mathcal B_{v,u}^{1-n}$, the component to $P_v$ is
left multiplication by $\alpha$, and for
$\beta\in\mathcal B_{v,w}^{1-m}$ with $0<m<n$ the component to
the $\beta$-summand is left multiplication by
\[
-(\beta\backslash \partial\alpha)\in
 e_wJ^{m-n+1}e_u.
\]
In particular, if $x_v$ is the generator of $P_v$ and
$x_\alpha$ is the generator of its $\alpha$-summand in the
totalization, then $|x_\alpha|=|\alpha|-1$ and
\begin{equation}\label{eq:mesh-resolution-generator-differential}
\partial(x_v)=0,\qquad
\partial(x_\alpha)=x_v\alpha-
\sum_{\beta\in\mathcal B_v}x_\beta(\beta\backslash \partial\alpha).
\end{equation}
The coefficient of $\beta$ in $\partial^2\alpha=0$, with $|\beta|=1-m$,
is precisely the matrix Maurer--Cartan equation
\[
-\partial(\beta\backslash \partial\alpha)
+(-1)^m\sum_{\gamma\in\mathcal B_v}
(\beta\backslash \partial\gamma)(\gamma\backslash \partial\alpha)=0.
\]
All products here retain their displayed order, also for noncommutative
$D_u$. When all $D_u=k$ and the bases are the quiver arrows, these
are the quiver formulas. For $L=S_v$, the basis-free construction is
left $D_v$-linear and is independent of these choices of right bases.

\begin{Cor}\label{cor:mesh-ext-dimension}
For $E=\Ext_A^*(R,R)$, there are natural identifications
\[
E^{<0}=0,\qquad E^0=R,\qquad E^n=\Hom_R(V^{1-n},R)\quad(n\geq1).
\]
When $R=\prod_v D_v$, for $u,v\in I$ we have
\[
\Ext_A^0(S_v,S_u)\cong
\begin{cases}D_v&u=v,\\0&u\ne v,\end{cases}
\qquad \Ext_A^{<0}(S_v,S_u)=0.
\]
For every $n\geq1$, there is a natural identification
\begin{equation}\label{eq:species-vertex-ext}
\Ext_A^n(S_v,S_u)
\cong\Hom_{D_u}(e_vV^{1-n}e_u,D_u),
\end{equation}
where $\Hom_{D_u}$ means homomorphisms of right $D_u$-modules.
For $D_u=k$, this is $D(e_vV^{1-n}e_u)$.
\end{Cor}
\begin{proof}
Applying $\Hom_A(-,R)$ to $C_R$ gives zero differential, since
the cone's structure maps have coefficients in $J$. Its generators
are $1$ in degree zero and $sV^{1-n}$ in degree $-n$, giving the
formula for $E$. The same calculation with $C_{S_v}$ and target
$S_u$ gives the vertexwise formulas, since
$\Hom_R(e_vV^{1-n},S_u)=\Hom_{D_u}(e_vV^{1-n}e_u,D_u)$.
\end{proof}

\begin{Cor}[Finiteness and global dimension]
\label{cor:species-finiteness}
\leavevmode
\begin{enumerate}
\item We have $\pvd A=\thick R$, and this category is Hom-finite.
\item Suppose that $H^0A$ is finite dimensional. Then $A$ is locally
finite. Moreover, for every integer $n\geq1$,
\[
\gl A\leq n\quad\Longleftrightarrow\quad V^{\leq-n}=0.
\]
\end{enumerate}
\end{Cor}
\begin{proof}
(1) By \eqref{eq:species-semisimple-quotient}, the simple
$H^0A$-modules occur as summands of $R$. Finite composition series
and the Postnikov filtration give $\pvd A=\thick R$, and
Corollary~\ref{cor:mesh-ext-dimension} gives Hom-finiteness.

(2) With $H^0A$ finite dimensional, (1) implies local finiteness
by \cite[Proposition~3.6(1)]{Fus}. As in the proof of
Theorem~\ref{thm:aus-cor-for-loc-fin}, this gives $\gl A=\pd_A R$.
If $V^{\leq-n}=0$, all generators of $C_R$ have degrees at least
$-n$, so $\Hom_A(C_R,N)$ has no terms above $n$ for every
$N\in\Mod H^0A$. Conversely, if $V^q\ne0$ for some $q\leq-n$,
then $\Ext_A^{1-q}(R,R)\ne0$ by Corollary~\ref{cor:mesh-ext-dimension}.
\end{proof}

\subsection{Yoneda multiplication from the quadratic differential}

Use the identifications of Corollary~\ref{cor:mesh-ext-dimension}.
An element $r\in E^0=R$ acts by left multiplication on $R$, which
lifts to left multiplication on $C_R$. Thus post- and precomposition
give the $R$-bimodule structure
\begin{equation}\label{eq:species-dual-actions}
(rfs)(v)=r f(sv)\qquad(r,s\in R).
\end{equation}
A choice of symmetrizing form on $R$ identifies this right
$R$-dual with the $k$-dual as a bimodule. We order Yoneda products
as composition, so $gf$ means $g\circ f$ with the required shifts.

\begin{Prop}[Yoneda multiplication]\label{prop:yoneda-quadratic}
Let $p,q\geq1$, $f\in\Hom_R(V^{1-p},R)$, and
$g\in\Hom_R(V^{1-q},R)$. For $v\in V^{1-p-q}$, write
\[
(\partial_2v)_{1-p,1-q}=\sum_i a_i\otimes b_i
\in V^{1-p}\otimes_R V^{1-q}.
\]
Then
\begin{equation}\label{eq:species-yoneda-product}
(gf)(v)=(-1)^{p+1}\sum_i g\bigl(f(a_i)b_i\bigr).
\end{equation}
This contraction is well defined over $R$. In particular, the
ordinary graded algebra $E$ depends only on $R$, $V$, and $\partial_2$.
\end{Prop}
\begin{proof}
Write $C=C_R$, with augmentation $\pi\colon C\to R$ and inclusion
$\iota\colon A\to C$. The cone provides $h=s\ell\colon J\to C$
of degree $-1$ with
\begin{equation}\label{eq:yoneda-lift-homotopy}
\partial_{\Hom}h=\iota|_J.
\end{equation}
Extend $f$ by zero outside $V^{1-p}$ and lift its cochain to the
right $A$-linear map $\widehat f\colon C\to A$ with
$\widehat f(1)=0$ and $\widehat f(sv)=f(v)$.
Since $\varepsilon\widehat f$ is closed, $\partial_{\Hom}\widehat f$
takes values in $J$. Hence
\[
F_f=\iota\widehat f-h\,\partial_{\Hom}\widehat f
\]
is closed, with $\pi F_f=\varepsilon\widehat f$ representing $f$.
On $sv$ the cone differential gives
\[
F_f(sv)=f(v)+(-1)^{p+1}h\bigl((f\otimes1)\delta(v)\bigr).
\]
The cochain for $g$ kills the first term. On the second it factors
off the next tensor letter and augments the rest, yielding exactly
the quadratic terms in \eqref{eq:species-yoneda-product}.
The contraction is balanced because $f(ar)b=f(a)rb$.
Together with \eqref{eq:species-dual-actions}, it gives all products in $E$.
\end{proof}

\begin{Cor}[Yoneda multiplication for quivers]
\label{cor:quiver-yoneda-quadratic}
Let $A=\widehat{kQ}$ satisfy {\rm (P1)--(P4)}, with
$R=k^{Q_0}$ and $V=kQ_1$. Let $p,q\geq1$, and let
$\alpha\colon u\to v$ and $\beta\colon w\to u$ have degrees
$1-p$ and $1-q$.
Let $\alpha^\vee\in\Ext_A^p(S_v,S_u)$ and
$\beta^\vee\in\Ext_A^q(S_u,S_w)$ be the cochains dual to their
resolution generators. If $c^\gamma_{\alpha,\beta}$ is the coefficient
of $\alpha\beta$ in $\partial_2\gamma$, then
\begin{equation}\label{eq:yoneda-quadratic-sign}
\beta^\vee\alpha^\vee
=(-1)^{p+1}
\sum_{\substack{\gamma\colon w\to v\\|\gamma|=1-p-q}}
c^\gamma_{\alpha,\beta}\gamma^\vee.
\end{equation}
\end{Cor}
\begin{proof}
Apply Proposition~\ref{prop:yoneda-quadratic} to the cochains dual
to $\alpha$ and $\beta$. The contraction in
\eqref{eq:species-yoneda-product} extracts the coefficient of
$\alpha\beta$ in $\partial_2\gamma$, giving the asserted formula.
\end{proof}

\subsection{Existence of minimal completed tensor models}
\label{subsec:direct-tensor-model}

We lift the coefficient algebra, construct a tensor resolution in
descending degrees, and cancel contractible pairs after completion.

\begin{Lem}[Strictification of the coefficient algebra]
\label{lem:separable-coefficient-strictification}
Let $R$ be a finite-dimensional separable $k$-algebra, let $A$ be a
connective dg $k$-algebra, and let $s\colon R\to H^0A$ be an algebra
homomorphism. There is a nonpositive dg algebra $A'$ joined to $A$
by a zigzag of dg algebra quasi-isomorphisms and an algebra map
$R\to(A')^0$ inducing $s$ under the resulting identification
$H^0A'\cong H^0A$.
\end{Lem}
\begin{proof}
Put $B=R^{\op}\otimes_k A$ and give $H^0A$ its left $R$-action via $s$.
Tensoring a bimodule section of $R\otimes_k R\to R$ over $R$ with $H^0A$
exhibits $H^0A$ as a summand of the free $H^0B$-module $R\otimes_k H^0A$.
Via $H^0\colon H^0(\proj_{\dg}B)\xrightarrow{\sim}\proj H^0B$,
lift $H^0A$ to $P\in\proj_{\dg}B$. Since $B$ is finite free over $A$,
$P$ restricts to an object of $\proj_{\dg}A$, and $H^0P\cong H^0A$
lifts to a homotopy equivalence $f\colon P\to A$.

By \cite[Proposition~8.3(a)]{ELO09}, $\End_A(P)$ and $\End_A(A)\cong A$
are quasi-isomorphic, with $H^0$ identified by conjugation with $f$.
Thus $A'=\tau^{\leq0}\End_A(P)$ is a nonpositive model of $A$, and the
left $R$-action on $P$ gives $R\to(A')^0$ inducing $s$.
\end{proof}

\begin{Lem}[Completion of a relative tensor dg algebra]
\label{lem:direct-tensor-completion}
Let $R$ be a finite-dimensional semisimple $k$-algebra and
$Q=(T_R(U),\partial)$ an augmented dg $R$-algebra with augmentation ideal
$J=(U)$. Assume $U$ is nonpositive and degreewise finite dimensional,
and $H^0Q$ is finite dimensional with semisimple quotient $R$ induced
by the augmentation. Then $Q$ is locally finite, and the degreewise
completion $Q\to\widehat Q=(\widehat T_R(U),\partial)$ is a quasi-isomorphism.
\end{Lem}
\begin{proof}
The graded factorization $J\cong U\otimes_R Q$ gives the first-letter
filtration from the proof of Lemma~\ref{lem:species-cofibrant-resolution},
so $P_Q:=\operatorname{Cone}(J\hookrightarrow Q)\to R$ is a cofibrant
right resolution. For a left dg $Q$-module $M$ with $H^{>i}M=0$, adjunction
against $R$-modules in degree $i$ gives
\[
H^i(R\otimes_Q^{\mathbb L}M)\cong R\otimes_{H^0Q}H^iM
\cong H^iM/\mathfrak r H^iM,
\qquad \mathfrak r=\rad H^0Q.
\]
Since $\mathfrak r$ is nilpotent, lifts of a basis of the quotient
generate $H^iM$ over the finite-dimensional algebra $H^0Q$.
Thus the formula detects vanishing and finite dimensionality.

For the completion, \eqref{eq:species-first-last-factor} gives
\[
P_Q\otimes_Q\widehat Q
\cong\operatorname{Cone}(\widehat J\hookrightarrow\widehat Q)
\xrightarrow{\sim} R,\qquad
\widehat J=\ker(\widehat Q\to R).
\]
Under the augmentations, the reduction of $Q\to\widehat Q$ is $1_R$.
Its cone is bounded above with acyclic reduction, so the formula
rules out a highest degree of nonzero cohomology.

For local finiteness, descend on $i$. Once $H^{>i}Q$ is degreewise finite,
$\tau^{>i}Q$ belongs to $\thick R$ as a left module. Both
$P_Q\otimes_Q R\cong R\oplus(U,\partial_1)[1]$ and $P_Q\otimes_Q Q\simeq R$
are locally finite, where $\partial_1$ is the linear differential. The truncation
triangle therefore gives locally finite reduction of $\tau^{\leq i}Q$,
so the formula gives $\dim_k H^iQ<\infty$.
\end{proof}

\begin{Prop}[Direct construction of a minimal tensor model]
\label{prop:direct-minimal-tensor-model}
Let $A$ be a locally finite connective dg $k$-algebra such that
$R=H^0A/\rad H^0A$ is basic and separable over $k$. Then $A$ is
joined by a zigzag of quasi-isomorphisms of ordinary dg algebras
to $(\widehat T_R(V),\partial)$ satisfying {\rm (P1)--(P4)}.
\end{Prop}
\begin{proof}
The Wedderburn--Malcev theorem gives an algebra section of $H^0A\to R$.
By Lemma~\ref{lem:separable-coefficient-strictification}, we may replace
$A$ by a nonpositive model with a lift $R\to A^0$ of this section.
The composite $A\to H^0A\to R$ is then an $R$-algebra augmentation.
Since $R^e$ is semisimple, cycle lifts and primitives can be chosen
$R$-bilinearly.

Cycle lifts of $H^*(\ker(A\to R))$ give a cohomologically surjective map
from its tensor algebra with zero differential to $A$. Adjoin
$\mathfrak r\otimes_R\mathfrak r$ in degree $-1$ with differential
$x\otimes y-\mu(x,y)$, where $\mathfrak r=\rad H^0A$ and $\mu$ is
its multiplication, and send these generators to primitives of the
images of the relations. The resulting $f\colon Q\to A$ induces
$H^0Q\cong H^0A$ by the multiplication relations.

For $i=-1,-2,\ldots$, adjoin $\ker H^if$, finite dimensional by
Lemma~\ref{lem:direct-tensor-completion}, in degree $i-1$.
For a cycle lift $z$, set $\partial x=z$ and send the new generator $x$ to
a primitive of $f(z)$. The new words have degree $<i$, so this
makes $H^if$ an isomorphism. The union gives an augmented
quasi-isomorphism $Q=(T_R(U),\partial)\to A$ with $U$ degreewise finite.

By semisimplicity of $R^e$, choose graded $R$-bimodule complements
$V$ to $\operatorname{im}\partial_1$ in $\ker\partial_1$ and $L$ to $\ker\partial_1$ in $U$.
Then $V\cong H^*(U,\partial_1)$, and $\partial_1|_L$ is a degree-one bijection onto
$\operatorname{im}\partial_1=\partial_1L$. Thus
\[
(U,\partial_1)=(V,0)\oplus(L\oplus\partial_1L,\partial_1),
\]
where the second summand is contractible.

The substitution $\partial_1\ell\mapsto\partial\ell$ fixing $V\oplus L$ in
$\widehat Q$ has identity linear part and is a continuous graded
algebra automorphism by Lemma~\ref{lem:formal-tensor-inverse}(1).
Hence the dg quotient
\[
\Gamma:=\widehat Q/\overline{(L,\partial L)}\cong\widehat T_R(V)
\]
is minimal since $\partial_1V=0$. By part~(2) of the lemma, $\widehat Q\to\Gamma$ is
a quasi-isomorphism since its linear part is the projection onto $(V,0)$.
Together with Lemma~\ref{lem:direct-tensor-completion} this gives
$A\xleftarrow{\sim}Q\xrightarrow{\sim}\widehat Q\xrightarrow{\sim}\Gamma$.
\end{proof}

\bibliographystyle{alpha} 
\bibliography{reference}

\end{document}